\documentclass{statsoc}
\usepackage[letterpaper]{geometry}
\makeatletter
\renewcommand*{\@opargbegintheorem}[3]{\trivlist
	\item[\hskip \labelsep{\bfseries #1\ #2}] \textbf{(#3)}\ \itshape}
\makeatother

\usepackage{etoolbox}

\makeatletter
\patchcmd{\@makecaption}
  {\parbox}
  {\advance\@tempdima-\fontdimen2} 
  {}{}
\makeatother

\usepackage{mathrsfs}
\usepackage{subfiles}

\usepackage{amsmath,amssymb,amsfonts,amsbsy,latexsym,dsfont,tikz}

\usepackage{enumerate}

\usepackage{subfigure}
\usepackage{graphicx}
\usepackage{graphics}
\usepackage{multirow}
\usepackage{amsmath}
\usepackage{hyperref}
\usepackage{url}
\usepackage{enumitem}
\usepackage{algorithm}
\usepackage{algorithmic}
\usepackage{natbib}
\setcitestyle{comma}
\newcommand{\va}[0]{\mathbf a}

\newcommand{\ve}[0]{\mathbf e}

\newcommand{\vX}[0]{\mathbf X}

\newcommand{\vZ}[0]{\mathbf Z}

\newcommand{\vzero}[0]{\mathbf 0}
\newcommand{\vone}[0]{\mathbf 1}
\newcommand{\mbf}[1]{\mbox{\boldmath$#1$}}
\newcommand{\Cov}[0]{\text{Cov}}
\newcommand{\Var}[0]{\text{Var}}

\newcommand{\calA}[0]{\mathcal{A}}
\newcommand{\calB}[0]{\mathcal{B}}
\newcommand{\calD}[0]{\mathcal{D}}

\newcommand{\calS}[0]{\mathcal{S}}

\newcommand{\E}[0]{\mathds{E}}

\newcommand{\N}[0]{\mathbb{N}}
\newcommand{\R}[0]{\mathbb{R}}

\newcommand{\Id}[0]{\text{Id}}

\newcommand{\sign}[0]{\text{sign}}

\newcommand{\argmax}[0]{\text{argmax}}
\newcommand{\ub}[0]{\underline{b}}
\newcommand{\us}[0]{\underline{s}}

\newcommand{\udelta}[0]{\underline{\delta}}

\newcommand*{\medcup}{\mathbin{\scalebox{1.5}{\ensuremath{\cup}}}}

\newcommand{\Prob}[0]{\mathds{P}}

\renewcommand{\vec}[0]{\text{vec}}

\renewcommand{\le}{\leqslant}
\renewcommand{\ge}{\geqslant}

\newtheorem{thm}{Theorem}[section]
\newtheorem{lem}[thm]{Lemma}

\newtheorem{cor}[thm]{Corollary}

\newtheorem{rmk}{Remark}

\title[Change point inference and identification]{Finite sample change point inference and identification for high-dimensional mean vectors}
\author[M. Yu and X. Chen]{Mengjia Yu and Xiaohui Chen}
\address{Department of Statistics, University of Illinois at Urbana-Champaign,
	725 S. Wright Street, Champaign, IL 61820,
	USA.}
\email{myu17@illinois.edu}

\begin{document}
	
	\begin{abstract}
		Cumulative sum (CUSUM) statistics are widely used in the change point inference and identification. 
		For the problem of testing for existence of a change point in an independent sample generated from the mean-shift model, we introduce a Gaussian multiplier bootstrap to calibrate critical values of the CUSUM test statistics in high dimensions. The proposed bootstrap CUSUM test is fully data-dependent and it has strong theoretical guarantees under arbitrary dependence structures and mild moment conditions. Specifically, we show that with a boundary removal parameter the bootstrap CUSUM test enjoys the uniform validity in size under the null and it achieves the minimax separation rate under the sparse alternatives when the dimension $p$ can be larger than the sample size $n$. 
		
		Once a change point is detected, we estimate the change point location by maximizing the $\ell^{\infty}$-norm of the generalized CUSUM statistics at two different weighting scales corresponding to covariance stationary and non-stationary CUSUM statistics. For both estimators, we derive their rates of convergence and show that dimension impacts the rates only through logarithmic factors, which implies that consistency of the CUSUM estimators is possible when $p$ is much larger than $n$.
		In the presence of multiple change points, we propose a principled bootstrap-assisted binary segmentation (BABS) algorithm to dynamically adjust the change point detection rule and recursively estimate their locations. We derive its rate of convergence under suitable signal separation and strength conditions.
		
		The results derived in this paper are non-asymptotic and we provide extensive simulation studies to assess the finite sample performance. The empirical evidence shows an encouraging agreement with our theoretical results.
	\end{abstract}

	\section{Introduction}
	\label{sec:introduction}
	
	This paper studies the problems of change point inference and identification for mean vectors of high-dimensional data in finite samples. High-dimensional data are now ubiquitous in many scientific and engineering fields and data heterogeneity is the rule rather than the exception. A central problem of studying data heterogeneity is to detect structural breaks in the underlying data generation process. Perhaps the two most fundamental questions for abrupt changes are: i) is there a change point in data? ii) if so, when does the change occur? In this work, we consider change point detection and identification for temporally independent data with cross-sectional dependence. Specifically, let $X_1^n = \{ X_1,\dots,X_n \}$ be a sequence of independent random vectors in $\R^p$ generated from the mean-shift model:
	\begin{equation}
	\label{eqn:mean_shifting_model}
	X_i := (X_{i1}, \ldots, X_{ip} )^\top = \mu + \delta_n \vone(i > m) + \xi_i, \qquad i = 1,\dots,n,
	\end{equation}
	where $\mu \in \R^p$ is the population mean parameter, $\delta_n \in \R^p$ is the mean-shift signal parameter, $m$ is the change point location, and $\xi_1,\dots,\xi_n$ are (temporally) independent and identically distributed (i.i.d.) mean-zero noise random vectors in $\R^p$ with common distribution function $F$. Let $\Sigma = \Cov(\xi_1)$ be the unknown noise covariance matrix that is not necessarily diagonal, and thus we allow cross-sectional (sometimes also referred as spatial) dependence among the components $X_{i1},\dots,X_{ip}$ for each $i=1,\dots,n$. Under the mean-shift model, if $\delta_n = 0$ or $m = n$, then $X_1,\dots,X_n$ form a sample of i.i.d.\ random vectors and no change point occurs. In this paper, our first goal is to test for whether or not there is a change point in the mean vectors $\mu_i = \E(X_i)$, i.e., to test for 
	\begin{equation}
	\label{eqn:change_point_mean_test}
	H_0: \delta_n = 0 \quad \text{and} \quad H_1 : \delta_n \neq 0 \text{ and there exists an } m \in \{1,\dots,n-1\},
	\end{equation}
	where the alternative hypothesis $H_1$ is parameterized by the change point signal $\delta_n$ and location $m$. If a change point is detected in the mean vectors (i.e., $H_0$ is rejected), then our second goal is to estimate the change point location $m$. 
	
	For i.i.d.\ Gaussian noise $\xi_i \sim N(0, \Sigma)$, the log-ratio of the maximized likelihoods between $H_{1}$ with a change point at $s = 1,\dots,n-1$ and $H_{0}$ without change point is given by 
	\begin{equation}
	\label{eqn:max_likelihood_ratio}
	\log(\Lambda_s) = Z_n(s)^\top \Sigma^{-1} Z_n(s) / 2,
	\end{equation}
	where
	\begin{equation}
	\label{eqn:cusum_mean_Rp}
	\qquad Z_n(s) = \sqrt{s(n-s) \over n} \Big({1\over s} \sum_{i=1}^s X_i - {1 \over n-s} \sum_{i=s+1}^n X_i \Big)
	\end{equation}
	is a sequence of the normalized mean differences before and after $s$. Then $H_0$ is rejected if $\max_{1 \le s < n} \log(\Lambda_s)$ is larger than a critical value. In literature, $\{Z_n(s)\}_{s=1}^{n-1}$ are often called the {\it cumulative sum} (CUSUM) statistics \citep{csorgohorvath1997}. Note that the log-ratio statistics of the maximized likelihoods in (\ref{eqn:max_likelihood_ratio}) require the knowledge or an estimate of the unknown covariance matrix $\Sigma$. In the high-dimensional setting where $p$ is larger (or even much larger) than $n$, estimation of $\Sigma$ itself becomes a challenging problem. Spectral norm consistency of $\Sigma$ (or the inverse $\Sigma^{-1}$) is possible under additional structural assumptions (such as sparsity or low-rankness) on the covariance matrix \citep{bickellevina2008a,bickellevina2008b,caizhangzhou2010a,caizhou2011a,chenxuwu2013a,BarigozziChoFryzlewicz2017}, which may not hold in practical applications. In contrast, tests based on the CUSUM statistics in (\ref{eqn:cusum_mean_Rp}) do not involve $\Sigma$ and they are more robust to the misspecification on covariance structures. Therefore, this motivates us to study the problems of change point testing and estimation based on the high-dimensional CUSUM statistics.
	
	To build a decision rule for change-point detection, we need to cautiously aggregate the (dependent) random vectors $Z_n(s), s=1,\dots,n-1$. \cite{enikeevaharchaoui2014} considers the change point detection on mean vectors under the mean-shift model (\ref{eqn:mean_shifting_model}) with i.i.d.\ $\xi_i \sim N(0, \sigma^2 I_p)$. They propose the linear and scan statistics based on the $\ell^2$-norm aggregation of the CUSUM statistics and derive the change point detection boundary. \cite{jirak2015} considers the $\ell^\infty$-norm aggregation of the CUSUM statistics and establishes a Gumbel limiting distribution under $H_0$. \cite{jirak2015} also considers the bootstrap approximations to improve the rate of convergence. \cite{wangsamworth2017} considers the estimation in the high-dimensional mean vectors in reduced dimensions by sparse projections and they derive the rate of convergence for estimating the change point location. In the three aforementioned papers, strong structural assumptions (i.e., cross-sectional sparsity in the sense that the components $\{X_{ij}\}_{j=1}^p$ of $X_{i}$ are independent or weakly dependent) are imposed to substantially reduce the intrinsic complexity.  \cite{chofryzlewicz2015} relaxes the sparsity assumption and consider the estimation problem in the (marginal) variances of high-dimensional time series under a multiplicative model. They propose a {\it sparsified binary segmentation} (SBS) performing the $\ell^1$-norm aggregation on a thresholded version of the CUSUM statistics where an additional sparsifying step with a tuning parameter is used to avoid noise accumulation in the aggregation. Since the SBS is sensitive to threshold tuning parameters, \cite{cho2016change} proposes a double CUSUM test procedure sorting the magnitudes of the $p$ components of CUSUM statistics and it may be viewed as a data-driven alternative for selecting the threshold in~\cite{chofryzlewicz2015}.
	
	In this paper, besides some mild moment conditions, we do not make any assumption on the cross-sectional dependence structure of the underlying data distribution. We consider the multivariate CUSUM statistics (\ref{eqn:cusum_mean_Rp}) in the $\ell^\infty$-norm aggregated form: 
	\begin{equation}
	\label{eqn:cusum_mean_L_inf}
	T_n = \max_{\us \le s \le n-\us} |Z_n(s)|_\infty := \max_{\us \le s \le n-\us} \max_{1 \le j \le p} |Z_{nj}(s)|,
	\end{equation}
	where $\us \in [1, n/2]$ is a user-specified {\it boundary removal} parameter. Removing boundary points is necessary in change point detection since the distributions of $|Z_n(s)|_\infty$ that are closer to endpoints are more difficult to approximate because of fewer data points. Then $H_0$ is rejected if $T_n$ is larger than a critical value such as the $(1-\alpha)$ quantile of $T_n$. Under $H_0$, $\{Z_n(s)\}_{s=1}^{n-1}$ is a centered and covariance stationary process in $\R^p$ (i.e., $\E[Z_n(s)] = 0$ and $\Cov(Z_n(s)) = \Sigma$). To approximate the distribution of $T_n$, extreme value theory is a common technique to derive the Gumbel-type limiting distributions \citep{LeadbetterLindgrenRootzen_1983,resnick1987a}. However, even for $p=1$, the convergence rate of maxima of the CUSUM process $\{Z_n(s)\}_{s=1}^{n-1}$ is known to be slow \citep{resnick1987a,hall1991}.
	
	\subsection{Our contributions} 	
	To overcome the fundamental difficulty in calibrating the distribution of $T_{n}$, we consider the bootstrap approximation to the {\it finite sample distribution} of $T_n$ without referring to a weak limit of $\{Z_n(s)\}_{s=1}^{n-1}$. In Section \ref{sec:model}, we propose a Gaussian multiplier bootstrap tailored to the CUSUM test statistics in (\ref{eqn:cusum_mean_Rp}). The proposed test is fully data-dependent and requires no tuning parameter (except for a pre-specified boundary removal parameter $\us$). This is in contrast with the thresholding-aggregation method in \cite{chofryzlewicz2015}, which requires further procedure to select threshold that is not easy to justify. We will show in Section \ref{subsec:validity_bootstrap_cp_test} that the proposed test is a uniformly valid inferential procedure under $H_0$ where $p$ can grow sub-exponentially fast in $n$ {\it and} no explicit condition on the dependence structure among $\{X_{ij}\}_{j=1}^p$ is needed. This is in contrast with \cite{enikeevaharchaoui2014,jirak2015} and \cite{wangsamworth2017} where the components are assumed to be either independent or weakly dependent, and with \cite{chofryzlewicz2015} and \cite{cho2016change} where the dimension can only grow polynomially fast in sample size. Moreover, we will show that, under a mild signal strength condition, our bootstrap CUSUM test is consistent in the sense that the sum of type I and type II errors is asymptotically vanishing \citep[Chapter 6.2]{ginenickl2016}. In addition, the requirement on the signal strength can achieve the minimax separation rate derived in \cite{enikeevaharchaoui2014} under the sparse alternative (i.e., the change occurs only in a few number of components).
	
	If a change point is detected, then we estimate its location by maximizing the $\ell^{\infty}$-norm of the generalized CUSUM statistics (\ref{eqn:generalized_cusum}) at two different weighting scales. The first estimator is based on the covariance stationary CUSUM statistics in (\ref{eqn:cusum_mean_Rp}). In Section \ref{subsec:rate_convergence_cp_location_estimator}, we show that it is consistent in estimating the location at the parametric rate $n^{-1/2}$ (up to logarithmic factors) for sub-exponential observations. The second estimator is a non-stationary CUSUM statistics assigning less weights on the boundary data points. In this case, we show that it achieves the best possible rate of convergence on the order $n^{-1}$ (up to logarithmic factors) under stronger side conditions. In both cases, dimension impacts the rate of convergence only through the logarithmic factors. Thus consistency of the CUSUM location estimators can be achieved when $p$ grows sub-exponentially fast in $n$. 
	
	Our bootstrap change point inference can be naturally extended to handle multiple change points via the generic binary segmentation (BS) technique. Once a change point is claimed by our bootstrap test and located by our estimator, BS continues the same testing and estimation procedure on the segments before and after the change until no further change point can be detected by the bootstrap test (cf.\ Algorithm \ref{alg:multi_cp} in Section \ref{subsec:methodology_mcp}). Thus, the bootstrap CUSUM test can dynamically adjust detection rule during the iterations. We derive the rate of convergence of this {\it bootstrap-assisted binary segmentation} (BABS) for recursively estimating the multiple change points under suitable signal separation and strength conditions.
	
	\subsection{Literature review}
	
	CUSUM statistics \citep{page1955} are originally introduced in the sequential testing problems to distinguish between the in-control hypothesis $\delta_n = 0$ and the out-control mean-shift hypothesis for a {\it given} $\delta_n \neq 0$ in model (\ref{eqn:mean_shifting_model}), aiming to minimize the expected average run length \citep{Lai2001_SS,Walk1945_AoMS,WaldWolfowitz1948_AoMS,page1955,chernoffzacks1964,hinkley1970,WoodallNcube2985_Technometrics,WaldWolfowitz1948_AoMS,Lorden1971_AoMS,qiu2013_SPC}. The current paper uses CUSUM statistics for fixed sampled size tests, as in many other statistical change point testing and estimation works \citep{yao1987_AoS,BrodskyDarkhovsky1993_Springer,csorgohorvath1997,kokoszkaleipus2000,bghk2009,fryzlewicz2014,carlstein1988,loader1996,HarizWylieZhang2007_AoS,harchaoui2010,frick2014, garreau2016,ombaovonsachsguo2005,auehormannhorvathreimher2009,bai1997,bhattacharya1987_JMA,zhang2010, harle2016}.
	There is a recent surge of literature on change point analysis for high-dimensional data: change point detection is considered in~\citep{enikeevaharchaoui2014,jirak2015}, estimation of the number and locations of change points is considered in~\citep{chofryzlewicz2015,jirak2015,wangsamworth2017,cho2016change}, and bootstrap inference is considered in~\citep{cho2016change} (without giving rigorous statistical guarantees).

	Finite sample approximations to the distribution of maxima for sums of {\it independent} mean-zero random vectors in high dimensions are studied in \cite{cck2013,cck2016a}. We highlight that validity of our bootstrap CUSUM test for the change point does not (at least directly) follow the Gaussian and bootstrap approximation results in \cite{cck2013,cck2016a}. The reason is that, in the change-point detection context, the extreme-value type test statistic $T_n$ defined in (\ref{eqn:cusum_mean_L_inf}) is the maximum of a sequence of {\it dependent} random vectors $Z_n(s), s = \us, \dots, n-\us$. Therefore, the distributional approximation results developed in \cite{cck2013,cck2016a} require considerable modifications tailored to the change point analysis. A main technical innovation of this work is that the CUSUM statistics are affine transformations of the independent data points in an augmented space so that we can make use of the high-dimensional Gaussian and bootstrap approximations without overpaying the price of the increased dimensionality in the embedded larger space.

	\subsection{Organization}
	
	The rest of this paper is organized as follows. The bootstrap change point test, change point location estimators and extension to multiple change points algorithm are described in Section \ref{sec:model}. In Section~\ref{sec:main_results}, we derive the size validity and power properties of the bootstrap test, the rate of convergence for two change point location estimators and the consistency of BABS. Section~\ref{sec:simulation} and \ref{sec:real_app} provide extensive simulation results and two real data examples, respectively. Discussions on detailed comparisons with literature, proofs of the main results, and additional simulation results are given in the Supplementary Material (SM). 
	
	\subsection{Notation}
	
	For $q > 0$ and a generic vector $x \in \R^p$, we denote $|x|_q = (\sum_{i=1}^p |x_i|^q)^{1/q}$ for the $\ell^q$ norm of $x$ and we write $|x| = |x|_2$. For a random variable $X$, denote $\|X\|_q = (\E|X|^q)^{1/q}$. For $\beta > 0$, let $\psi_\beta(x) = \exp(x^\beta) - 1$ be a function defined on $[0,\infty)$ and $L_{\psi_\beta}$ be the collection of all real-valued random variables $X$ such that $\E[\psi_\beta(|X| / C)] < \infty$ for some $C > 0$. For $X \in L_{\psi_\beta}$, define $\|X\|_{\psi_\beta} = \inf\{C > 0 : \E[\psi_\beta(|X| / C)]  \le 1 \}$. Then, for $\beta \in [1,\infty)$, $\|\cdot\|_{\psi_\beta}$ is an Orlicz norm and $(L_{\psi_\beta}, \|\cdot\|_{\psi_\beta})$ is a Banach space \citep{ledouxtalagrand1991}. For $\beta \in (0, 1)$, $\|\cdot\|_{\psi_\beta}$ is a quasi-norm, i.e., there exists a constant $C(\beta) > 0$ such that $\|X+Y\|_{\psi_\beta} \le C(\beta) (\|X\|_{\psi_\beta} + \|Y\|_{\psi_\beta})$ holds for all $X, Y \in L_{\psi_\beta}$ \citep{adamczak2008}. Let $\rho(X, Y) = \sup_{t \in \R} |\Prob(X \le t) - \Prob(Y \le t)|$ be the Kolmogorov distance between two random variables $X$ and $Y$. We shall use $C_1,C_2,\dots$ and $K_1,K_2,\dots$ to denote positive and finite constants that may have different values. We assume $n \ge 4$ and $p \ge 3$.

	\section{Methodology}
	\label{sec:model}
	
	\subsection{Bootstrap CUSUM test}
	\label{subsec:bootstrap_cp_test}
	
	We first introduce a bootstrap procedure to approximate the distribution of $T_n$. Let $e_1,\dots,e_n$ be i.i.d.\ $N(0,1)$ random variables independent of $X_1^n$. Let $\bar{X}_s^- = s^{-1} \sum_{i=1}^s X_i$ and $\bar{X}_s^+ = (n-s)^{-1} \sum_{i=s+1}^n X_i$ be the left and right sample averages at $s$. Define
	\begin{equation}
	\label{eqn:bootstrap_cusum}
	Z_n^*(s) = \sqrt{n-s \over ns} \sum_{i=1}^s e_i (X_i - \bar{X}_s^-) - \sqrt{s \over n(n-s)} \sum_{i=s+1}^n e_i (X_i - \bar{X}_s^+).
	\end{equation}
	Then the bootstrap test statistic is defined as
	\begin{equation}
	\label{eqn:bootstrap_cusum_mean_L_inf}
	T_n^* = \max_{\us \le s \le n-\us} |Z_n^*(s)|_\infty,
	\end{equation}
	and the $(1-\alpha)$ conditional quantile of $T_n^*$ given $X_1^n$, defined as $q_{T_{n}^{*} \mid X_{1}^{n}}(1-\alpha) = \inf\{ t \in \R : \Prob(T_n^* \le t | X_1^n) \ge 1-\alpha \}$, is used as critical values of the bootstrap test to approximate the quantiles of $T_n$. In particular, for any $\alpha \in (0,1)$, we reject $H_0$ if $T_n > q_{T_{n}^{*} \mid X_{1}^{n}}(1-\alpha)$. Note that the Gaussian multiplier bootstrap test statistic $T_n^*$ and its conditional quantile $q_{T_{n}^{*} \mid X_{1}^{n}}(1-\alpha)$ are {\it computable} since we can draw Monte Carlo samples by simulating the multiplier random variables $e_1,\dots,e_n$ to approximate the distribution of $T_{n}^{*}$. 
		
	\begin{rmk}[Comments on centering terms in the bootstrap CUSUM statistics]
		We may also consider the following version of bootstrap CUSUM statistics
		\[
		\tilde{Z}_n^*(s) = \sqrt{n-s \over ns} \sum_{i=1}^s e_i X_i - \sqrt{s \over n(n-s)} \sum_{i=s+1}^n e_i X_i
		\]
		without left and right centering by $\bar{X}_s^-$ and $\bar{X}_s^+$. It can be shown that the bootstrap CUSUM tests based on $Z_n^*(s)$ and $\tilde{Z}_n^*(s)$ have the same rate of convergence in the size and power analysis (Theorem~\ref{thm:bootstrap_approx_cusum_statistic}, Corollary~\ref{cor:validitiy_bootstrap_cusum_test}, and Theorem~\ref{thm:power_bootstrap_cusum_test}).
		Since $\Cov(Z_n^*(s) | X_{1}^{n}) \le \Cov(\tilde{Z}_n^*(s) | X_{1}^{n})$ as a matrix inequality,
		$\tilde{Z}_n^*(s)$ incurs a larger (conditional) covariance matrix than $Z_n^*(s)$ and it is recommended to use $Z_n^*(s)$ rather than $\tilde{Z}_n^*(s)$.
		\qed
	\end{rmk}
	
	\begin{rmk}[Generalization to time series: a block multiplier bootstrap CUSUM test]
	Since the CUSUM test statistics $Z_{n}(s)$ in (\ref{eqn:cusum_mean_Rp}) can be re-written as a block sum, the Gaussian multiplier bootstrap CUSUM test statistics in \eqref{eqn:bootstrap_cusum} and \eqref{eqn:bootstrap_cusum_mean_L_inf} can be modified to a block version to accommodate the temporal dependence for time series data. Let $M,B$ be positive integers such that $n = M B$. We divide the sample $X_{1}^{n}$ into $B$ blocks of size $M$. In particular, for $b=1,\dots,B$, let $L_{b} = \{(b-1)M+1,\dots,bM\}$ be the $b$-th block indices. Then, for $s=1,\dots,n-1$, we can write $Z_{n}(s)$ in (\ref{eqn:cusum_mean_Rp}) as 
	\[
	Z_{n}(s) = \sqrt{n-s \over ns} \sum_{b=1}^{B} \sum_{i \in L_{b}} X_{i} \vone(i \le s) - \sqrt{s \over n(n-s)} \sum_{b=1}^{B} \sum_{i \in L_{b}} X_{i} \vone(i > s).
	\]
	For any $\alpha \in (0,1)$, we reject $H_{0}$ if the test statistic $T_{n} = \max_{\us \le s \le n-\us} |Z_{n}(s)|_{\infty}$ is larger than a critical value. Since the distributions of $Z_{n}(s)$ under $H_{0}$ and $H_{1}$ for dependent error processes are different from the i.i.d.\ errors, we need to accommodate the dependence in calibrating the distributions of the test statistic $T_{n}$. The idea is to modify the Gaussian multiplier bootstrap $Z_{n}^{*}(s)$ in (\ref{eqn:bootstrap_cusum}) and the bootstrap CUSUM test statistic $T_{n}^{*}$ in (\ref{eqn:bootstrap_cusum_mean_L_inf}) to their block versions. Specifically, to approximate the (finite sample) distribution of $T_{n}$, we use a {\it block Gaussian multiplier bootstrap} tailored to the CUSUM statistics setting. Let $e_{1},\dots,e_{B}$ be i.i.d.\ standard Gaussian random variables. Define
	\[
	Z_{n}^{\sharp}(s) = \sqrt{n-s \over n s} \sum_{b=1}^{B} e_{b} V_{b}^{-}(s) - \sqrt{s \over n (n-s)} \sum_{b=1}^{B} e_{b} V_{b}^{+}(s),
	\]
	where $V_{b}^{-}(s) = \sum_{i \in L_{b}} (X_{i}-\bar{X}_{s}^{-}) \vone(i \le s) \text{ and } V_{b}^{+}(s) = \sum_{i \in L_{b}} (X_{i}-\bar{X}_{s}^{+}) \vone(i > s).$
	Then the distribution of $T_{n}$ is approximated by its bootstrap analog given by $T_{n}^{\sharp} = \max_{\us \le s \le n-\us} |Z_{n}^{\sharp}(s)|_{\infty}$ and we reject $H_{0}$ if $T_{n} > q_{T_{n}^{\sharp} \mid X_{1}^{n}}(1-\alpha)$, where $q_{T_{n}^{\sharp} \mid X_{1}^{n}}(1-\alpha)$ is the $(1-\alpha)$ conditional quantile of $T_{n}^{\sharp}$ given $X_{1}^{n}$. Note that if the block size $M = 1$ (i.e., $B = n$), then $Z_{n}^{\sharp}(s) = Z_{n}^{*}(s)$. Thus the bootstrap CUSUM test statistic for independent sequences is a special case of the block CUSUM test statistic. Generally, larger $M$ is needed for stronger temporal dependence, while this would reduce the effective sample size. Some empirical performance of the block bootstrap CUSUM test is assessed in the SM Section~\ref{subsec:time_series_block_bootstrap_CUSUM_test}.
	\qed
	\end{rmk}


	\subsection{Estimating the change point location under the alternative hypothesis}
	\label{subsec:identification_cp}
	
	If a change point is detected in the mean vectors (i.e., $H_0$ is rejected), then our next goal is to identify the change point location $m$. Specifically, we estimate $t_m = m / n, m=1,\dots,n,$ where the data $X_1^n$ are observed at evenly spaced time points and their index variables are normalized to $[0,1]$. We consider the change point location estimator based on the generalized CUSUM statistics \citep{HarizWylieZhang2007_AoS}:
	\begin{equation}
	\label{eqn:generalized_cusum}
	Z_{\theta,n}(s) = \left[s(n-s) \over n\right]^{1-\theta} \left({1\over s} \sum_{i=1}^s X_i - {1 \over n-s} \sum_{i=s+1}^n X_i \right),
	\end{equation}
	where $0 \le \theta < 1$ is a weighting parameter. Obviously, the CUSUM statistics $Z_n(s)$ in (\ref{eqn:cusum_mean_Rp}) is a special case of $\theta = 1/2$, i.e., $Z_n(s) = Z_{1/2,n}(s)$. Then we estimate $m$ by 
	\begin{equation}
	\label{eqn:cusum_identification}
	\hat{m}_\theta = \argmax_{1 \le s < n} |Z_{\theta,n}(s)|_\infty.
	\end{equation}
	and we use $t_{\hat{m}_\theta} = \hat{m}_\theta / n$ to estimate $t_{m}$. It is seen that, for smaller values of $\theta$, $Z_{\theta,n}(s)$ assigns less weights on the boundary data points. Therefore, if the true change point location is bounded away from the two endpoints, we expect that $t_{\hat{m}_\theta}$ with a smaller weighting parameter can achieve better rate of convergence. For example, if $t_m \in (0,1)$ is fixed and $p=1$, then it is known that the $\{Z_{0,n}(s)\}_{s=1}^{n-1}$ converges weakly to a functional of the Weiner process and the corresponding maximizer $\hat{m}_0$ achieves the best possible rate of convergence of the order $n^{-1}$ \citep{bai1997,HarizWylieZhang2007_AoS}. Instead of considering the whole family of the generalized CUSUM statistics indexed by $\theta \in [0,1)$, we consider two important cases of $\theta = 1/2$ (covariance stationary) and $\theta = 0$ (non-stationary) in this paper. For $\theta = 1/2$, $Z_{1/2,n}(s)$ is related to the proposed bootstrap CUSUM statistics $Z_n^*(s)$ in (\ref{eqn:bootstrap_cusum}) and the log-ratio statistics in (\ref{eqn:max_likelihood_ratio}) under normality with $\Sigma = \sigma^2 \Id_p$. For $\theta = 0$, $Z_{0,n}(s)$ is related to the parametric bootstrap in \citep{jirak2015}.
	
	\begin{rmk}[Comments on the boundary removal]
		\label{rmk:boundary_removal}
		It should be noted that we must remove boundary points to approximate the distribution of $T_n$. If the boundary points are included in the maxima $T_n$ and $T^*_n$, then the conditional distribution of $T^*_n$ (given $X_1^n$) does not provide an accurate approximation to the distribution of $T_n$. Theorem~\ref{thm:bootstrap_approx_cusum_statistic} and \ref{thm:power_bootstrap_cusum_test} provide the precise rate of convergence that characterizes the boundary removal parameter $\us$ to ensure the consistency (in terms of the sum of type I and type II errors) of the bootstrap CUSUM test. On the other hand, the estimation problem in (\ref{eqn:cusum_identification}) does not exclude the endpoints outside $[\us, n-\us]$. However, in practice, if the existence of a change point is not known as a priori and it is decided by a test, then the boundary restriction is implicitly imposed for both testing and estimation in empirical applications \citep{bai1997}. Further discussions on the theoretical choice of $\us$ can be found in Remark~\ref{rem:theoretical_choice_boundary_removal}.
		\qed
	\end{rmk}
	
	\subsection{Bootstrap-assisted binary segmentation (BABS) for multiple change points}
	\label{subsec:methodology_mcp}
	Suppose there are $\nu$ change points $m_0 = 1 < m_1 < \dots < m_\nu < m_{\nu+1} = n$ and consider the following multiple mean-shifts model:
	\begin{equation}
	\label{eqn:model_mcp}
	X_i =  \mu + \sum_{k=1}^\nu \delta_n^{(k)} 1 (i > m_{k}) + \xi_i, \quad i = 1,\dots,n,
	\end{equation}
	where $\delta_n^{(k)} \in \R^p$ are non-zero mean-shift vectors and $\xi_{i}$ are again i.i.d.\ mean-zero random vectors in $\R^{p}$.	Without loss of generality, we may assume $\mu=0$ and $\delta_n^{(0)} = \delta_n^{(\nu+1)} = 0$ such that the mean vectors $\mu_i = \E[X_i]$ are piecewise constant $\mu_{1+m_k} = \dots = \mu_{m_{k+1}} = \sum_{l=0}^k \delta_{n}^{(l)}$. Given a beginning time point $b$ and an ending time point $e$, we can compute the CUSUM statistics on the initial data segment $\{X_i\}_{i=b}^e$:
	\[
	Z_{n,b,e}(s) = \sqrt{(s-b+1)(e-s) \over e-b+1} \Big({1\over s-b+1} \sum_{i=b}^s X_i - {1 \over e-s} \sum_{i=s+1}^e X_i \Big).
	\]
	Note that the normalization in $Z_{n,b,e}(s)$ corresponds to the case $\theta = 1/2$ in~\eqref{eqn:generalized_cusum}. It can be shown that the maximizer of $|\E Z_{n,b,e}(s)|_\infty, s=b, \dots, e$ always occurs at one of the change points $\{m_k, k = 1, \dots, \nu\} \cap [b,e]$ (cf. Lemma~\ref{lem:max_delta_s_location} in the SM Section~\ref{sec:proofs}). Therefore, under multiple change points model~\eqref{eqn:model_mcp}, we can use $Z_{n,b,e}(s)$ to locate one shift in the interval $[b,e]$.  If our bootstrap CUSUM test (calculated based on $\{X_i\}_{i=b}^e$) rejects $H_0$ at the significance level $\alpha$, then $\hat{m}_b^e = \argmax_{s=b, \dots, e} |Z_{n,b,e}(s)|_\infty $ is marked as a change point. 
	Thus, we may recursively apply the binary segmentation to search along the two directions $[b,\hat{m}_b^e]$ and $[\hat{m}_b^e +1, e]$ until no further change point would be detected by the subsequent bootstrap tests. The pseudo-code for our bootstrap-assisted binary segmentation algorithm for multiple change points detection and estimation, referred as BABS($\alpha,b, e$), is summarized in Algorithms \ref{alg:multi_cp}.
	
	\begin{algorithm}[htp]
		\caption{BABS($\alpha,b, e$)}\label{alg:multi_cp}
		\begin{algorithmic}[1]
			
			\IF {$e - b + 1 < 2\us$}
			\STATE {STOP}
			\ELSE
			\STATE $\hat{m}_b^e = \argmax_{s=b, \dots, e} |Z_{n,b,e}(s)|_\infty $
			\IF {our bootstrap CUSUM test concludes the existence of a change in $[b,e]$}
			\STATE {add $\hat{m}_b^e$ to the set of estimated change-points;}
			\STATE {BABS($\alpha, b, \hat{m}_b^e$);}
			\STATE {BABS($\alpha, \hat{m}_b^e +1, e$).}
			\ELSE
			\STATE {STOP}
			\ENDIF
			\ENDIF
			\RETURN{estimated change points.}
		\end{algorithmic}
	\end{algorithm}

	\section{Theoretical results}
	\label{sec:main_results}

	\subsection{Size and power of the bootstrap CUSUM test: single change point}
	\label{subsec:validity_bootstrap_cp_test}
	
	Denote $\Prob_0(\cdot)$ and $\Prob_1(\cdot)$ as the probability computed under $H_0$ and $H_1$, respectively.
	Our first main result (cf.\ Theorem \ref{thm:bootstrap_approx_cusum_statistic}) is to establish finite sample bounds for the (random) Kolmogorov distance between $T_n$ and $T_n^*$: $\rho^*(T_n, T_n^*) = \sup_{t \in \R} |\Prob_0(T_n \le t) - \Prob_0(T_n^* \le t | X_1^n)|.$
	From this, we can derive the asymptotic bootstrap validity for certain high-dimensional scaling limit of $(n, p)$. Particularly, with $\rho^*(T_n, T_n^*) = o_\Prob(1)$, we can show that type I error of the test is asymptotically controlled at the exact nominal level $\alpha \in (0,1)$, i.e., $\Prob_0(T_n > q_{T_{n}^{*} \mid X_{1}^{n}}(1-\alpha)) \to \alpha$ (cf. Corollary~\ref{cor:validitiy_bootstrap_cusum_test}).	
	Let $\ub, \bar{b}, q > 0$. We make the assumptions:
	\begin{enumerate}
		\setlength\itemsep{0em}
		\item[(A)] $\Var(\xi_{ij}) \ge \ub$ for all $j = 1,\dots,p$.
		\item[(B)] $\E[|\xi_{ij}|^{2+\ell}] \le \bar{b}^\ell$ for $\ell=1,2$ and for all $i=1,\dots,n$ and $j=1,\dots,p$.
		\item[(C)] $\|\xi_{ij}\|_{\psi_1} \le \bar{b}$ for all $i=1,\dots,n$ and $j=1,\dots,p$.
		\item[(D)] $\E[\max_{1 \le j \le p} (|\xi_{ij}|/\bar{b})^q] \le 1$ for all $i=1,\dots,n$.
	\end{enumerate}
	Condition (A) is a non-degeneracy assumption. Condition (B) is a mild moment growth condition. Without loss of generality, we may take $\bar{b} \ge 1$. Conditions (C) and (D) impose sub-exponential and uniform polynomial moment requirements on the observations, respectively. Define $\varpi_{1,n} = \Big({\log^7 (np) \over \us}\Big)^{1/6}$ and $\varpi_{2,n} = \Big({n^{2/q} \log^3 (np) \over \gamma^{2/q} \us}\Big)^{1/3}.$
	
	\begin{thm}[{\bf Main result I:} bounds on the Kolmogorov distance between $T_n$ and $T_n^*$ under $H_0$]
		\label{thm:bootstrap_approx_cusum_statistic}
		Suppose $H_0$ is true and assume (A) and (B) hold. Let $\gamma \in (0, e^{-1})$ and suppose that $\log(\gamma^{-1}) \le K \log(pn)$ for some constant $K > 0$. \\
		(i) If (C) holds, then there exists a constant $C > 0$ only depending on $\ub, \bar{b}, K$ such that
		\begin{equation}
		\label{eqn:bootstrap_cusum_subexp}
		\Prob(\rho^*(T_n, T_n^*) \le C \varpi_{1,n}) \geq 1 - \gamma.
		\end{equation}
		(ii) If (D) holds, then there exists a constant $C > 0$ only depending on $\ub, \bar{b}, K, q$ such that
		\begin{equation}
		\label{eqn:bootstrap_cusum_poly}
		\Prob(\rho^*(T_n, T_n^*) \le C \{\varpi_{1,n} + \varpi_{2,n} \}) \geq 1 - \gamma.
		\end{equation}
	\end{thm}
	
	Based on Theorem \ref{thm:bootstrap_approx_cusum_statistic}, we have the uniform size validity of the bootstrap CUSUM test.
	
	\begin{cor}[Uniform size validity of the bootstrap CUSUM test]
		\label{cor:validitiy_bootstrap_cusum_test}
		Suppose $H_0$ is true and assume (A) and (B) hold. Let $\gamma \in (0, e^{-1})$ and suppose that $\log(\gamma^{-1}) \le K \log(pn)$ for some constant $K > 0$. \\
		(i) If (C) holds, then there exists a constant $C > 0$ only depending on $\ub, \bar{b}, K$ such that
		\begin{equation}
		\label{eqn:size_validity_bootstrap_cusum_subexp}
		\sup_{\alpha \in (0,1)} |\Prob_0(T_n \le q_{T_{n}^{*} \mid X_{1}^{n}}(\alpha)) - \alpha| \le C \varpi_{1,n} + \gamma.
		\end{equation}
		Consequently,  if $\log^{7}(np) = o(\us)$, then $\Prob_0(T_n \le q_{T_{n}^{*} \mid X_{1}^{n}}(\alpha)) \to \alpha$ uniformly in $\alpha \in (0,1)$ as $n \to \infty$.\\
		(ii) If (D) holds, then there exists a constant $C > 0$ only depending on $\ub, \bar{b}, K, q$ such that
		\begin{equation}
		\label{eqn:size_validity_bootstrap_cusum_poly}
		\sup_{\alpha \in (0,1)} |\Prob_0(T_n \le q_{T_{n}^{*} \mid X_{1}^{n}}(\alpha)) - \alpha| \le C \{ \varpi_{1,n} + \varpi_{2,n} \} + \gamma.
		\end{equation}
		Consequently, if $\max\{\log^{7}(np), n^{2/q} \log^{3+\varepsilon}(np)\} = o(\us)$ for some $\varepsilon > 0$, then $\Prob_0(T_n \le q_{T_{n}^{*} \mid X_{1}^{n}}(\alpha)) \to \alpha$ uniformly in $\alpha \in (0,1)$ as $n \to \infty$.
	\end{cor}
	
	\begin{rmk}[Choice of boundary removal parameter]
	\label{rem:theoretical_choice_boundary_removal}
		There is a trade-off for the choice of boundary removal parameter: the larger $\us$, the smaller of the error bounds and the more data points are removed from the change point detection (so that the regime allowed by the bootstrap CUSUM test is smaller). In theory, the lower bound of $\us$ is given in Corollary~\ref{cor:validitiy_bootstrap_cusum_test} for size validity of the bootstrap CUSUM test. Specifically, if the data distribution has sub-exponential tail (i.e., Condition (C) holds), then we need $\us \gg \log^7 (np)$ for $ \varpi_{1,n} = o(1)$; if the data distribution has polynomial tail (i.e., Condition (D) holds) with $q > 0$, then we need $\us \gg \max\{\log^7 (np), n^{2/q} \log^{3+\varepsilon} (np)\}$ for $ \varpi_{1,n}+\varpi_{2,n} = o(1)$. This implies that we can choose $\us = c_1 n^{c_2}$ for some constants $c_1>0, 1 \ge c_2 >0$ in either sub-exponential or polynomial case. 
		
		As a leading example, we consider a {\it fixed} normalized true change point location $t_m = m / n \in (0,1)$. Then we may choose $\us = c_{0} n$ for some small constant $c_{0} > 0$  in order to include the true change point in the interval $[\us, n-\us]$. For this setup, asymptotic size validity of the bootstrap CUSUM test is obtained if $p = O(e^{n^c})$ for some $1/7 > c > 0$ under the sub-exponential moment condition on the observations.
		\qed
	\end{rmk}
	
	
	Our second main result is to analyze the power of the bootstrap CUSUM test. We are mainly interested in characterizing the change point signal strength (quantified by the $\ell^\infty$ norm of $\delta_n$) and the location $t_m$ such that $H_0$ and $H_1$ can be (asymptotically) separated by our bootstrap CUSUM test. Without loss of generality, we may assume that $|\delta_n|_\infty \le 1$.
	
	\begin{thm}[{\bf Main result II:} power of the bootstrap CUSUM test under $H_1$]
		\label{thm:power_bootstrap_cusum_test}
		Suppose $H_1$ is true with a change point $m \in [\us, n-\us]$ and assume (A) and (B) hold. Let $\zeta \in (0, 1/2)$ and $\gamma \in (0, e^{-1})$ such that $\log(\gamma^{-1}) \le K \log(np)$ for some constant $K > 0$. \\
		(i) If (C) holds and 
		\begin{equation}
		\label{eqn:power_signal_lower_bound}
		|\delta_n|_\infty \ge {C_1 \sqrt{\log(\zeta^{-1}) \log(np) + \log(np/\alpha) \over n t_m (1-t_m)}}
		\end{equation}
		for some large enough constant $C_1 := C_1(\bar{b},\ub,K) > 0$, then there exists a constant $C_2 := C_2(\bar{b},\ub,K) > 0$ such that 
		\begin{equation}
		\label{eqn:power_lower_bound_subexp}
		\Prob_1(T_n > q_{T_{n}^{*} \mid X_{1}^{n}}(1-\alpha)) \ge 1 - \gamma - C_2 \varpi_{1,n} - 2 \zeta.
		\end{equation}
		(ii) If (D) holds and $|\delta_n|_\infty$ obeys (\ref{eqn:power_signal_lower_bound}) for some large enough constant $C_1 := C_1(\bar{b},\ub,K,q) > 0$, then there exists a constant $C_2 := C_2(\bar{b},\ub,K,q) > 0$ such that 
		\begin{equation}
		\label{eqn:power_lower_bound_poly}
		\Prob_1(T_n > q_{T_{n}^{*} \mid X_{1}^{n}}(1-\alpha)) \ge 1 - \gamma - C_2 \{ \varpi_{1,n} + \varpi_{2,n} \} - 2 \zeta.
		\end{equation}
	\end{thm}
	
	
	\begin{rmk}[Rate-optimality on the change point detection for sparse alternatives]
		For i.i.d.\ Gaussian errors $\xi_i \sim N(0, \Id_p)$ in the mean-shift model (\ref{eqn:mean_shifting_model}), the change point detection boundary is characterized in \cite{enikeevaharchaoui2014}. Suppose a change $a > 0$ occurs in the first $k$ components of $\delta_n = (a, \dots, a, 0, \dots, 0)^\top$
		at the location $m$ in the sequence $X_1,\dots,X_n$. Following \cite{enikeevaharchaoui2014}, we consider the scaling limit $p = n^{c_1}$ and $k = p^{1-c_2}$ for some $c_1 > 0$ and $c_2 \in [0,1)$. If $c_2 \in (1/2,1)$, then the number of components with a change point is highly sparse. In this case, the minimax separation for $H_0$ and $H_1$ is given by $a = r_p \sqrt{\log(p) / {n t_m (1-t_m)}}.$
		Specifically, detection is impossible if $\limsup_{p \to \infty} r_p < \sqrt{2 c_2 - 1}$ and detection is possible if $\liminf_{p \to \infty} r_p > \sqrt{2 c_2 / (1-\log{2})}$. Choosing $\alpha_n = n^{-c}$ for some constant $c > 0$ in Corollary \ref{cor:validitiy_bootstrap_cusum_test} and Theorem \ref{thm:power_bootstrap_cusum_test}, we see that if $a \ge C_* \sqrt{\log(\zeta^{-1}) \log(p) / n t_m (1-t_m)}$
		for some large constant $C_* > 0$, then our bootstrap CUSUM change point test achieves the minimax separation rate in the high sparsity regime (with stronger side conditions to ensure the bootstrap validity). Hence, the signal strength requirement for detection in the proposed bootstrap test achieves the minimax optimal rate under the sparse alternatives. On the other hand, under the dense alternatives $c_2 \in [0,1/2]$, our bootstrap CUSUM test remains consistent in detection in the sense that the sum of type I and type II errors converges to zero. However, in such case, the bootstrap CUSUM test does not achieve the detection boundary and the minimax separation rate \citep{enikeevaharchaoui2014}.
		\qed
	\end{rmk}
	
	\begin{rmk}[Monotonicity of power in the signal strength]
		Inspecting the proof of Theorem \ref{thm:power_bootstrap_cusum_test}, we see that the type II error of the bootstrap CUSUM test is bounded by a probability depending on the change point signal strength $|\delta_n|_\infty$ and location $m$ (cf. equation (\ref{eqn:type_II_error_lower_bound}) in the SM Section~\ref{sec:proofs}). Specifically, 
		\[
		\text{Type II error} \le \Prob_1(\tilde{T}_n \ge \tilde \Delta - q_{T_{n}^{*} \mid X_{1}^{n}}(1-\alpha)),
		\]
		where $\tilde\Delta = \sqrt{n t_m (1-t_m)} |\delta_n|_\infty$, $\tilde{T}_n = \max_{\us \le n \le n-\us} |Z_n^\xi(s)|$, and $Z_n^\xi(s)$ are the CUSUM statistics computed on the $\xi_1^n$ noise random vectors. Since the distribution of $\tilde{T}_n$ does not depend on $\delta_n$ and the conditional quantile $q_{T_{n}^{*} \mid X_{1}^{n}}(1-\alpha)$ is bounded by $O(\sqrt{\log(np)})$ with high probability under $H_1$, the power of the bootstrap CUSUM test is lower bounded by a quantity that is non-decreasing in $|\delta_n|_\infty$. Simulation examples in Section \ref{sec:simulation} confirm our theoretical observation. In addition, since $t_m (1-t_m)$ is maximized at $t_m = 1/2$, a change point near the middle is easier to detect than it is near the boundary.
		\qed
	\end{rmk}

	\subsection{Rate of convergence of the change point location estimator}
	\label{subsec:rate_convergence_cp_location_estimator}
	Our third main result is concerned with the rate of convergence of the change point location estimator $t_{\hat{m}_\theta}$, where $\hat{m}_\theta$ is defined through (\ref{eqn:cusum_identification}) and (\ref{eqn:generalized_cusum}). We first consider the case of $\theta = 1/2$ corresponding to the covariance stationary CUSUM statistics.
	
	\begin{thm}[{\bf Main result III:} rate of convergence for change point location estimator: $\theta = 1/2$]
		\label{thm:rate_location_estimator}
		Suppose that (B) holds and $H_1$ is true. Suppose that $\log(\gamma^{-1}) \le K\log(np)$ for some constant $K > 0$. \\
		(i) If (C) holds, then there exists a constant $C := C(\bar{b},K) > 0$ such that 
		\begin{equation}
		\label{eqn:rate_location_estimator_subexp}
		\Prob_1 \Big( |t_{\hat{m}_{1/2}} - t_{m}|  \le {C \log^2(np) \over \sqrt{n t_{m}(1-t_{m})} |\delta_n|_\infty} \Big) \ge 1- \gamma.
		\end{equation}
		(ii) If (D) holds with $q > 2$, then there exists a constant $C := C(\bar{b},K,q) > 0$ such that 
		\begin{equation}
		\label{eqn:rate_location_estimator_poly}
		\Prob_1 \Big( |t_{\hat{m}_{1/2}} - t_{m}|  \le {C n^{1/q} (\log(np) + \gamma^{-1/q}) \over \sqrt{n t_{m}(1-t_{m})} |\delta_n|_\infty} \Big) \ge 1- \gamma.
		\end{equation}
	\end{thm}
	
	Note that the non-degeneracy Condition (A) is not needed in estimating the change point location. Consider a fixed $t_m \in (0,1)$ as a leading example in Remark~\ref{rem:theoretical_choice_boundary_removal}. Theorem~\ref{thm:rate_location_estimator} guarantees consistency of $t_{\hat{m}_{1/2}}$ if the signal strength satisfying: i) $|\delta|_\infty \gg n^{-1/2} \log^2(np)$ in the sub-exponential moment case; ii) $|\delta|_\infty \gg n^{-1/2+1/q} \log(np)$ in the polynomial moment case. From Part (i) of Theorem \ref{thm:rate_location_estimator}, it should also be noted that the change point location estimator $t_{\hat{m}_{1/2}}$ does not attain the optimal rate of convergence. Consider the setup where $t_m \in (0,1)$, $p=1$, and $|\delta_n| = c$ is a constant signal. Then the rate of convergence in (\ref{eqn:rate_location_estimator_subexp}) reads $O(\log^2(n) / \sqrt{n})$; that is, up to a logarithmic factor, the change point estimator has the rate of convergence $n^{-1/2}$. In such setup, however, it is known that the best possible rate of convergence for estimating the change point location is $n^{-1}$ \citep{HarizWylieZhang2007_AoS}, which is achieved by maximizing $|Z_{0,n}(s)|$. Therefore, it is interesting to study the impact of dimensionality on the rate in the case of $\theta = 0$ when the true change point $t_m \in (0,1)$ is fixed. This is the content of the next theorem. Denote $\underline{\delta}_n = \min_{j \in \calS} |\delta_{nj}|$.

	\begin{thm}[{\bf Main result IV:} rate of convergence for change point location estimator: $\theta = 0$]
		\label{thm:rate_location_estimator_nonstationary}
		Suppose that (B) holds and $H_1$ is true with a change point $m$ satisfying $c_1 \le t_m \le c_2$ for some constants $c_1,c_2 \in (0,1)$. Suppose that $\log^3(np) \le K n$ and $\log(\gamma^{-1}) \le K\log(np)$ for some constant $K > 0$. \\
		(i) If (C) holds,
		then there exists a constant $C := C(\bar{b}, K, c_1, c_2) > 0$ such that 
		\begin{equation}
		\label{eqn:rate_cp_estimation_theta=0_subsexp}
		\Prob_1 \left( |t_{\hat{m}_0} - t_m| \le {C \log^4(np) \over n \underline{\delta}_n^2} \right) \ge 1-\gamma.
		\end{equation}
		(ii) If (D) holds for some $q \ge 2$, 
		then there exists a constant $C := C(\bar{b}, K, q, c_1, c_2) > 0$ such that 
		\begin{equation}
		\label{eqn:rate_cp_estimation_theta=0_poly}
		\Prob_1 \left( |t_{\hat{m}_0} - t_m| \le {C \log(np) \over n \underline{\delta}_{n}^2} \, \max \left\{ 1, {n^{2/q} \log(np) \over \gamma^{2/q}} \right\} \right) \ge 1-\gamma.
		\end{equation}
	\end{thm}
	
	Based on Theorem~\ref{thm:rate_location_estimator_nonstationary}, we see that the dimension impacts the optimal rate of convergence for estimating the change point location only on the logarithmic scale. Compared with Theorem~\ref{thm:rate_location_estimator}, we see that faster convergence of $t_{\hat{m}_0}$ than that of $t_{\hat{m}_{1/2}}$ is possible when $t_m \in (0,1)$ is fixed and the dimension is allowed to grow sub-exponentially fast in the sample size. On the other hand, $t_{\hat{m}_{1/2}}$ is more robust to estimate the change point when its location is near the boundary, i.e., $t_m \to 0$ and $t_m \to 1$ are allowed to maintain the consistency in Theorem~\ref{thm:rate_location_estimator}; see our simulation result in Section~\ref{sec:simulation} for numeric comparisons. 

	\subsection{Rate of convergence of BABS}
	\label{subsec:consistency_mcp}
	
	Under the multiple mean-shifts model~\eqref{eqn:model_mcp}, we consider the testing problem for $H_{0}$ against the alternative hypothesis with multiple change points
	\begin{equation}
	\label{eqn:change_point_mean_test_mcp}
	H_1^{'} : \delta_n^{(k)} \neq 0 \text{ for some } 1 = m_0 < m_1 < \dots < m_\nu < m_{\nu+1} = n \text{ and } \nu \ge 1.
	\end{equation}
	The following Lemma~\ref{lem:power_mcp} controls the power of our bootstrap CUSUM test based on $T^{*}_{n}$ in~\eqref{eqn:bootstrap_cusum_mean_L_inf} in presence of multiple change points. This is the initial step of the BABS (Algorithm~\ref{alg:multi_cp}) and the power control is crucial for deriving the overall rate of convergence of recursively estimating the multiple change points. Denote $\bar\delta_n = \min_{k=1, \dots, \nu} |\delta_n^{(k)}|_\infty$.
	
	\begin{lem}[Power of the bootstrap CUSUM test under $H'_1$]
		\label{lem:power_mcp}
		Suppose $H_1^{'}$ is true under the multiple mean-shift model (\ref{eqn:model_mcp}) with $\{m_k\}_{k=1}^\nu \subset [\us, n-\us]$. Assume (A), (B) and $\min_{k=0, \dots, \nu} |m_{k+1} - m_k| \ge D_\nu$ for some $D_\nu > \us$. Let $\zeta \in (0, 1/2)$ and $\gamma \in (0, e^{-1})$ such that $\log(\gamma^{-1}) \le K \log(np)$ for some constant $K > 0$. \\
		(i) If (C) holds and 
		\begin{equation}
		\label{eqn:power_signal_lower_bound_mcp}
		\max_{k=1, \dots, \nu} { D_\nu^2 \bar\delta_n \over \sqrt{n^3 t_{m_{k}} (1-t_{m_{k}})}} \ge C_1 \left[ \sqrt{\log(\zeta^{-1}) \log(np) + \nu^2 \log(np/\alpha)}  \right]
		\end{equation}
		for some large enough constant $C_1 := C_1(\bar{b},\ub,K) > 0$, then there exists a constant $C_2 := C_2(\bar{b},\ub,K) > 0$ such that (\ref{eqn:power_lower_bound_subexp}) holds.\\
		(ii) If (D) holds and $|\delta_n|_\infty$ obeys (\ref{eqn:power_signal_lower_bound_mcp}) for some large enough constant $C_1 := C_1(\bar{b},\ub,K,q) > 0$, then there exists a constant $C_2 := C_2(\bar{b},\ub,K,q) > 0$ such that (\ref{eqn:power_lower_bound_poly}) holds.
	\end{lem}
	
	\begin{rmk} [Comments on the signal strength under multiple change points alternative]
		The signal strength on the LHS of (\ref{eqn:power_signal_lower_bound_mcp}) depends on the smallest mean shift $\bar\delta_n$ in $\ell^\infty$-norm and change point locations that are closest to boundary, which is the most difficult situation for CUSUM statistic to detect mean change.
		If $\nu = 1$, then $D_\nu = \min\{m, n-m\}$ and ${m(n-m) / n} \le D_\nu \le {2m(n-m) / n}.$
		Thus the LHS of  (\ref{eqn:power_signal_lower_bound_mcp}) has the same order as $n^{1/2} (t_m(1-t_m))^{3/2} |\delta_{n}|_\infty$, which is stronger than the requirement of lower bound $n^{1/2} (t_m (1-t_m))^{1/2} |\delta_{n}|_\infty$ in Theorem~\ref{thm:power_bootstrap_cusum_test}.  This extra cost comes from handling the possible mean shift cancellation in analyzing the general case of multiple change points. If the single change point is bounded from boundaries (i.e., $t_m$ can be treated as a constant), then Lemma \ref{lem:power_mcp} gives the same lower bound~\eqref{eqn:power_signal_lower_bound} as in Theorem \ref{thm:power_bootstrap_cusum_test}.
		\qed
	\end{rmk}

	Now we turn to the bootstrap-assisted binary segmentation algorithm BABS($\alpha,b, e$). We make the following assumptions in addition to (A)-(D).
	\begin{enumerate}[label=\alph*)]
		\setlength\itemsep{0em}
		\item $\min_{k=0, \dots, \nu} |m_{k+1} - m_k| \ge D_\nu$, where $D_\nu \ge n^{\Theta}$ for some $\Theta \le 1$.
		\item $\min_{k=1, \dots, \nu} \min_{j \in \calD_k} |\delta^{(k)}_{nj}| \ge \udelta_n$, where $\calD_k = \{1 \le j \le p: \delta^{(k)}_{nj} \neq 0 \}$ and $\udelta_n \ge n^{-\omega}$ for some $\omega \ge 0$.
		\item $\Theta - {\omega \over 2} > {3 \over 4}$.
		\item $n^{{3 \over 2}\Theta-1-\omega} >  C \max \{ \log^2(np), \sqrt{\log(\zeta^{-1}) \log(np) + \nu^2 \log(np/\alpha)} \}$, where $\zeta, K$ are constants defined in Lemma~\ref{lem:power_mcp}. Here $C > 0$ is a constant depending only on $\bar{b},\ub,K$ under (C) and on $\bar{b},\ub,K,q$ under (D).
		\item $\epsilon_n < \us < D_\nu$, where $\epsilon_n$ is defined in~\eqref{eqn:eps_bs} below.
	\end{enumerate}
	Assumptions a)-c) are standard signal separation and strength requirements in estimating the multiple change point locations via binary segmentation, see e.g.,Theorem 1 in \citep{fryzlewicz2014}. Assumption d) ensures that the bootstrap CUSUM test is able to consistently detect the mean-shift signals (cf. Lemma~\ref{lem:power_mcp}). 
	Under Assumption d), the expected signal size $|\E Z_{n,b,e}(s)|_\infty$ dominates the random behavior of $Z_{n,b,e}(s)$, thus obeying (\ref{eqn:power_signal_lower_bound_mcp}) with large probability.
	Assumption e) is a minimal condition on the boundary removal parameter $\us$, which is smaller than the separation distance between any consecutive change points and larger than the rate of convergence $\epsilon_n$ for consistently estimating all change point locations. Note that the signal strength requirement in estimation depends on $\min_{j \in \calD_k} |\delta^{(k)}_{nj}|$ in assumption b), which is typically stronger than $\max_{1 \le j \le p} |\delta^{(k)}_{nj}|$ used in the testing problem.

	
	
	\begin{thm}[{\bf Main result V:} rate of convergence of BABS]
		\label{thm:binseg_consistency_mcp}
		Let $\hat{\nu}$ denote the number of change points and $\hat{m}_1 < \dots <\hat{m}_{\hat{\nu}}$ the change point locations estimated from BABS$(\alpha,1,n)$.  Assume (A), (B) and a)-e) hold.
		Let $\gamma \in (0, e^{-1})$ such that $\log (\gamma^{-1}) \le K \log( D_\nu p) \le K \log( n p) $ and $\zeta$ is defined as in Theorem \ref{thm:power_bootstrap_cusum_test}. Define
		\begin{equation}
		\label{eqn:eps_bs}
		\epsilon_n = \left\{
		\begin{array}{cc}
		{n^2 \log^4(np) \over D_\nu^{2} \udelta_n^{2} }, & \mbox{if (C) holds} \\
		{n^{2+6/q} (\log^2(np) + \gamma^{-2/q}) \over D_\nu^{2} \udelta_n^{2}}, & \mbox{if (D) holds}\\
		\end{array}
		\right. .
		\end{equation}
		(i) If (C) holds, then there exist constants $C_0 = C_0(\bar{b}, \ub, K), C_0^{'}=C_0^{'}(\alpha, \bar{b}, K)$ such that
		\[
		\Prob(\calS_n) \ge 1- 2 \gamma - \nu (\gamma +2\zeta + C_0 \varpi_{1,n}) - (\nu+1) \alpha,
		\]
		where $\calS_n = \left\{ \hat{\nu} = \nu \mbox{ and } \max_{k=1, \dots, \nu} |\hat{m}_k - m_k| \le C_0^{'} \epsilon_n \right\}$.\\
		(ii) If (D) holds, then there exist constants $C_0 = C_0(\bar{b}, \ub, K, q), C_0^{'}=C_0^{'}(\alpha, \bar{b}, K, q)$ such~that
		\[
		\Prob(\calS_n) \ge 1- 2 \gamma - \nu \{\gamma +2\zeta + C_0 (\varpi_{1,n}+\varpi_{2,n})\} - (\nu+1) \alpha.
		\]
	\end{thm}
	
	Theorem~\ref{thm:binseg_consistency_mcp} reveals an interesting size-power trade-off of the BABS algorithm in multiple change point detection and estimation. For smaller $\alpha$, stronger signal strength and larger change point separation are needed to fulfill Assumption d) required by Theorem~\ref{thm:binseg_consistency_mcp}. For larger $\alpha$, the bootstrap CUSUM test used in BABS tends to reject more $H_0$ and consequently the BABS is likely to over-estimate the number of change points under $H'_1$. This is reflected by the multiple testing term $(\nu+1) \alpha$ in the lower bound of $\Prob(\calS_n)$ in Theorem~\ref{thm:binseg_consistency_mcp}. Moreover, under $H'_1$, both type I error (quantified by $\alpha$) and type II error (quantified by $\zeta$), together with their Bonferroni type multiple testing adjustment (quantified by $\nu$), affect consistency of the BABS algorithm.

	
	Binary segmentation was also considered in \cite{chofryzlewicz2015} and \cite{cho2016change}, both of which are consistent under their own conditions with different rate of convergence $\epsilon_n$. Our rate of $\epsilon_n$ is similar to that in \cite{chofryzlewicz2015} up to a logarithmic factor when $\Theta\in (3/4,1]$ and it is sharper than that in \cite{cho2016change} for sparse alternative. Comparisons on explicit rates are given in Remark~\ref{rmk:compare_BS} in the SM.


	\section{Simulation studies}
	\label{sec:simulation}
	
	In this section, we perform extensive simulation studies to investigate the size and power of the proposed bootstrap change point test, the estimation error of the change point location(s), as well as empirical performance of BABS. 
	In all setups, 200 bootstrap samples (if necessary) are drawn for each simulation.
	
	\subsection{Setup}
	We generate i.i.d.\ $\xi_i$ in the mean-shift model (\ref{eqn:mean_shifting_model}) from three distributions. 
	\begin{enumerate}[leftmargin=0.5cm,itemindent=.5cm,labelwidth=\itemindent,align=parleft]
		\setlength\itemsep{0.0em}
		\item Multivariate Gaussian distribution: $\xi_i \sim N(0, V)$.
		\item Multivariate elliptical $t$-distribution with degree of freedom $\nu$: $\xi_i \sim t_\nu(V)$ with the probability density function \cite[Chapter 1]{muirhead1982}
		\begin{equation*}
		f(x; \nu, V) = {\Gamma(\nu+p)/2 \over \Gamma(\nu/2) (\nu \pi)^{p/2} \det(V)^{1/2}} \left( 1 + {x^\top V^{-1} x \over \nu} \right)^{-(\nu+p)/2}.
		\end{equation*}
		The covariance matrix of $\xi_i$ is $\Sigma = \nu / (\nu-2) V$. In our simulation, we use $\nu=6$.
		\item Contaminated Gaussian: $\xi_i \sim \text{ctm-Gaussian}(\varepsilon, \nu, V)$ with density 
		\begin{equation*}
		f(x; \varepsilon, \nu, V) = (1-\varepsilon) \ {{\exp\left(-{x^\top V^{-1} x \over2} \right)} \over (2\pi)^{p/2} \det(V)^{1/2}}  + \varepsilon \ { {\exp\left(-{x^\top V^{-1} x \over2 \nu^2} \right)} \over (2 \pi \nu^2)^{p/2} \det(V)^{1/2}} . 
		\end{equation*}
		The covariance matrix of $\xi_i$ is $\Sigma = [ (1-\varepsilon) + \varepsilon \nu^2 ] V$. We will fix $\varepsilon = 0.2$ and $\nu = 2$.
	\end{enumerate}
	We consider three cross-sectional dependence structures of $V$ for each distribution. 
	\begin{enumerate}[leftmargin=0.5cm,itemindent=.7cm,labelwidth=\itemindent,align=parleft, label=(\Roman*)]
		\setlength\itemsep{0.0em}
		\item Independent: $V=\Id_p$, where $\Id_p$ is the $p \times p$ identity matrix.
		\item Strongly dependent (compound symmetry): $V=0.8 J + 0.2 \Id_p$, where $J$ is the $p \times p$ matrix containing all ones.
		\item Moderately dependent (autoregressive): $V_{ij} = 0.8^{|i-j|}$.
	\end{enumerate}

	\subsection{Simulation results for single change point model.}
	\label{subsec:simulation_single_cp}
	
	\subsubsection{Size of the bootstrap CUSUM test} 
	\label{subsubsec:size_comparison}
	
	We fix the sample size $n=500$ and vary the dimension $p=10, 300, 600$. For the bootstrap CUSUM test, we set the boundary removal parameter $\us=30, 40$. 
	For a significance level $\alpha \in (0, 1)$, we denote $\hat{R}(\alpha)$ as the proportion of empirically rejected null hypothesis in 1000 simulations. 
	
	Under $H_0$, the upper half of Table~\ref{tab:our_size_and_nominal05} reports the uniform error-in-size $\sup_{\alpha \in (0,1)} |\hat{R}(\alpha) - \alpha|$, a quantity that reflects the Kolmogorov distance between distributions of $T_n$ and its bootstrap analog $T_n^*$: the smaller uniform error-in-size, the closer $\rho^*(T_n,T_n^*)$. Each column corresponds to a combination of noise distributions and cross-sectional dependence structures. The lower half of Table~\ref{tab:our_size_and_nominal05} shows the empirical type I error $\hat{R}(\alpha)$ at the significance level $\alpha=0.05$. 
	We can draw several conclusions for our bootstrap CUSUM test by comparing results under different choices of boundary removal parameter, distribution family, and cross-sectional dependence structure. 
	First, in most cases the uniform error-in-size of $\us = 40$ are smaller than those of $\us = 30$, meaning  that the greater the $\us$, the better the approximation under $H_0$.
	Moreover, $\hat{R}(0.05)$ is generally close to the nominal size $0.05$ for $\us = 40$. 
	Next, uniform errors-in-size is usually smaller for the Gaussian distribution than that of $t_{6}$ or ctm-Gaussian cases. 
	Lower Table~\ref{tab:our_size_and_nominal05} delivers a similar message that Gaussianity helps to control $\hat{R}(0.05)$.
	Finally, our method is robust to the cross-sectional dependence structure. In many cases, stronger dependence (II$>$III$>$I) is more beneficial for reducing the approximation errors. 
	In summary, size can be better controlled if $\us$ is large, data are Gaussian distributed, and strong cross-sectional dependence exists.
	As a visualization of the accuracy for size control, Figure~\ref{fig: AlphaEg} displays three example curves of $\hat{R}(\alpha)$ for our proposed test where $p=600, \us=40$. The rejection rate $\hat{R}(\alpha)$ follows closely along the diagonal line in dash (i.e. the line of $\hat{R}(\alpha) = \alpha$).
	
	A thorough comparison between our bootstrap CUSUM test and two benchmark methods (i.e., bootstrap log-ratio of maximized likelihood test and the oracle test with known covariance matrix) can be found in the SM Section~\ref{subsec:compare_H0_benchmark_simu}. In SM Section~\ref{subsec:compare_H0_benchmark_simu}, we also compare our method with the test statistics in \cite{jirak2015} (denoted as $B_n$) and \cite{enikeevaharchaoui2014} (denoted as $\psi$) under the setting $n=500, p=600, \us=40$. 

	\begin{table}
	\caption{\label{tab:our_size_and_nominal05} Uniform error-in-size, $\sup_{\alpha \in [0,1]} |\hat{R}(\alpha) - \alpha|$, and empirical type I error with nominal level 0.05,  $\hat{R}(0.05)$, for our bootstrap CUSUM test under $H_0$, where $p = 10,300,600$, $\us = 30,40$ and data are simulated from all combinations of distribution families and covariance dependence structures.}
		\centering
		\fbox{%
		\begin{tabular}{c|l|lll|lll|lll}
			\multicolumn{2}{c|}{\multirow{2}{*}{}} & \multicolumn{3}{c|}{Gaussian} & \multicolumn{3}{c|}{$t_6$} & \multicolumn{3}{c}{ctm-Gaussian} \\ \cline{3-11} 
			\multicolumn{2}{c|}{}      & I        & II       & III     & I       & II      & III    & I         & II        & III       \\ \hline
			\multicolumn{11}{c}{$\sup_{\alpha \in [0,1]} |\hat{R}(\alpha) - \alpha|$} \\\hline
			\multirow{2}{*}{$p=10$}    & $\us=30$  & 0.034    & 0.036    & 0.041   & 0.048   & 0.041   & 0.039  & 0.036     & 0.042     & 0.021     \\ 
			& $\us=40$  & 0.042    & 0.034    & 0.037   & 0.043   & 0.037   & 0.033  & 0.041     & 0.042     & 0.043     \\ \hline 
			\multirow{2}{*}{$p=300$}   & $\us=30$  & 0.054    & 0.051    & 0.050   & 0.085   & 0.036   & 0.049  & 0.115     & 0.025     & 0.065     \\ 
			& $\us=40$  & 0.046    & 0.026    & 0.035   & 0.058   & 0.030   & 0.040  & 0.057     & 0.032     & 0.055     \\ \hline
			\multirow{2}{*}{$p=600$}   & $\us=30$  & 0.051    & 0.035    & 0.048   & 0.122   & 0.044   & 0.088  & 0.103     & 0.030     & 0.096     \\ \cline{2-11} 
			& $\us=40$  & 0.060    & 0.055    & 0.046   & 0.083   & 0.038   & 0.087  & 0.079     & 0.026     & 0.057    \\\hline
			\multicolumn{11}{c}{$\hat{R}(0.05)$} \\\hline
			\multirow{2}{*}{$p=10$}  & $\us=30$       & 0.051    & 0.052    & 0.051   & 0.044   & 0.056   & 0.034  & 0.043     & 0.056     & 0.052     \\ 
			& $\us=40$       & 0.046    & 0.055    & 0.052   & 0.045   & 0.050   & 0.048  & 0.054     & 0.053     & 0.040     \\ \hline
			\multirow{2}{*}{$p=300$}     & $\us=30$       & 0.026    & 0.054    & 0.039   & 0.018   & 0.034   & 0.018  & 0.021     & 0.043     & 0.027     \\
			& $\us=40$       & 0.045    & 0.043    & 0.044   & 0.024   & 0.046   & 0.036  & 0.026     & 0.056     & 0.034     \\ \hline
			\multirow{2}{*}{$p=600$}     & $\us=30$       & 0.026    & 0.060    & 0.027   & 0.010   & 0.034   & 0.020  & 0.010     & 0.053     & 0.019     \\
			& $\us=40$       & 0.031    & 0.038    & 0.036   & 0.020   & 0.044   & 0.016  & 0.015     & 0.042     & 0.027     \\
		\end{tabular} }
	\end{table}

	\begin{figure}
		\centering
		\includegraphics[trim=0 10 10 35,clip, width= 6in]{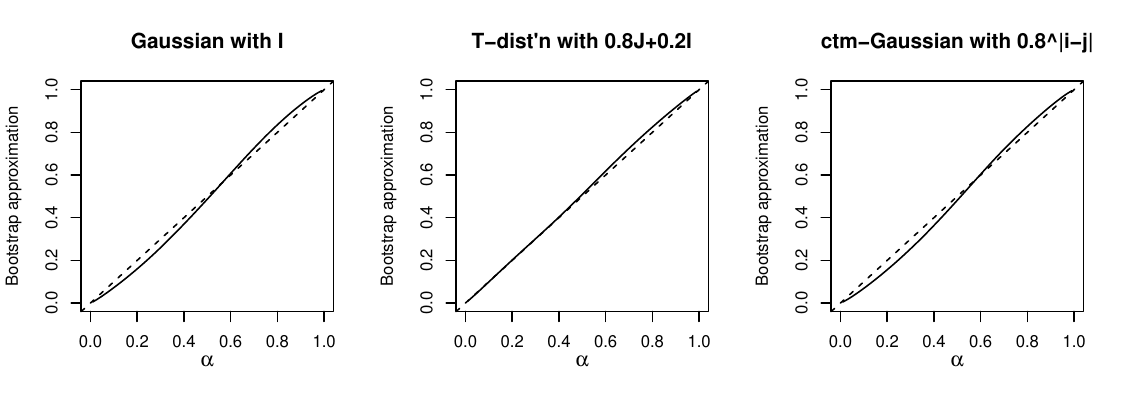} 
		\vspace{-1cm}
		\caption{Empirical rejection rate, $\hat{R}(\alpha)$, in selected data generating schemes under $H_0$: (Left) Gaussian distribution with Covariance I; (Middle) $t_6$ distribution with Covariance II; (Right) ctm-Gaussian distribution with Covariance III. Parameters: $n=500$, $p=600$, $\us=40$. }
		\label{fig: AlphaEg} 
	\end{figure}

	\subsubsection{Power of the bootstrap CUSUM test}\label{subsubsec:power_comparison}
	Under $H_1: \mu_1=\cdots=\mu_m \neq \mu_{m+1}=\cdots=\mu_n$, we consider the single change point location $m$ at $\{50, 150, 250\}$ (i.e.,\ $t_{m} = m/n = 1/10, 3/10, 5/10$ for $n=500$). 
	Denote $k$ as the number of components that have change points, i.e., $\delta_{n,1} = \cdots = \delta_{n,k} \ne 0$.
	Two types of the mean-shift signal are considered: $k=1$ for sparse signal and $k=50$ for dense signal. 
	Due to the space limit, we only present the sparse alternative case of $k=1$ in this section, and results of dense signal for $\psi$ can be found in the SM Section~\ref{subsec:additional_tabs_figs}. To analyze the power under $H_1$, we fix $n=500$, $p=600$, $\us=40$ and the significance level $\alpha=0.05$.  
	
	We first investigate the impact of change point location and distribution to our test. Figure~\ref{fig:power_our_compare_subfig:our} shows the empirical power curves v.s.\ the signal strength $|\delta_n|_\infty = |\delta_{n,1}|$. 
	In all cases, the powers monotonically increase and eventually reach 1 as $|\delta_{n}|_\infty$ gets large enough.
	Comparing the three curves corresponding to ctm-Gaussian distribution (I) in Figure~\ref{fig:power_our_compare_subfig:our}, we observe that change points closer to boundaries are harder to detect at the same signal strength. 
	When we narrow down to the four curves corresponding to $t_m = 1/10$, we can see the distributional influence: Gaussianity and cross-sectional dependence is helpful to improve the power. More simulation results of powers can be found in Table~\ref{tab:power} in the SM Section~\ref{subsec:additional_tabs_figs}. 
	
	\begin{figure}
		\centering
		\subfigure[]{
			\includegraphics[trim=10 0 10 40,clip,width=0.31\textwidth,height=0.23\textheight]{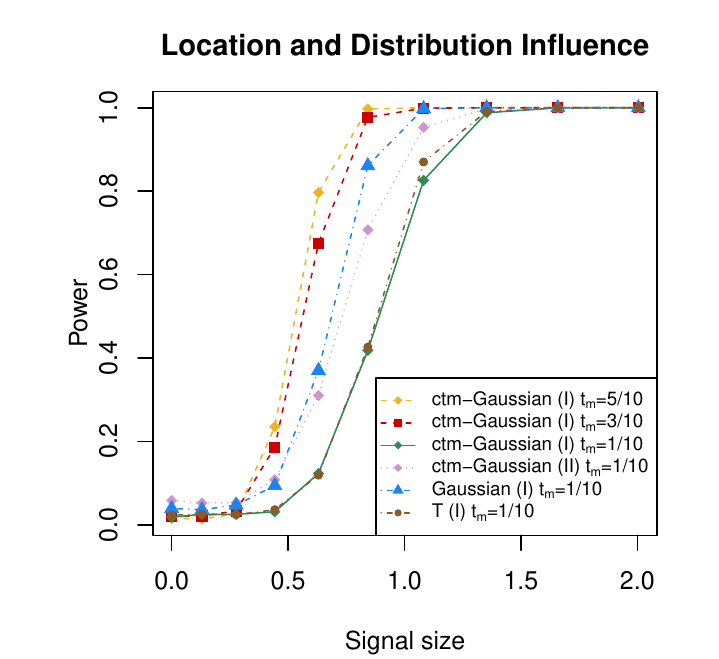}
			\label{fig:power_our_compare_subfig:our}}
		\subfigure[]{
			\includegraphics[trim=10 0 15 40,clip, width=0.31\textwidth,height=0.23\textheight]{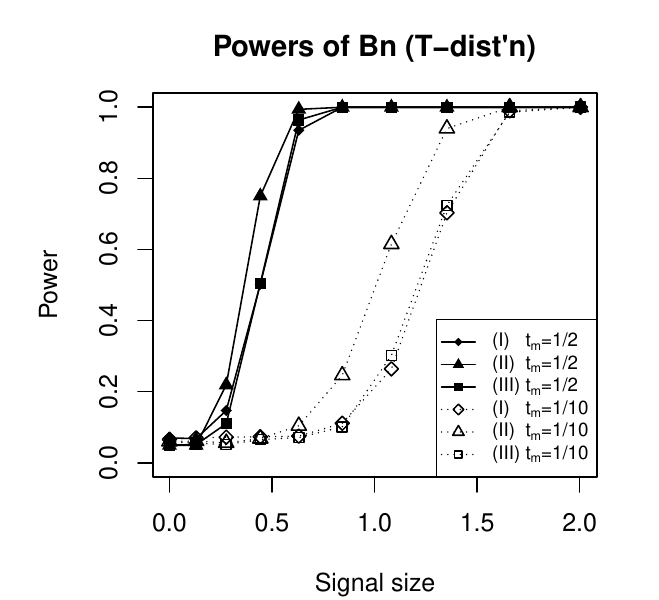}
			\label{fig:power_our_compare_subfig:a}}
		\subfigure[]{
			\includegraphics[trim=10 0 20 40,clip, width=0.31\textwidth,height=0.23\textheight]{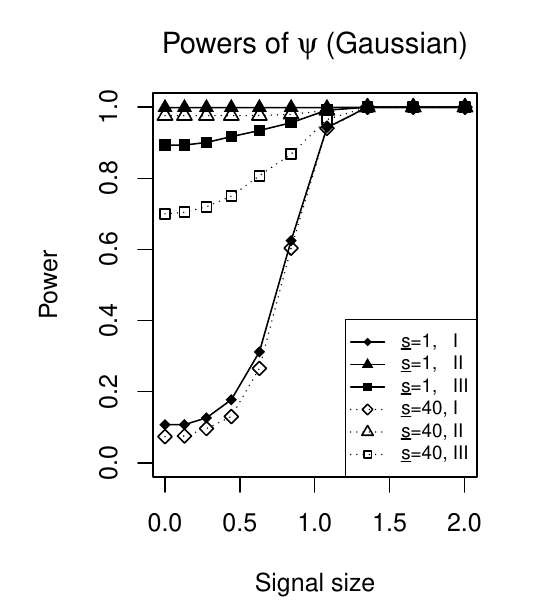}
			\label{fig:power_our_compare_subfig:b}}
			\vspace{-0.4cm}
		\caption{Power curves under sparse alternative: (a) $T_n^*$ (our) under selected distributions, covariances and $t_m=1/10, 3/10, 5/10$; (b) $B_n$ \citep{jirak2015} under $t_6$ distribution, three covariances and $t_m = 1/2, 1/10$; (c)  original and improved $\psi$ (i.e., $\us=1$ and $\us=40$) \citep{enikeevaharchaoui2014} under Gaussian distribution, three covariances and $t_m=1/2$. Parameters: $n=500, p=600$.}
		\label{fig:power_our_compare} 
	\end{figure}

	Next, we compare our method with $B_n$ of \cite{jirak2015} and $\psi$ of \cite{enikeevaharchaoui2014}. Figure~\ref{fig:power_our_compare_subfig:a} shows power trends of $B_n$ under $t$-distributed data for $k=1$ when $t_m=1/2,1/10$. The test $B_n$ performs better for central change point ($t_m=1/2$) than boundary change point ($t_m=1/10$). Compared to Figure~\ref{fig:power_our_compare_subfig:our}, we see that boundary change point brings more challenge to $B_n$ than to our test because our powers increase faster than $B_n$ under the same setup.
	Table~\ref{tab:Jirak_power} in the SM Section~\ref{subsec:additional_tabs_figs} gives a more detailed power report of $B_n$ in all scenarios under sparse $H_1$. We note that although $B_n$ returns slightly higher power than ours at $t_m=1/2$, it is computed with true long-run variance and it tends to over reject $H_0$ (i.e., size distortion).
	Figure~\ref{fig:power_our_compare_subfig:b} displays power trends of the $\psi$ ($\us=1$, no boundary removal) and $\psi$-improved (boundary removal with $\us=40$) for Gaussian distributed data at $t_m=1/2$. Neither $\psi$ test has valid power curve when the independent covariance assumption is violated. Furthermore, unreported results (due to the space concern) show invalid power curves in all other non-Gaussian distributed data. It is unsurprising since $\psi$ suffers from serious size distortion (cf.\ Table~\ref{tab:compareH0_size_nominal05} in the SM Section~\ref{subsec:additional_tabs_figs}). 
	We refer to Table~\ref{tab:Psi_power} in the SM Section~\ref{subsec:additional_tabs_figs} for complete power reports of $\psi$ in Gaussian scenarios with independent components under both sparse and dense $H_1$.

	\subsubsection{Performance of the location estimators}
	\label{subsubsec:location_simu}
	Now we examine the performance of our location estimators $t_{\hat{m}_{1/2}}$ and $t_{\hat{m}_0}$ under sparse alternative where $t_m = 1/10, 3/10, 5/10$. The performance measure is the root-mean-square error (RMSE). 
	Figure~\ref{fig:loc} shows comparison between $t_{\hat{m}_{1/2}}$ and $t_{\hat{m}_0}$ across different change point locations and signal sizes $|\delta_{n}|_\infty$. 
	First, Figure~\ref{fig:loc_subfig:a} illustrates that boundary change points (such as $t_m=1/10$) are harder to estimate as the RMSEs are uniformly larger than that for $t_m=5/10$. Second, as implied by Theorem~\ref{thm:rate_location_estimator} and \ref{thm:rate_location_estimator_nonstationary}, RMSEs of $t_{\hat{m}_{0}}$ are smaller than that of $t_{\hat{m}_{1/2}}$ at $t_m = 1/2$ because $Z_{\theta,n}(s)$ assigns less weights to the boundary points for smaller values of $\theta$. This is also empirically confirmed by Figure~\ref{fig:loc_subfig:b}: $t_{\hat{m}_{1/2}}$ is more in  favor of boundary points when $|\delta_{n}|_\infty=0$, and $t_{\hat{m}_0}$ slightly leans to the center when change ($|\delta_{n}|_\infty= 0.842$) is not in the middle of sequence.

	\begin{figure}
		\centering
		\subfigure[]{
			\includegraphics[trim=0 10 10 46.33,clip, width=0.3\textwidth]{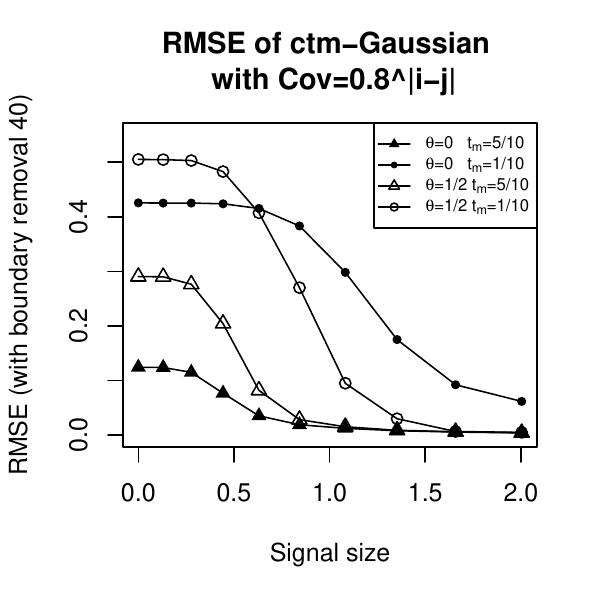}
			\label{fig:loc_subfig:a}}
		\hspace{0.15in}
		\subfigure[]{
			\includegraphics[trim=0 10 0 50,clip, width=0.6\textwidth]{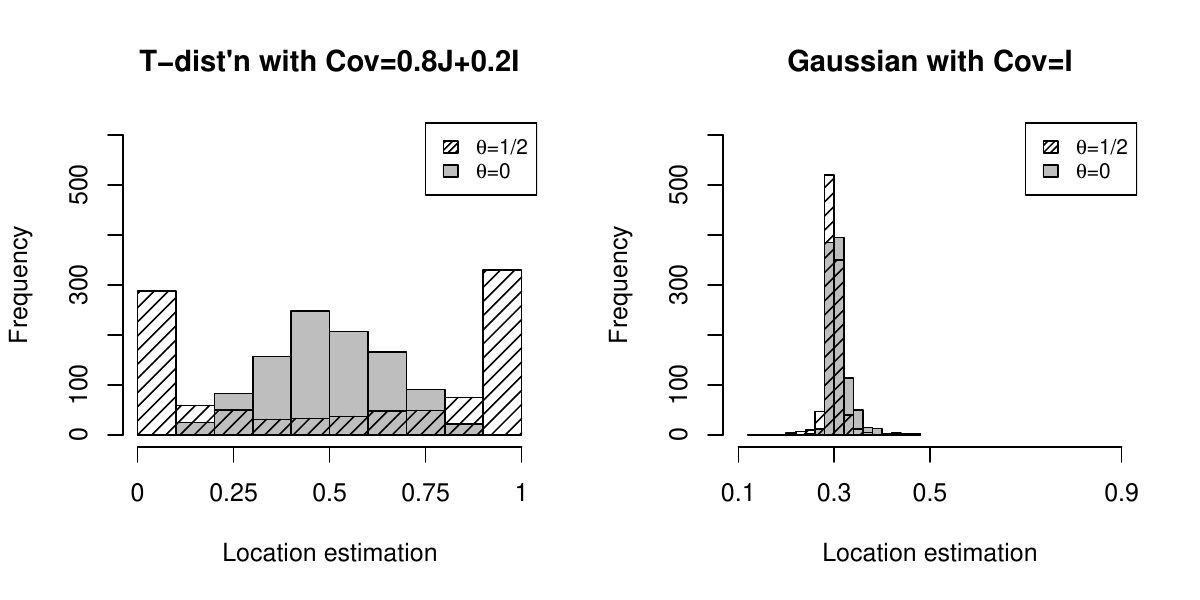}
			\label{fig:loc_subfig:b}}
			\vspace{-0.4cm}
		\caption{Comparison of our two estimators $t_{\hat{m}_{1/2}}$ and $t_{\hat{m}_0}$: (a) RMSEs v.s.\ $|\delta_{n,1}|$ for $t_m = 1/10, 5/10$ under ctm-Gaussian distribution with Covariance III; (b) histograms of estimated $t_m = 3/10$ over 1000 simulation for two scales of $\theta$ at $\delta_{n,1}=0$ under $t_6$ distribution with Covariance II (left) and at $\delta_{n,1}=0.842$ under Gaussian distribution with Covariance I (right). Parameters: $n=500, p=600$. }
		\label{fig:loc} 
	\end{figure}
	
	Next, we compare our estimators with \cite{wangsamworth2017} and \cite{chofryzlewicz2015}.	In \cite{wangsamworth2017}, a projection based estimator \texttt{Inspect} is proposed.	Theoretical analysis of this algorithm requires the data follow Gaussian distribution. In \cite{chofryzlewicz2015}, the proposed \texttt{SBS} estimator is the maximizer of threshold $\ell^1$-aggregated CUSUM statistics after thresholding. This method is sensitive to threshold tuning parameters selected by bootstrap. Both approaches allow multiple change points. For now, we first compare with their single change point versions (see R packages \texttt{InspectChangepoint} and \texttt{hdbinseg}). We also include a truncated version of our location estimator $\hat{m}_\theta = \argmax_{\us \le s < n-\us} |Z_{\theta,n}(s)|_\infty$ (cf. Remark~\ref{rmk:boundary_removal}) for fair comparison.
	Both $k=1, 50$ are considered, where $k$ represents signal density ($\delta_{n,1} = \cdots = \delta_{n,k} \ne 0$).

	Figure~\ref{fig:loc_subfig:wang} compares our non-truncated $t_{\hat{m}_0}$ with \texttt{Inspect} that has no boundary removal, where $t_m=3/10$ and data are $t_6$ distributed with Covariance II. The RMSEs of our estimator is uniformly smaller than \texttt{Inspect} if $k=1$ or data do not have an isotropic Gaussian distribution.
	Figure~\ref{fig:loc_subfig:cho} shows non-monotone RMSEs of \texttt{SBS} returned from \texttt{hdbinseg} when $\us = 40, t_m=5/10, k=1$ and data are from ctm-Gaussian with Covariance II. We also note that the empirical detection rate of change point drops from $91\%$ to $72\%$ when $|\delta_n|_\infty$ grows from 0.84 to 2. Although the \texttt{SBS} works well for $t_m=5/10,k=50$ or $t_m = 1/10,k=1$ (cf.\ SM Section~\ref{subsec:rmse_Supplementary material}), it means that large CUSUM values under sparse $H_1$ may lead to unreasonable selection of their threshold. The full RMSEs of our method with sparse signal and selected RMSEs of \cite{wangsamworth2017} and \cite{chofryzlewicz2015} are reported in Table~\ref{tab:RMSE_our}, \ref{tab:RMSE_Wang} and \ref{tab:RMSE_Cho} in the SM Section~\ref{subsec:rmse_Supplementary material}.

	\begin{figure}
		\centering
		\subfigure[]{
			\includegraphics[trim=0 10 280 55,clip,width=0.32\textwidth,height=0.23\textheight]{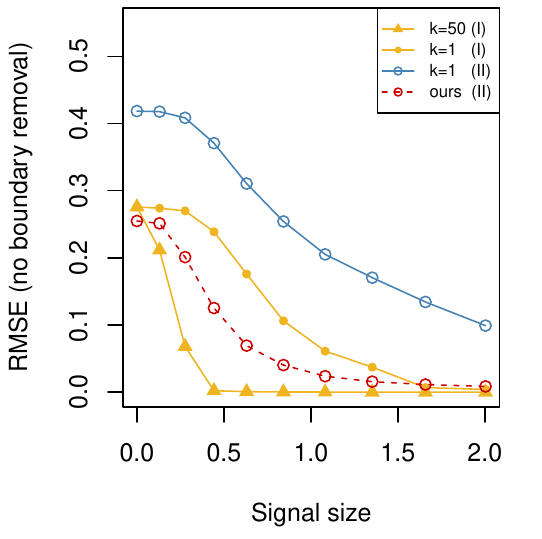}
			\label{fig:loc_subfig:wang}}
		\subfigure[]{
			\includegraphics[trim=250 10 0 55,clip,width=0.36\textwidth,height=0.23\textheight]{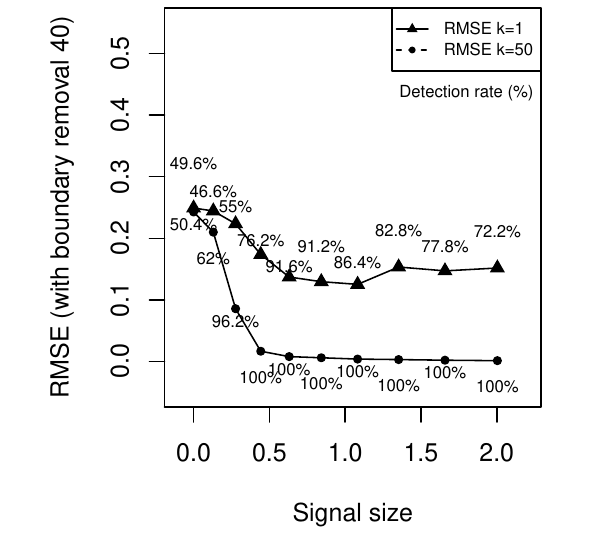}
			\label{fig:loc_subfig:cho}}
			\vspace{-0.4cm}
		\caption{RMSEs: (a) \texttt{Inspect} under $t_6$ distribution with Covariance II and $t_m = 3/10$; (b) \texttt{SBS} under ctm-Gaussian distribution with Covariance II and $t_m = 5/10$. Parameters: $n=500,p=600$.}
		\label{fig:loc_compare_iid} 
	\end{figure}


	\subsection{Multiple change points estimation using BABS}
	
	In the multiple change-point scenario, we first consider the $k$-th component of $\delta_n^{(k)}$ to have the same mean shift, i.e. $\delta_{n,1}^{(1)} = \delta_{n,2}^{(2)} = \dots = \delta_{n,\nu}^{(\nu)} = \delta \ne 0$.
	Since change point estimation can be viewed as a special case of clustering, the accuracy (consistency) can be measured by the Adjusted Rand Index (ARI) \citep{rand1971objective,hubert1985comparing}. We also report the average ARI over 500 simulations. The bootstrap resampling is $B = 200$.
	
	First, we provide simulation results for BABS using $t_6$ distribution with Covariance III as an illustrative example. We consider $n = 1000, p=1200$, $ \us = 40, \alpha = 0.05$ and the two change points $(m_1, m_2) = (300, 600)$. From Table~\ref{tab:mcp_identification_23} (left), we see that when signal is small (e.g.,\ $\delta = 0.317$), BABS cannot locate mean shift accurately. As signal gets larger, both the number and the locations of change points can be estimated more consistently, and ARI is also increasing to 1 (i.e.,\ perfect estimation).
	We further add one more change point, $(m_1, m_2, m_3) = (300, 600, 800)$. Table~\ref{tab:mcp_identification_23} (right) shows that the estimation is slightly worse under the same signal size in the left. This is because the effective sample size cuts down after each binary segmentation and one more multiple testing adjustment is needed in the stopping criterion of BABS. Nonetheless, our algorithm eventually detects all change points consistently when signal is large enough (e.g.,\ $\delta=2$).
	Figure~\ref{fig:mcp_identification_2} and \ref{fig:mcp_identification_3} in the SM Section~\ref{subsec:additional_tabs_figs} visualize this observation for the two cases, respectively.

	\begin{table}
		\caption{\label{tab:mcp_identification_23} Counts of estimated $\hat\nu$ and ARI along $\delta$ in multiple change point setup for BABS, where $(m_1, m_2) = (300, 600)$ (left) and $(m_1, m_2, m_3) = (300, 600, 800)$ (right). The numbers of correct estimation in each column are highlighted in bold. (Parameters: $n = 1000, p=1200$, $ \us = 40, \alpha = 0.05$ and data are from $t_6$-distribution with Covariance III.) 
		}
		\setlength\tabcolsep{2pt}
		\fbox{
			\begin{tabular}{cl|rrrrr| |rrrrr}
				\multicolumn{2}{c|}{} &\multicolumn{5}{c| |}{$(m_1, m_2) = (300, 600)$} & \multicolumn{5}{c}{$(m_1, m_2, m_3) = (300, 600, 800)$}\\ \hline
				\multicolumn{2}{c|}{$\delta$}   & 0   & 0.317 & 0.733 & 1.282 & 2.004  & 0   & 0.317 & 0.733 & 1.282 & 2.004  \\ \hline
				\multirow{4}{*}{\begin{tabular}[c]{@{}c@{}}\\Estimated\\ number of \\ change\\ points\end{tabular}} &0   & \textbf{497}     & 378     & 1       & 0       & 0   & \textbf{491}     & 360     & 0       & 0       & 0  \\
				&1   & 3       & 117     & 13      & 0       & 0     &  9       & 133     & 5       & 0       & 0 \\
				&2   & 0       & \textbf{5}       & \textbf{458}     & \textbf{464}     & \textbf{470}  &0       & 7       & 141     & 0       & 0 \\
				&3   & 0       & 0       & 26      & 35      & 30     & 0       & \textbf{0}       & \textbf{328}     & \textbf{455}     & \textbf{474} \\
				&4   & 0       & 0       & 2       & 1       & 0    	  &0       & 0       & 25      & 40      & 25  \\
				&5   & 0       & 0       & 0       & 0       & 0        &0       & 0       & 1       & 5       & 1    \\ \hline
				\multicolumn{2}{c|}{Sum}                                        & 500   & 500   & 500   & 500   & 500  & 500   & 500   & 500   & 500   & 500 \\ \hline
				\multicolumn{2}{c|}{ARI}                                        & 0.994 & 0.128 & 0.935 & 0.978 & 0.989    & 0.982 & 0.106 & 0.871 & 0.973 & 0.989\\                
			\end{tabular}}
	\end{table}
	
	Next, we experiment the setup in Section 5.3 of \cite{wangsamworth2017} to compare BABS with \texttt{Inspect}. Consider $n=2000, p=200, (m_1, m_2, m_3) = (500,1000,1500)$ and data are Gaussian distributed with identity cross-sectional covariance. We consider the complete-overlap mean structure, i.e., changes occur in the same $k$ coordinates, and set $(|\delta_{n}^{(1)}|_2 , |\delta_{n}^{(2)}|_2 , |\delta_{n}^{(3)}|_2) = (\delta, 2\delta, 3\delta)$ for signal strength $\delta \in \{ 0.4, 0.6\}$. Table~\ref{tab:mcp_compare_to_wang} below and Figure~\ref{fig:compare_binseg_inspect} in the SM Section~\ref{subsec:additional_tabs_figs} summarize estimation performance of our BABS  and the \texttt{Inspect}. If $k = 40$, $|\delta_{n}^{(i)}|_\infty = i \cdot \delta/\sqrt{k}$ is too small for our method to detect. However, when $k=1$ such that $\delta_{n}^{(i)}, i=1,2,3$ is sparse (with the $\ell^2$-norm unchanged), then our algorithm shows superiority in terms of both $\hat\nu$ and ARI. Again, our BABS algorithm has advantage when the $\ell^\infty$-norm of signal is bounded below and likewise when data are non-Gaussian or cross-sectionally dependent as shown in Section~\ref{subsubsec:location_simu} and \ref{subsubsec:power_comparison}. For comparison with \cite{chofryzlewicz2015, cho2016change, jirak2015}, we refer to Section 5.3 of \cite{wangsamworth2017} for a comprehensive simulation study.

	\begin{table}
		\caption{\label{tab:mcp_compare_to_wang} Comparison between BABS and \texttt{Inspect} for complete-overlap mean structure of change points that are dense ($k=40$) and sparse ($k=1$).  Parameters: $n=2000, p=200, (m_1, m_2, m_3) = (500,1000,1500)$ and $(|\delta_{n}^{(1)}|_2 , |\delta_{n}^{(2)}|_2 , |\delta_{n}^{(3)}|_2) = (\delta, 2\delta, 3\delta)$ for signal strength $\delta \in \{ 0.4, 0.6\}$.}
		\setlength\tabcolsep{2pt}
		\fbox{
		\begin{tabular}{cc|rrrr|rrrr}
			\multicolumn{2}{c|}{\multirow{3}{*}{\begin{tabular}[c]{@{}c@{}}Parameters\\ in \\ methods\end{tabular}}} & \multicolumn{4}{c|}{$k=40$}                          & \multicolumn{4}{c}{$k=1$}                           \\
			\multicolumn{2}{c|}{}                                                                                    & \multicolumn{2}{c}{$\delta=0.4$} & \multicolumn{2}{c|}{$\delta=0.6$} & \multicolumn{2}{c}{$\delta=0.4$} & \multicolumn{2}{c}{$\delta=0.6$} \\
			\multicolumn{2}{c|}{}                                                                                    & \texttt{BABS}           & \texttt{Inspect}       & \texttt{BABS}         & \texttt{Inspect}        & \texttt{BABS}           & \texttt{Inspect}        & \texttt{BABS}          & \texttt{Inspect}        \\ \hline
			\multirow{7}{*}{\begin{tabular}[c]{@{}c@{}}Estimated\\ number\\ of\\ change\\ points\end{tabular}}  & 0 & 0           & 0          & 0          & 0          & 0           & 0          & 0          & 0          \\
			& 1 & 16          & 0          & 0          & 0          & 0           & 0          & 0          & 0          \\
			& 2 & 329         & 426        & 194        & 182        & 14          & 466        & 0          & 172        \\
			& \textbf{3} & \textbf{121}         & \textbf{71}         & \textbf{206}        & \textbf{295}        & \textbf{399}         & \textbf{30}         & \textbf{438}        & \textbf{305}        \\
			& 4 & 28          & 3          & 83         & 17         & 78          & 4          & 56         & 19         \\
			& 5 & 6           & 0          & 16         & 6          & 7           & 0          & 6          & 4          \\
			& 6 & 0           & 0          & 1          & 0          & 2           & 0          & 0          & 0          \\ \hline
			\multicolumn{2}{c|}{Sum}                                                                                 & 500         & 500        & 500        & 500        & 500         & 500        & 500        & 500        \\ \hline
			\multicolumn{2}{c|}{ARI}                                                                                 & 0.591       & 0.711      & 0.683      & 0.854      & 0.941       & 0.704      & 0.975      & 0.883     \\
		\end{tabular}}
	\end{table}

	\section{Real data applications}\label{sec:real_app}

	\subsection{Array CGH data}

	The microarray dataset \texttt{aCGH} from the \texttt{ecp} R package in \cite{james2013ecp} consists of $p=43$ individuals with bladder tumors. There are $n=2215$ log-intensity-ratio fluorescent measurements of DNA segments that share almost identical change points because the individuals have the same disease. We set $B=1000, \alpha=0.05, \us = 60$. Our BABS  finds 27 change points in the copy-number that are marked as red vertical dashed lines in Figure~\ref{fig:binseg_real}. Compared to \texttt{Inspect}, which identifies 254 change points by using the default threshold, our discovery is more reasonable and stable under the existence of outliers (e.g.,\ the segment (1724,1836) or (1965, 2044)).

	\begin{figure}
		\centering
		\includegraphics[trim = 0 25 0 55, clip, width = 0.83\textwidth]{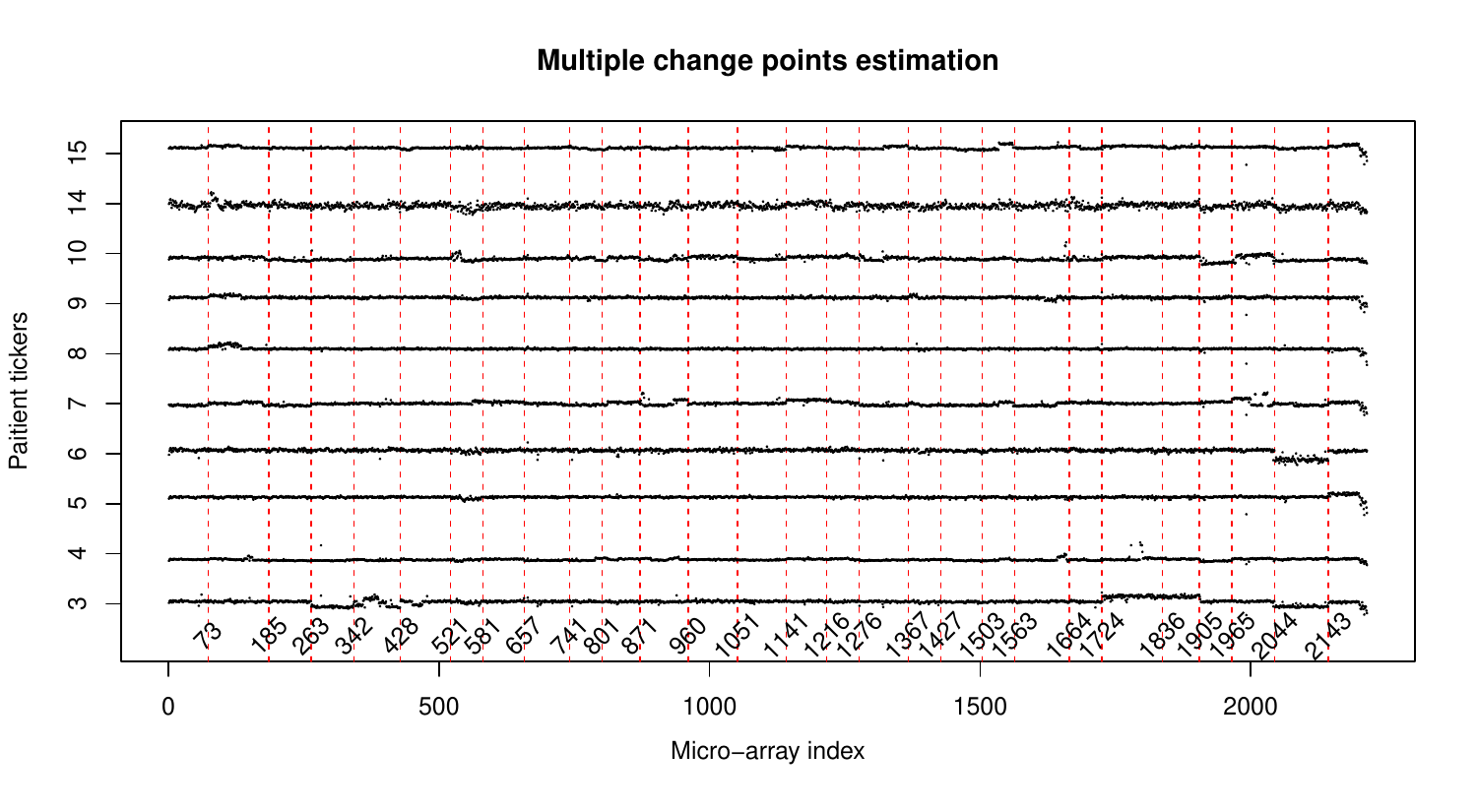}
		\caption{Change point estimation in the log-intensity-ratio fluorescent measurements of aCGH data (the first 10 patients displayed) estimated by BABS using whole \texttt{aCGH} data. Parameters: $n=2215, p=43, B=1000, \alpha=0.05, \us = 60$. }
		\label{fig:binseg_real}
	\end{figure}
	
	\subsection{Stock return data}
	
	We run our block bootstrap CUSUM test and location estimators to stock return data that is available on \url{https://finance.yahoo.com}. The dataset (read through the R package \texttt{BatchGetSymbols}) contains daily closing prices of $p=440$ stocks from the S\&P500 index during the trading days between August 27, 2007 to August 24, 2009 ($n=503$ time points). The daily closing prices are transformed to log scales due to their multiplicative nature. For this dataset, the CUSUM test statistic $T_n = 38.699$. We set $\us = \lfloor 0.05n \rfloor = 25$, bootstrap repeats $B = 200$ and the block sizes $M = 1, 2, 5, 10$ for our block bootstrap calibration. Table~\ref{fig: Realdata:bootstrap_quantiles} shows the (conditional) quantiles of our block bootstrap CUSUM test statistic. Compared with the critical value 8.861 corresponding to the 99\%-quantile for $M=10$, we reject $H_0$. In addition, $\hat{m}_{1/2}=265$ corresponds to September 12, 2008, the last trading day before Lehman Brothers Holdings Inc.\ declared bankruptcy on September 15, 2008. Figure~\ref{fig: Realdata_finance} in the SM Section~\ref{subsec:additional_tabs_figs} plots the top 5 stocks in this financial crash. 
	
	\begin{table}
		\caption{\label{fig: Realdata:bootstrap_quantiles} Quantiles of bootstrapped CUSUM test statistic in the stock return data on log-scale for different block sizes.}		
		\fbox{
		\begin{tabular}{c|cccc}
			& $M=1$    & $M=2$    & $M=5$   & $M=10$   \\
			\hline
			$q_{0.90}$ & 2.350 & 3.380 & 5.219 & 6.981 \\
			$q_{0.95}$ & 2.470 & 3.860 & 5.464 & 7.327 \\
			$q_{0.99}$ & 2.782 & 4.580 & 5.746 & 8.861 \\
		\end{tabular} }
	\end{table}
	
	A binary segmentation extension based on the block bootstrap CUSUM test is considered as well. We implement this extension with  $\hat{m}_{1/2}$ and $\hat{m}_0$ separately, whose estimations are shown in Figure~\ref{fig:Realdata_finance_mcp} on top and bottom, respectively. We set $\us = 25$, $B=400$, $M=5$ and $\alpha = 0.05$ for both scenarios.
	Overall, the two estimators share common time points on detection when mean-shift signals are large enough. There are seven overlapping change points identified by both algorithms, and the one using $\hat{m}_{1/2}$ additionally locates June 12 and July 21 in 2008 but misses October 17, 2008. That is, in the interval between April 21 and September 12 in 2008, $\hat{m}_{1/2}$ is sensitive to change points on the boundary points  (i.e., June 12 and July 21, 2008). However, $\hat{m}_0$ is sensitive to the middle change point, namely October 17, 2008, in the interval between September 12 and November 26 in 2008. This exactly reflects our observation in the simulation. 
	
	\begin{figure}
		\centering
		\includegraphics[trim = 0 24 0 55, clip, height=0.26\textheight]{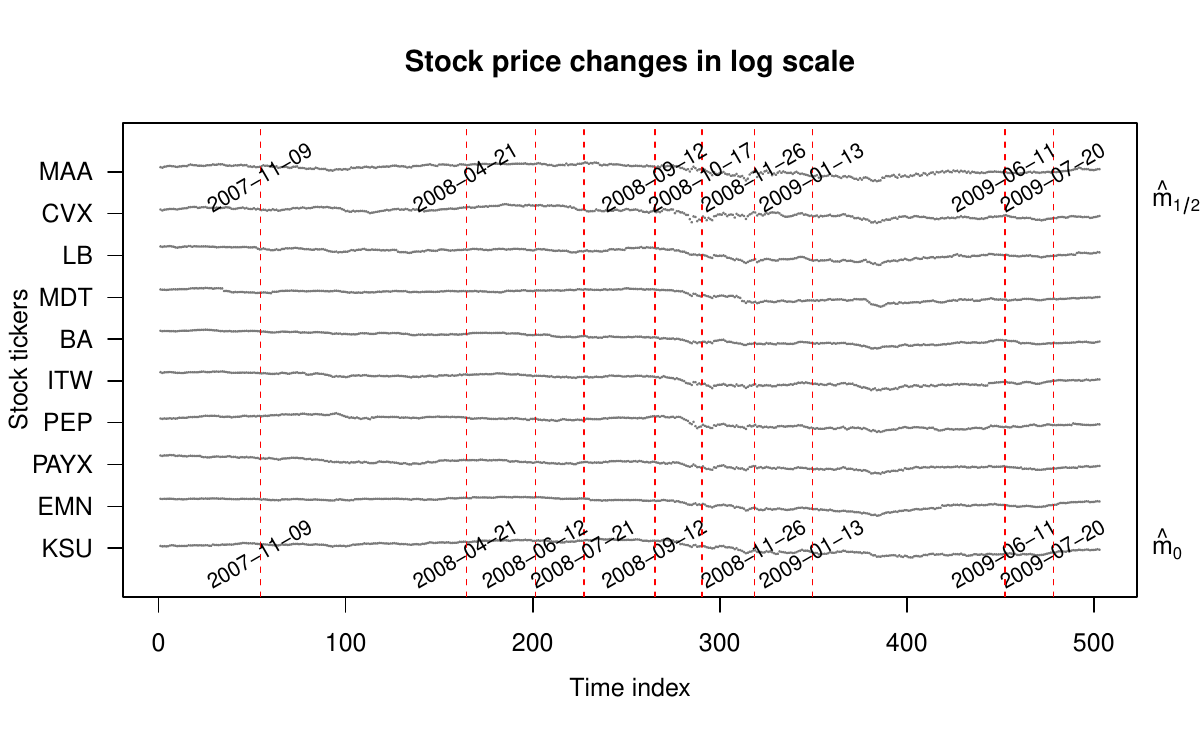}
		\vspace{-0.3cm}
		\caption{Change points detected by extension of BABS using the block CUSUM bootstrap test and two different location estimators $\hat{m}_{1/2}$ and $\hat{m}_{0}$ in the stock return data on log-scale. Parameters: $\us = 25$, $B=400$, $M=5$ and $\alpha = 0.05$.}
		\label{fig:Realdata_finance_mcp}
	\end{figure}
	
	We would like to make two comments for this example. First, the estimator $\hat{m}_{1/2}$ is compatible with our test statistic $T_n$ using stationary weight $\theta = 1/2$, while $\hat{m}_0$ is also a reasonable choice since stock prices are likely to be integrated. Therefore, we evaluate both of them in this stock prices dataset.
	Second, neither of the two algorithms identifies change point between January 13, 2009 and June 11, 2009, even though there seems to be fluctuations in the mean return. One possible reason is that there exists non-synchronous change points (e.g.,\ around time index at 380), which are not estimable. However, this is a common issue in multiple change-point analysis and it is necessary to make some minimal separation or spacing assumptions, cf. \cite{fryzlewicz2014,chofryzlewicz2015,cho2016change,BarigozziChoFryzlewicz2017}. 

		\section*{Acknowledgments}
		The authors would like to thank two anonymous referees, an Associate Editor, and the Co-Editor Simon Wood for their careful and constructive comments that greatly improved the quality of this paper. This work was completed in part with the high-performance computing resource provided by the Illinois Campus Cluster Program at UIUC. X. Chen's research was supported in part by NSF DMS-1404891, NSF CAREER Award DMS-1752614, UIUC Research Board Awards (RB17092,  RB18099), and a Simons Fellowship.  X. Chen acknowledges that part of this work was carried out at the MIT Institute for Data, System, and Society (IDSS).

	\bibliographystyle{agsm}
	{\footnotesize
		\bibliography{hdcp}}
	\newpage

	\appendix
	\begin{center}
		{\Large Supplementary Material} \\ \vspace{0.2in}
		
		This supplementary material contains the discussions, proofs, and additional simulation and real data results of the main paper.
	\end{center}

	\section{Further discussions}
	\label{smsec:discussions}
	
	In this section, we provide more detailed comparisons with other change point detection and estimation methods in literature.
	
	\begin{rmk}[Comparison with \cite{jirak2015} under $H_0$]
	\label{rmk:compare_jirak2015}
		In a related work, \cite{jirak2015} considers the change point tests for high-dimensional time series based on the following version of the CUSUM statistics
		\[
		B_{nj} = { 1 \over \hat\sigma_j \sqrt{n} } \max_{1 \le s \le n} \left| \sum_{i=1}^s X_{ij} - {s \over n} \sum_{i=1}^n X_{ij} \right|, \quad j = 1,\dots,p,
		\]
		where $\hat\sigma_j^2$ is a consistent estimator for the long-run variance of $\{X_{ij}\}_{i \in \mathbb{N}}$. Then $H_0$ is rejected if $\tilde{T}_n = \max_{1 \le j \le p} B_{nj}$ is larger than a critical value. Under $H_0$ and the spatial sparsity conditions (Assumption 2.2 in \cite{jirak2015}), the author establishes a Gumbel limiting distribution for $\tilde{T}_n$ (after suitable normalizations). To improve the rate of convergence, the author also proposes a parametric bootstrap $\tilde{T}_n^Y = \max_{1 \le j \le p} B_{nj}^Y$, where 
		\[
		B_{nj}^Y = {1 \over \sqrt{n}} \max_{1 \le s \le n} \left| \sum_{i=1}^s Y_{ij} - {s \over n} \sum_{i=1}^n Y_{ij} \right|, \quad j = 1,\dots,p,
		\]
		and $\{Y_{ij} : 1 \le i \le n, 1 \le j \le p\}$ is an array of i.i.d.\ $N(0,1)$ random variables. Asymptotic bootstrap validity is derived under the same spatial sparsity assumption as in the Gumbel limit. There is an important difference between $\tilde{T}_n^Y$ in \cite{jirak2015} and our bootstrap test based on $T_n^*$. Note that the conditional covariance matrices of $Z_n^*(s)$ given $X_1^n$ are sample analogs of covariance matrices of $Z_n(s)$. 
		On the contrary, since $\{Y_{ij}\}$ are i.i.d., even when $X_1,\dots,X_n$ are independent observations, the parametric bootstrap $B_{nj}^Y$ does not mimic the general dependence structure among the components $\{X_{ij}\}_{j=1}^p$.
	\qed
	\end{rmk}
	
	\begin{rmk}[Comparison with \cite{jirak2015} under $H_1$]
		\label{rmk:compare_jirak2015_H1}
		\cite{jirak2015} proposed bootstrap testing procedures for change point under the alternative, which are different from the parametric bootstrap under $H_0$ (cf.\ Remark \ref{rmk:compare_jirak2015}). Specifically, the author considers several block versions of the multiplier and empirical bootstraps under $H_1$ in the time series setting. All of the tests under $H_1$ in \cite{jirak2015} require a minimum signal strength condition (cf. Assumption 4.3  therein): 
		\begin{equation}
		\label{eqn:jirak_min_signal_condition}
		\limsup_{n \to \infty} {\log{n} \over K_b \min_{j \in \calS} \delta_{nj}^2} = 0,
		\end{equation}
		where $\calS = \{j \in \{1,\dots,p\} : \delta_{nj} \neq 0\}$ and $K_b$ is the size of blocks. There are two major differences from the bootstraps in \cite{jirak2015}. First, our Gaussian multiplier bootstrap CUSUM test is asymptotically valid and powerful for a change point under both $H_0$ and $H_1$. 
		On the contrary, Section 4 in \cite{jirak2015} designed a series estimators of $t_m$ to adapt to estimation of $\hat\sigma_j^2, j=1, \ldots, p$ under $H_1$, and the estimators must rely on the assumption that each dimension has at most one change point. 
		However, our approach without estimation of $t_m$ can test against more generalized multiple change-point problem where power can still be guaranteed (cf. Lemma~\ref{lem:power_mcp}).
		Second, detection by our bootstrap CUSUM test relies on a lower bound on the signal strength quantified by $|\delta_{n}|_\infty$, which is much weaker than (\ref{eqn:jirak_min_signal_condition}). For example, it is possible that the minimum signal strength $\min_{j \in \calS} \delta_{nj}^2$ decays to zero faster than $(\log{n})/K_b$, while our bootstrap CUSUM test remains valid since it only requires $|\delta_n|_{\infty}$ satisfies a mild lower bound in (\ref{eqn:power_signal_lower_bound}).
		\qed
	\end{rmk}

	\begin{rmk}[Comparison with the sparse projection CUSUM method in~\cite{wangsamworth2017}]
		\label{rmk:compare_sparse_projection}
		\cite{wangsamworth2017} considered a different sparse projection estimator, denoted as $\tilde{m}$, for change point location (even though their estimator is based on the $\theta = 1/2$ normalization). 
		Then, Theorem 1 in \cite{wangsamworth2017} for single change point in our notation reads: if $X_{i} \sim N(\mu_{i}, \sigma^{2} I_{p})$ are independent and the change point signal $\delta_{n}$ satisfies $\|\delta_{n}\|_{0} \leq k$ and $\|\delta_{n}\|_{2} \geq \vartheta$ such that
		\begin{equation}
		\label{eqn:wang-samworth_condition}
		{\sigma \over \vartheta \tau} \sqrt{k \log(p \log{n}) \over n} \lesssim 1,
		\end{equation}
		where $\tau = \min\{t_m, 1-t_m\}$, then with probability tending to one, $\tilde{m}$ obeys
		\begin{equation}
		\label{eqn:wang-samworth_rate}
		n^{-1} |\tilde{m} - m| \lesssim {\sigma^{2} \log\log{n} \over n \vartheta^{2}}.
		\end{equation}
		Here, $\vartheta / \sigma$ can be thought as a signal-to-noise ratio.
		
		Let us compare $\tilde{m}$ in~\cite{wangsamworth2017} with our estimators $\hat{m}_{1/2}$ and $\hat{m}_{0}$ in the highly sparse regime where $k=1$, $\vartheta = n^{-c_{1}}$, and $m = n^{c_{2}}$ for some $c_{1}, c_{2} \geq 0$. For simplicity, let $\sigma = 1$. For such configuration, $t_{\tilde{m}} = \tilde{m} / n$ attains the nearly minimax-optimal rate of convergence $n^{-1+2c_{1}} \log\log{n}$ if \eqref{eqn:wang-samworth_condition} holds, i.e., we need
		\[
		\sqrt{\log(p\log{n})} \lesssim n^{c_{2}-c_{1}-1/2},
		\]
		which necessarily requires that $c_{2} > c_{1} + 1/2$ as $n \to \infty$. It means that the change point location cannot be too close to the boundary $m \gg n^{c_{1}+1/2}$ in order to obtain the optimal rate for \cite{wangsamworth2017}. Thus, if the change point is not close to the boundary, then the sparse projection estimator is nearly optimal at the rate $O(n^{-1+2c_{1}})$ (where $c_{1}$ can be arbitrarily small to zero), and for the boundary scenario, their estimator loses such optimality.
		
		In our Theorem~\ref{thm:rate_location_estimator_nonstationary}, it is shown that $t_{\hat{m}_{0}}$ achieves the nearly optimal rate (up to logarithmic factors) if $m = C n$ for some constant $C \in (0,1)$. On the other hand, we showed in our Theorem 3.4 that our estimator $t_{\hat{m}_{1/2}}$ can deal with the {\it ``more boundary"} case as long as
		\[
		\log^{2}(np) \ll n^{c_{2}/2 - c_{1}}.
		\]
		In addition, there are other side differences between our assumptions and the ones in~\cite{wangsamworth2017}, where the latter are more stringent on the data generation mechanism.
		\qed
	\end{rmk}

	\begin{rmk}[Comparison with other binary segmentation type methods]
	\label{rmk:compare_BS}
	In \cite{chofryzlewicz2015} and \cite{cho2016change}, binary segmentation was also used to extend their single change-point algorithms to the multiple change-point alternative. Theorem 1 in
	\cite{chofryzlewicz2015} discussed the asymptotic consistency of $\Prob(\calS_n) \rightarrow 1$ for SBS algorithm under two cases. With proper tuning on their threshold parameter $\pi_n$, they claimed $\epsilon_n = O(\log^{2+c_1} (n))$ for any $c_1>0$ when $\Theta=1$ for $D_\nu = O(n^\Theta)$ and $\epsilon_n = O(n^{2-2\Theta})$ when $\Theta\in (3/4,1)$, while the rate of BABS for both cases is $\epsilon_n = O(n^{2-2\Theta} \log^4 (np))$, which is nearly the same as \cite{chofryzlewicz2015} up to a logarithmic factor.
	In \cite{cho2016change}, the author proposed a binary segmentation via double CUSUM algorithm (DCBS).
	The DCBS statistic has two CUSUM transforms in both time points with $\theta = 1/2$ in \eqref{eqn:generalized_cusum} and	ordered spatial variables (i.e., cross-sectional features) in a different weighting $\varphi$. Consider the special case where $\varphi = 1/2$ and $\delta_{n,1}^{(l)} = \dots = \delta_{n,k}^{(l)} = \udelta_n \neq 0, l = 1,\cdots,\nu$. Then Theorem 3.3 in \cite{cho2016change} shows that $\Prob(\calS_n) \rightarrow 1$ for DCBS where $\epsilon_n = \udelta_n^{-2} p k^{-2} n^{5-4\Theta} \log^2(n)$ when $\Theta\in (6/7,1]$, while our BABS needs $\epsilon_n = \udelta_n^{-2} n^{2-2\Theta} \log^4(np)$. Under sparse alternative where $k < p^{1/2} n^{3/2-\Theta} \log^{-2} (np)$, BABS always has a better rate. However, we allow $p$ to be as sub-exponentially large in $n$, while Assumption 4 in \cite{chofryzlewicz2015} required $pn^{-\log n} \rightarrow 0$ and Condition (A2) in \cite{cho2016change} fixed $p$ as a polynomial order of $n$. 
	\qed
	\end{rmk}
	
	\begin{rmk}[Extension to change point estimation in covariance matrices]
	Our BABS algorithm can be tailored to other high-dimensional change point estimation problems beyond the mean vectors. For instance, one can consider the estimation problem of multiple change points in covariance matrices for centered and {\it component-wise} sub-Gaussian data $X_{i} \in \R^{p}$ such that $\|X_{ij}\|_{\psi_2} \le B$. Since $\|X_{ij}\|_{\psi_2} = \|X^2_{ij}\|_{\psi_1}$, the triangle inequality of the $\psi_{1}$-norm implies that
	 \[
	 \|X_{ij} X_{ik}\|_{\psi_1} \le \|X^2_{ij}/2\|_{\psi_1} + \|X^2_{ik}/2\|_{\psi_1} = \|X_{ij}\|_{\psi_2}/2 + \|X_{ik}\|_{\psi_2}/2.
	 \]
	Thus, if $X_{i}$ has bounded sub-Gaussian entries, then the vectorized version of the empirical covariance matrix $\xi_i = \mbox{vec}(X_i X_i^T) \in \R^{p^2}$ has all entries satisfying the bounded sub-exponential assumption (C) in Section~\ref{subsec:validity_bootstrap_cp_test}. So Theorem~\ref{thm:binseg_consistency_mcp} (part (i)) implies that with a constant signal strength, $\Prob(\calS_n) \ge 1- O( \nu \varpi_{1,n} )$ for $\epsilon_n = O(D_\nu^{-2} n^{2} \log^4(np))$.

	\cite{WangYuRinaldo2017} studied a similar covariance matrix change points estimation procedure for centered and sub-Gaussian data points $X_i \in \R^p$ such that $\max_{1 \le i \le n} \|X_{i}\|_{\psi_2} \le B$, where $\|X\|_{\psi_2} = \sup_{v \in \mathbb{S}^{p-1}} \|v^{T} X\|_{\psi_{2}}$ and $\mathbb{S}^{p-1}$ is the unit sphere in $\R^{p}$. \cite{WangYuRinaldo2017} proposed a binary segmentation in operator norm (BSOP) by recursively maximizing the operator norm of the matrix-valued CUSUM statistics (without the bootstrap calibration). 
	Theorem 1 in \cite{WangYuRinaldo2017} stated that $\Prob(\calS_n) \ge 1- O(9^p n^{3-cp})$ for $\epsilon_n = O(D_\nu^{-2} n^{5/2} \sqrt{p \log(n)})$ and $p =O(n/\log(n))$. 
	Our BABS improves the BSOP in the following aspects. First, our sub-exponential tail condition is weaker than \cite{WangYuRinaldo2017} in view of $\max_{1 \le j \le p} \|X_{ij}\|_{\psi_{2}} \leq \|X_{i}\|_{\psi_2}$. 
	Second, we allow the dimension $p$ to be sub-exponentially large in the sample size $n$. This is an essential benefit by working with the $\ell^{\infty}$-norm (rather than the operator norm).
	\qed
	\end{rmk}

	\section{Proof of main results in Section~\ref{sec:main_results}}
	\label{sec:proofs}
	In this section, we prove the main theoretical results in Section~\ref{sec:main_results}.
	We first present a useful maximal inequality for weighted partial sums of independent and centered random vectors.
	
	\begin{lem}[Talagrand's inequality for weighted partial sums of independent and centered random vectors]
		\label{lem:maximal_subexponential}
		Let $X_1,\dots,X_n$ be independent and centered random vectors in $\R^p$ and $\{a_{is}\}_{i,s=1}^n$ be an $n \times n$ matrix of real numbers. Define 
		\[
		\begin{gathered}
		W_{n,sj} = \sum_{i=1}^n a_{is} X_{ij}, \quad Z_n = \max_{1 \le s \le n} \max_{1 \le j \le p} |W_{n,sj}|, \\
		M = \max_{1 \le s,i \le n} \max_{1 \le j \le p} |a_{is} X_{ij}|,  \quad \sigma^2 = \max_{1 \le s \le n} \max_{1 \le j \le p} \sum_{i=1}^n a_{is}^2 \E(X_{ij}^2).
		\end{gathered} 
		\]
		\\
		(i) Let $\beta \in (0,1]$ and suppose that $\|X_{ij}\|_{\psi_\beta} < \infty$ for all $i = 1,\dots,n$ and $j=1,\dots,p$. Then, $\forall \eta \in (0,1]$, there exists a constant $C > 0$ depending only on $\beta$ and $\eta$ such that we have for $t > 0$
		\begin{equation}
		\label{eqn:maximal_subexponential}
		\Prob(Z_n \ge (1+\eta) \E[Z_n] + t) \le \exp\left(-{t^2 \over 3 \sigma^2} \right) + 3 \exp\left\{ - \left( {t \over C \|M\|_{\psi_\beta}} \right)^\beta \right\}.
		\end{equation}
		(ii) Let $s \ge 1$ and suppose that $\E|X_{ij}|^s < \infty$ for all $i = 1,\dots,n$ and $j=1,\dots,p$. Then, $\forall \eta \in (0,1]$, there exists a constant $C > 0$ depending only on $s$ and $\eta$ such that we have for $t > 0$
		\begin{equation}
		\label{eqn:maximal_polynomial}
		\Prob(Z_n \ge (1+\eta) \E[Z_n] + t) \le \exp\left(-{t^2 \over 3 \sigma^2} \right) + C {\E[M^s] \over t^s}.
		\end{equation}
	\end{lem}
	
	\subsection{Proof of Theorem \ref{thm:bootstrap_approx_cusum_statistic}}
	
	\begin{proof}[of Theorem \ref{thm:bootstrap_approx_cusum_statistic}]
		Suppose that $H_0$ is true. We may assume $\log^7 (np) \le \us$ for otherwise (\ref{eqn:bootstrap_cusum_subexp}) and (\ref{eqn:bootstrap_cusum_poly}) trivially hold by choosing the constant $C > 0$ large enough therein. For $s = 1,\dots,n-1$, let
		\begin{equation}
		\label{eqn:a_is_cusum_statistic}
		a_{is} = \left\{
		\begin{array}{cc}
		\sqrt{n-s \over n s} & \text{if } 1 \le i \le s \\
		-\sqrt{s \over n (n-s)} & \text{if } s+1 \le i \le n \\
		\end{array} \right. 
		\end{equation}
		and $\va_s = (a_{1s}, \dots, a_{ns})^\top$. Denote $\vX = (X_1, \dots, X_n)$ as the $p \times n$ data matrix and $A = (\va_{\us},\dots,\va_{n-\us})$. Then we can write
		\begin{equation}
		\label{eqn:stacked_Zn}
		Z_n(s) = \sum_{i=1}^n a_{is} X_i = \vX \va_s.
		\end{equation}
		Since $\E[Z_n(s)] = 0$ under $H_0$, without loss of generality, we may assume $\mu_i \equiv 0$. Note that, for any $1 \le s \le s' \le n-1$, we have 
		$$
		\Cov(Z_n(s), Z_n(s')) = \sum_{i=1}^n a_{is} a_{is'} \Sigma = \Sigma \sqrt{s (n-s') \over s' (n-s)}.
		$$
		
		\underline{Step 1: Gaussian approximation for CUSUM statistic.} Let
		$$
		\vZ_n = (Z_n(\us), \dots, Z_n(n-\us)) = \vX (\va_{\us},\dots,\va_{n-\us}) = \vX A
		$$
		be the CUSUM transformation of $\vX$. Let $\vec(\vZ_n)$ be the column stacked version of $\vZ_n$, i.e. $\vec(\vZ_n) = (Z_n(\us)^\top, \dots, Z_n(n-\us)^\top)^\top$ is the $[(n-2\us+1)p] \times 1$ vector associated with $\vZ_n$. Then we can write
		$$
		\vec(\vZ_n) = \vec(\vX A) = (A^\top \otimes I_p) \vec(\vX),
		$$
		where $\otimes$ is the Kronecker product of two matrices. Since $\E[\vec(\vX)] = 0$ and $\Cov(\vec(\vX)) = \Gamma$, where $\Gamma$ is the block diagonal matrix of size $(pn) \times (pn)$ with $\Sigma$ being the diagonal sub-matrices, we have $\E[\vec(\vZ_n)] = 0$ and $\Cov(\vec(\vZ_n)) = (A^\top \otimes I_p) \Gamma (A \otimes I_p).$ Let
		\begin{equation}
		\label{eqn:gaussian_analog_cusum_statistic}
		Y_n \sim N(0, (A^\top \otimes I_p) \Gamma (A \otimes I_p)).
		\end{equation}
		be a joint mean-zero Gaussian random vector in $\R^{(n-2\us+1)p}$ with the same covariance matrix as $\vec(\vZ_n)$. Denote $\bar{Y}_n = |Y_n|_\infty$. Since $\Cov(Z_n(s)) = \Sigma$ for all $s = 1,\dots,n-1$, it follows from (A) that $\Var(Z_{nj}(s)) \ge \ub$ for all $1 \le j \le p$ and $\us \le s \le n-\us$. Since $1 \le \us \le n/2$, we have 
		\begin{eqnarray*}
			\sum_{i=1}^n |a_{is}|^3 &=& \sqrt{n \over s(n-s)} - {2 \over n} \sqrt{s(n-s) \over n} \le \sqrt{n \over \us (n-\us)} \le \sqrt{2 \over \us}, \\
			\sum_{i=1}^n |a_{is}|^4 &=& {n \over s(n-s)} - {3 \over n} \le {n \over \us (n-\us)} \le {2 \over \us}.
		\end{eqnarray*}
		Set $B_n = (2 \bar{b}^2 \us^{-1} n)^{1/2}$. By assumption (B), we have for $\ell = 1,2$,
		$$
		n^{-1} \sum_{i=1}^n |n^{1/2} a_{is}|^{2+\ell} \E|X_{ij}|^{2+\ell} \le B_n^\ell.
		$$
		Note that $\us^{1/2}  |a_{is}| \le 1$ for all $s = \us, \dots, n-\us$. 
		
		\underline{\it{Part (i)}}. If (C) holds, then we have 
		$$
		\E\left[ \exp\left( n^{1/2} |a_{is}| |X_{ij}| / B_n \right) \right] \le \E\left[ \exp \left( \us^{1/2}  |a_{is}| |X_{ij}| / \bar{b} \right) \right] \le 2.
		$$
		By Proposition 2.1 in \cite{cck2016a}, there exists a constant $C_1 > 0$ depending only on $\ub$ and $\bar{b}$ such that 
		\begin{equation}
		\label{eqn:GA_subexp}
		\rho(T_n, \bar{Y}_n) \le C(\ub) \left( {B_n^2 \log^7(pn) \over n} \right)^{1/6} \le C_1 \varpi_{1,n}.
		\end{equation}
		
		\underline{\it{Part (ii)}}. If (D) holds, then we have 
		$$
		\E\left\{ \max_{1 \le j \le p} \max_{\us \le s \le n-\us} (|n^{1/2} a_{is}| |X_{ij}| / B_n)^q \right\} \le \left[ \us^{1/2} (\max_{\us \le s \le n-\us} |a_{is}|) \right]^q \E\left[ \max_{1 \le j \le p} (|X_{ij}| / \bar{b})^q \right] \le 1
		$$
		for all $i = 1,\dots,n$. By Proposition 2.1 \cite{cck2016a}, there exists a constant $C_1 > 0$ depending only on $\ub,\bar{b},q$ such that 
		\begin{eqnarray}
		\nonumber
		\rho(T_n, \bar{Y}_n) &\le& C(\ub,q) \left\{ \left( {B_n^2 \log^7(pn) \over n} \right)^{1/6} + \left( {B_n^2 \log^3(pn) \over n^{1-2/q}} \right)^{1/3} \right\} \\
		\label{eqn:GA_poly}
		&\le& C_1 \{ \varpi_{1,n} + \varpi_{2,n} \}.
		\end{eqnarray}
		
		\underline{Step 2: Gaussian comparison for $\bar{Y}_n$ and bootstrap CUSUM statistic $T_n^*$.} Let
		\begin{equation}
		\label{eqn:cusum_sample_cov_mat}
		\begin{gathered}
		\hat{S}_{n,s}^- = {1 \over s} \sum_{i=1}^s (X_i - \bar{X}_s^-) (X_i - \bar{X}_s^-)^\top, \\
		\hat{S}_{n,s}^+ = {1 \over n-s} \sum_{i=s+1}^n (X_i - \bar{X}_s^+) (X_i - \bar{X}_s^+)^\top,
		\end{gathered}
		\end{equation}
		be the sample covariance matrices based on the left and right observations at $s$, respectively. Then
		$$
		Z_n^*(s) | X_1^n \sim N \left(0, {n-s \over n} \hat{S}_{n,s}^- + {s \over n} \hat{S}_{n,s}^+ \right).
		$$
		Let
		$$
		\va^*_{is} = \left\{
		\begin{array}{cc}
		\sqrt{n-s \over n s} (X_i - \bar{X}_s^-) & \text{if } 1 \le i \le s \\
		-\sqrt{s \over n (n-s)} (X_i - \bar{X}_s^+) & \text{if } s+1 \le i \le n \\
		\end{array} \right. 
		$$
		and $A_s^* = (\va_{1s}^*, \dots, \va_{ns}^*)$. Let $\ve = (e_1,\dots,e_n)^\top$. Then we can write
		$$
		Z_n^*(s) = \sum_{i=1}^n \va_{is}^* e_i = A_s^* \ve.
		$$
		Let
		$$
		Z_n^* = 
		\left( \begin{array}{c}
		Z_n^*(\us) \\
		\vdots \\
		Z_n^*(n-\us) \\
		\end{array} \right) 
		= 
		\left( \begin{array}{c}
		A_{\us}^* \ve \\
		\vdots \\
		A_{n-\us}^* \ve \\
		\end{array} \right) 
		=
		\left( \begin{array}{c}
		A_{\us}^* \\
		\vdots \\
		A_{n-\us}^* \\
		\end{array} \right) \ve
		:=
		{\mbf A}^* \ve.
		$$
		Since $\ve \sim N(0, \Id_n)$, it follows that $Z_n^* | X_1^n \sim N(0, {\mbf A}^* {{\mbf A}^*}^\top).$ Next, we compute an explicit expression for the covariance matrix of $Z_n^*$ given $X_1^n$. Some routine algebra show that for any $\us \le s \le s' \le n-\us$, 
		\begin{eqnarray*}
			&& \Cov(Z_n^*(s), Z_n^*(s') | X_1^n) = \Sigma \sqrt{s (n-s') \over s' (n-s)} + {R_1 \over n} \sqrt{(n-s) (n-s') \over s s'} + {R_2 \over n} \sqrt{s (n-s') \over s' (n-s)}\\
			&& \qquad  + {R_3 \over n} \sqrt{s (n-s') \over s' (n-s)} + {R_4 \over n} \sqrt{s s'  \over (n-s) (n-s')} - {R_5 \over n} \sqrt{s (n-s') \over s' (n-s)},
		\end{eqnarray*}
		where $R_1 = \sum_{i=1}^s [(X_i - \bar{X}_s^-) (X_i - \bar{X}_{s'}^-)^\top - \Sigma]$, $R_2 = \sum_{i=1}^s [(X_i - \bar{X}_s^+) (X_i - \bar{X}_{s'}^-)^\top - \Sigma]$, $R_3 = \sum_{i=s'+1}^n  [(X_i - \bar{X}_s^+) (X_i - \bar{X}_{s'}^-)^\top - \Sigma]$, $R_4 = \sum_{i=s'+1}^n  [(X_i - \bar{X}_s^+) (X_i - \bar{X}_{s'}^+)^\top - \Sigma]$, and $R_5 = \sum_{i=1}^n  [(X_i - \bar{X}_s^+) (X_i - \bar{X}_{s'}^-)^\top - \Sigma]$. Let 
		\begin{equation}
		\label{eqn:Delta_hat_defns}
		\begin{gathered}
		\hat\Delta_1 = \max_{\us \le s \le n-\us} \left| {1 \over s} \sum_{i=1}^s (X_i X_i^\top - \Sigma) \right|_\infty, \quad \hat\Delta_3 = \max_{\us \le s \le n-\us} |\bar{X}_s^-|_\infty, \\
		\hat\Delta_2 = \max_{\us \le s \le n-\us} \left| {1 \over n-s} \sum_{i=s+1}^n (X_i X_i^\top - \Sigma) \right|_\infty, \quad \hat\Delta_4 =  \max_{\us \le s \le n-\us} |\bar{X}_s^+|_\infty.
		\end{gathered}
		\end{equation} 
		Then there exists a universal constant $K_1 > 0$ such that 
		\begin{equation*}
		|\Cov(Z_n^* | X_1^n) - \Cov(Y_n)|_\infty \le K_1 \hat\Delta,
		\end{equation*}
		where 
		\begin{equation}
		\label{eqn:Delta_hat}
		\hat\Delta = \max\{ \hat\Delta_1, \hat\Delta_2 \} + \max\{\hat\Delta_3^2, \hat\Delta_4^2\}.
		\end{equation}
		Let $\bar\Delta$ be a positive real number and $E = \{ \hat\Delta \le \bar\Delta \}$. By Lemma C.1 in \cite{chen2017+}, there exists a constant $C_2 > 0$ depending only on $\ub$ such that on the event $E$, we have 
		\begin{equation*}
		\rho^*(\bar{Y}_n, \bar{Z}^*_n) \le C_2 \bar\Delta^{1/3} \log^{2/3}(np).
		\end{equation*}
		
		\underline{\it{Part (i)}}. If (C) holds, then we choose 
		\begin{equation}
		\label{eqn:Delta_bound_subexp}
		\bar\Delta = C_3  \us^{-1/2} \log^{3/2}(np)
		\end{equation}
		for some large enough constant $C_3 := C_3(\ub, \bar{b}, K) > 0$. By Lemma \ref{lem:hatDelta}, we have $\Prob(E) \ge 1-\gamma$. Then there exists a constant $C_4 := C_4(\ub, \bar{b}, K) > 0$ such that 
		\begin{equation}
		\label{eqn:GC_subexp}
		\rho^*(\bar{Y}_n, \bar{Z}^*_n) \le C_4 \varpi_{1,n}
		\end{equation}
		holds with probability at least $1-\gamma$. Combining (\ref{eqn:GA_subexp}) and (\ref{eqn:GC_subexp}), we obtain (\ref{eqn:bootstrap_cusum_subexp}).
		
		\underline{\it{Part (ii)}}. If (D) holds, then we choose 
		\begin{equation}
		\label{eqn:Delta_bound_poly}
		\bar\Delta = C_5 \{ \us^{-1/2} \log^{1/2}(np) + \gamma^{-2/q} \us^{-1} n^{2/q} \log(np) \}
		\end{equation}
		for some large enough constant $C_5 := C_5(\ub, \bar{b}, K, q) > 0$. By Lemma \ref{lem:hatDelta}, we have $\Prob(E) \ge 1-\gamma$. Then there exists a constant $C_6 := C_6(\ub, \bar{b}, K, q) > 0$ such that 
		\begin{equation}
		\label{eqn:GC_poly}
		\rho^*(\bar{Y}_n, \bar{Z}^*_n) \le C_6 \{ \varpi_{1,n} + \varpi_{2,n} \}.
		\end{equation}
		holds with probability at least $1-\gamma$. Combining (\ref{eqn:GA_poly}) and (\ref{eqn:GC_poly}), we obtain (\ref{eqn:bootstrap_cusum_poly}).
	\end{proof}

	\begin{lem}[Bound on $\max_{1 \le i \le 4} \hat\Delta_i$]
		\label{lem:hatDelta}
		Suppose $H_0$ is true and assume (A) and (B) hold. Let $\gamma \in (0, e^{-1})$ and suppose that $\log(\gamma^{-1}) \le K \log(pn)$ for some constant $K > 0$. Let $\hat\Delta_i,i=1,\dots,4$ be defined in (\ref{eqn:Delta_hat_defns}). \\
		(i) If (C) holds and $\log^5(np) \le \us$, then there exists a constant $C > 0$ depending only on $\bar{b}, K$ such that 
		$$
		\Prob(\max_{1 \le i \le 4} \hat\Delta_i \le C  \us^{-1/2} \log^{3/2}(np)) \ge 1 -\gamma.
		$$
		(ii) If (D) holds with $q \ge 4$, then there exists a constant $C > 0$ depending only on $\bar{b}, K, q$ such that 
		$$
		\Prob(\max_{1 \le i \le 4} \hat\Delta_i \le C  \{ \us^{-1/2} \log^{3/2}(np) + \gamma^{-2/q} \us^{-1} n^{2/q} \log(np) \} ) \ge 1 -\gamma.
		$$
	\end{lem}

	\begin{proof}[of Corollary \ref{cor:validitiy_bootstrap_cusum_test}]
		Under $H_0$, we write $\Prob = \Prob_0$. Let $Y_n$ be a joint Gaussian random vector defined in (\ref{eqn:gaussian_analog_cusum_statistic}) and $\bar{Y}_n = |Y_n|_\infty$. Let $\rho_\ominus(\alpha) = \Prob( \{T_n \le q_{T_{n}^{*} \mid X_{1}^{n}}(\alpha) \} \ominus \{T_n \le q_{\bar{Y}_n}(\alpha) \} )$ and $A \ominus B = (A\setminus B) \cup (B\setminus A)$ be the symmetric difference of two events $A$ and $B$. Note that 
		\begin{eqnarray*}
			|\Prob(T_n \le q_{T_{n}^{*} \mid X_{1}^{n}}(\alpha)) - \alpha| &\le&  |\Prob(T_n \le q_{T_{n}^{*} \mid X_{1}^{n}}(\alpha)) - \Prob(T_n \le q_{\bar{Y}_n}(\alpha))| + \rho(T_n, \bar{Y}_n) \\
			&\le&  \Prob( \{T_n \le q_{T_{n}^{*} \mid X_{1}^{n}}(\alpha) \} \ominus \{T_n \le q_{\bar{Y}_n}(\alpha) \} ) + \rho(T_n, \bar{Y}_n) \\
			&=& \rho_\ominus(\alpha) + \rho(T_n, \bar{Y}_n).
		\end{eqnarray*}
		By Lemma C.3 in \cite{chen2017+}, there exists a constant $C > 0$ only depending on $\ub$ such that for any real number $\bar\Delta > 0$, we have 
		$$
		\rho_\ominus(\alpha) \le 2 [ \rho(T_n, \bar{Y}_n) + C \bar\Delta^{1/3} \log^{2/3}(np) + \Prob(\hat\Delta > \bar\Delta) ],
		$$
		where $\hat\Delta$ is defined in (\ref{eqn:Delta_hat}).
		
		\underline{\it{Part (i)}}. Assume (C) and choose $\bar\Delta$ in (\ref{eqn:Delta_bound_subexp}). By Lemma \ref{lem:hatDelta}, we have $\Prob(\hat\Delta > \bar\Delta) < \gamma/2$. Combining with (\ref{eqn:GA_subexp}), we get (\ref{eqn:size_validity_bootstrap_cusum_subexp}). Uniform convergence of $\Prob(T_n \le q_{T_{n}^{*} \mid X_{1}^{n}}(\alpha))$ to $\alpha$ follows by choosing $\gamma = n^{-1}$.
		
		\underline{\it{Part (ii)}}. Assume (D) and choose $\bar\Delta$ in (\ref{eqn:Delta_bound_poly}). By Lemma \ref{lem:hatDelta}, we have $\Prob(\hat\Delta > \bar\Delta) < \gamma/2$. Combining with (\ref{eqn:GA_poly}), we get (\ref{eqn:size_validity_bootstrap_cusum_poly}). Uniform convergence of $\Prob(T_n \le q_{T_{n}^{*} \mid X_{1}^{n}}(\alpha))$ to $\alpha$ follows by choosing $\gamma = [\log(np)]^{-q\varepsilon/2}$.
	\end{proof}
	
	\subsection{Proof of Theorem \ref{thm:power_bootstrap_cusum_test}}
	
	\begin{proof}[of Theorem \ref{thm:power_bootstrap_cusum_test}]
		Under $H_1$, without loss of generality, we may assume $\mu=0$. Then  
		\begin{equation}
		\label{eqn:power_Y_mean_shift}
		\xi_{i} = \left\{
		\begin{array}{ll}
		X_i, & \text{if } 1 \le i \le m \\
		X_i-\delta_n, & \text{if } m+1 \le i \le n \\
		\end{array} \right. .
		\end{equation}
		Observe that the CUSUM statistic (computed on $X_1,\dots,X_n$) in (\ref{eqn:cusum_mean_Rp}) can be written as 
		$$
		Z_n(s) = Z_n^\xi(s) + \Delta_s,
		$$
		where
		$$ 
		Z_n^\xi(s) = \sqrt{s(n-s) \over n} \left\{{1\over s} \sum_{i=1}^s \xi_i - {1 \over n-s} \sum_{i=s+1}^n \xi_i \right\}
		$$
		and \vspace{-0.05in}
		\begin{equation}
		\label{eqn:Delta_s}
		\vspace{-0.05in}
		\Delta_s = \left\{
		\begin{array}{ll}
		-\sqrt{s \over n(n-s)} (n-m)\delta_n, & \text{if } 1 \le s \le m \\
		-\sqrt{n-s \over ns} m \delta_n, & \text{if } m+1 \le s \le n \\
		\end{array} \right. 
		\end{equation}
		is the mean shift. Note that $|\Delta_s|_\infty$ reaches its maximum at $s=m$, i.e.,
		$$
		\max_{\us \le s \le n - \us} |\Delta_s|_\infty = \max_{1\le s \le n} |\Delta_s|_\infty = |\Delta_m|_\infty = \sqrt{m(n-m) \over n} |\delta_n|_\infty := \tilde\Delta.
		$$
		Let $ \tilde{T}_n = \max_{\us \le s \le n - \us}|Z_n^\xi(s)|_\infty$. Then we have 
		$$
		T_n = \max_{\us \le s \le n - \us}|Z_n(s)|_\infty
		= \max_{\us \le s \le n - \us}|Z_n^\xi(s) + \Delta_s|_\infty
		\ge \tilde\Delta -\tilde{T}_n,
		$$
		from which it follows that the type II error of our bootstrap test obeys
		\begin{equation}
		\label{eqn:type_II_error_lower_bound}
		\text{Type II error} = \Prob_1 (T_n \le q_{T_{n}^{*} \mid X_{1}^{n}}(1-\alpha)) \le \Prob_1 (\tilde{T}_n \ge \tilde \Delta - q_{T_{n}^{*} \mid X_{1}^{n}}(1-\alpha)).
		\end{equation}
		Let $\beta_n \in (0,1)$ and 
		$
		\hat\Delta := \hat\Delta(X_1^n) = q_{T_{n}^{*} \mid X_{1}^{n}}(1-\alpha) + q_{\tilde{T}_n}(1-\beta_n).
		$
		Clearly, $\hat\Delta$ is a random quantity that is $\sigma(X_1,\dots,X_n)$-measurable. Then, 
		\begin{eqnarray}
		\nonumber
		\Prob_1 (\tilde{T_n} \ge \tilde \Delta - q_{T_{n}^{*} \mid X_{1}^{n}}(1-\alpha)) &\le& \Prob_1 (\tilde{T}_n \ge \hat\Delta - q_{T_{n}^{*} \mid X_{1}^{n}}(1-\alpha)) + \Prob_1 (\hat\Delta > \tilde\Delta) \\
		\label{eqn:type_II_error_bound}
		&\le&  \beta_n + \Prob_1 (\hat\Delta > \tilde\Delta).
		\end{eqnarray}
		Observe that the distribution of $\tilde{T}_n$ does not depend on $\delta_n$. Hence, $\tilde{T}_n$ has the same distributions as $T_n$ under $H_0$. 
		
		\underline{\it{Part (i).}} Assume (C). Let $Y_n \sim N(0, (A^\top \otimes I_p) \Gamma (A \otimes I_p))$ be a joint mean-zero Gaussian random vector in $\R^{(n-2\us+1)p}$, where $A$ is defined in (\ref{eqn:a_is_cusum_statistic}). Denote $\bar{Y}_n = |Y_n|_\infty$.  By the Gaussian approximation (\ref{eqn:GA_subexp}), there exists a constant $C_1 := C_1(\ub, \bar{b}, K) > 0$ such that 
		$$
		\Prob(\tilde{T}_n > t) \le \Prob(\bar{Y}_n > t) + C_1 \varpi_{1,n}
		$$
		holds for all $t \in \R$. By Lemma 2.2.2 in \cite{vandervaartwellner1996}, $\|\bar{Y}_n\|_{\psi_2} \le C_2 \log^{1/2}(np)$, where $C_2 > 0$ is a constant depending only on $\bar{b}$. So we have $\forall t > 0$,
		$$
		\Prob(\bar{Y}_n > t) \le 2 \exp[-(t /\|\bar{Y}_n\|_{\psi_2})^2] \le 2 \exp[- C_2^{-2} \log^{-1}(np) t^2 ].
		$$
		Choosing $t = C_3 [\log(\zeta^{-1}) \log(np)]^{1/2}$ for some large enough constant $C_3 > 0$, we get $\Prob(\bar{Y}_n > t) \le 2 \zeta^{C_3^2/C_2^2} \le 2 \zeta$. Now, take $\beta_n = C_1 \varpi_{1,n} + 2 \zeta$.
		Since $q_{\tilde{T}_n}(1-\beta_n) = \inf\{t \in \R : \Prob(\tilde{T}_n > t) < \beta_n\}$, we deduce that
		\begin{equation}
		\label{eqn:type_II_error_tilde_T_n}
		q_{\tilde{T}_n}(1-\beta_n) \le C_3  [\log(\zeta^{-1}) \log(np)]^{1/2}.
		\end{equation}
		Next, we deal with $q_{T_{n}^{*} \mid X_{1}^{n}}(1-\alpha)$. Recall that $\hat{S}^-_{n,s}$ and $\hat{S}^+_{n,s}$ are defined in (\ref{eqn:cusum_sample_cov_mat}). By the Bonferroni inequality, we have 
		$$
		\Prob_{1} (\max_{\us \le s \le n-\us} |Z_n^*(s)|_\infty > t | X_1^n) \le 2np [ 1 - \Phi(t / \bar\psi) ],
		$$
		where
		$$
		\bar\psi^2 = \max_{\us \le s \le n-\us} \max_{1 \le j \le p} \left\{ {n-s \over n} \hat{S}^-_{n,s,jj} + {s \over n} \hat{S}^+_{n,s,jj} \right\},
		$$
		and $\hat{S}^-_{n,s,jj}$, $\hat{S}^+_{n,s,jj}$ are the $(j,j)$-th diagonal entry of $\hat{S}^-_{n,s}$, $\hat{S}^+_{n,s}$ respectively. Then it follows that there exists a universal constant $K_1 > 0$ such that 
		\begin{equation*}
		q_{T_{n}^{*} \mid X_{1}^{n}}(1-\alpha) \le \bar\psi \Phi^{-1} (1-\alpha/(2np)).
		\end{equation*}
		Next, we bound the quantiles $t_{n,\alpha} := \Phi^{-1} (1-\alpha/(2np))$. Recall that $n \ge 4$, $p \ge 3$, and $\alpha \in (0,1)$. Since $\Phi^{-1}(\cdot)$ is a strictly increasing function, we have $t_{n,\alpha} \ge  \Phi^{-1}(23/24) > 1.73$. By the standard Gaussian tail bound $1-\Phi(x) < \phi(x) / x$ for all $x > 0$, we deduce that 
		\[
		{\alpha \over 2np} = 1 - \Phi(t_{n,\alpha}) < {\phi(t_{n,\alpha}) \over t_{n,\alpha}} < 0.25 \exp\left(-{t_{n,\alpha}^2 \over 2}\right).
		\]
		Therefore, $t_{n,\alpha} < \sqrt{2 \log(np/(2\alpha))}$ and 
		\begin{equation}
		\label{eqn:type_II_error_T_n*}
		q_{T_{n}^{*} \mid X_{1}^{n}}(1-\alpha) \le K_1 \bar\psi \log^{1/2}(np/\alpha)
		\end{equation}
		for some universal constant $K_1 > 0$. Define 
		\begin{equation*}
		\begin{gathered}
		\hat{S}_{n,s}^{\xi,-} = {1 \over s} \sum_{i=1}^s (\xi_i - \bar{\xi}_s^-) (\xi_i - \bar{\xi}_s^-)^\top, \\
		\hat{S}_{n,s}^{\xi,+} = {1 \over n-s} \sum_{i=s+1}^n (\xi_i - \bar{\xi}_s^+) (\xi_i - \bar{\xi}_s^+)^\top,
		\end{gathered}
		\end{equation*}
		where $\bar{\xi}_s^- = s^{-1} \sum_{i=1}^s \xi_i$ and $\bar{\xi}_s^+ = (n-s)^{-1} \sum_{i=s+1}^n \xi_i$. By Lemma \ref{lem:bound_bar_psi}, we have 
		$$
		\bar\psi^2 \le 2 \max_{\us \le s \le n-\us} \max_{1 \le j \le p} \left\{ {n-s \over n} \hat{S}^{\xi,-}_{n,s,jj} + {s \over n} \hat{S}^{\xi,+}_{n,s,jj} \right\} + 4 |\delta_n|_\infty^2.
		$$
		Since 
		$$
		\max_{\us \le s \le n-\us} \max_{1 \le j \le p} {n-s \over n} \hat{S}^{\xi,-}_{n,s,jj} \le \max_{1 \le j \le p} \Sigma_{jj} + \max_{\us \le s \le n-\us} \max_{1 \le j \le p} \left| {1 \over s} \sum_{i=1}^s (\xi_{ij}^2  - \Sigma_{jj}) \right| + \max_{\us \le s \le n-\us} \max_{1 \le j \le p} \left| \bar{\xi}_{sj}^- \right|^2,
		$$
		it follows from Lemma \ref{lem:hatDelta} that there exists a constant $C_4 > 0$ depending only on $\bar{b}, K$ such that 
		$$
		\Prob\left( \max_{\us \le s \le n-\us} \max_{1 \le j \le p} \left| \sum_{i=1}^s {1 \over s} (\xi_{ij}^2  - \Sigma_{jj}) \right| \ge C_4 \us^{-1/2} \log^{3/2}(np) \right) \le \gamma / 4
		$$
		and the same probability bound holds for $ \max_{\us \le s \le n-\us} \max_{1 \le j \le p} | \bar{\xi}_{sj}^-|^2$. Combining with (\ref{eqn:type_II_error_T_n*}), we deduce that there exists a constant $C_5 > 0$ depending only on $\bar{b}, K$ such that 
		$$
		\Prob_{1} \left( q_{T_{n}^{*} \mid X_{1}^{n}}(1-\alpha) \ge C_5 \max\{|\delta_n|_\infty, 1\} \log^{1/2}(np/\alpha) \right) \le \gamma.
		$$
		Then (\ref{eqn:power_lower_bound_subexp}) follows from the last inequality together with (\ref{eqn:power_signal_lower_bound}), (\ref{eqn:type_II_error_bound}), and (\ref{eqn:type_II_error_tilde_T_n}).
		
		\underline{\it{Part (ii).}} Assume (D). By the Gaussian approximation (\ref{eqn:GA_poly}), there exists a constant $C_1 := C_1(\ub, \bar{b}, K, q) > 0$ such that 
		$$
		\Prob(\tilde{T}_n > t) \le \Prob(\bar{Y}_n > t) + C_1 \{\varpi_{1,n} + \varpi_{2,n}\}
		$$
		holds for all $t \in \R$. By the same argument as in Part (i), we have (\ref{eqn:type_II_error_T_n*}) and (\ref{eqn:type_II_error_tilde_T_n}) hold with $\beta_n = C_1 \{\varpi_{1,n} + \varpi_{2,n}\} + 2 \zeta$. By Lemma \ref{lem:hatDelta}, there exists a constant $C_2 > 0$ depending only on $\bar{b}, K, q$ such that 
		$$
		\Prob\left( \max_{\us \le s \le n-\us} \max_{1 \le j \le p} \left| \sum_{i=1}^s {1 \over s} (\xi_{ij}^2  - \Sigma_{jj}) \right| \ge C_2 \{ \us^{-1/2} \log^{1/2}(np) + \gamma^{-2/q} \us^{-1} n^{2/q} \log(np) \} \right) \le \gamma / 4
		$$
		and the same probability bound holds for $ \max_{\us \le s \le n-\us} \max_{1 \le j \le p} | \bar{\xi}_{sj}^-|^2$. Then the rest of the proof follows similar lines as in Part (i).
	\end{proof}

	\begin{lem}[Bound on $\bar\psi$]
		\label{lem:bound_bar_psi}
		Assume that $X_1,\dots,X_n$ are independent random vectors that are generated from the model (\ref{eqn:mean_shifting_model}). Then we have 
		\begin{equation}
		\label{eqn:bound_bar_psi}
		\bar\psi^2 \le 2 \max_{\us \le s \le n-\us} \max_{1 \le j \le p} \left\{ {n-s \over n} \hat{S}^{\xi,-}_{n,s,jj} + {s \over n} \hat{S}^{\xi,+}_{n,s,jj} \right\} + 4 |\delta_n|_\infty^2.
		\end{equation}
	\end{lem}
	
	\subsection{Proof of Theorem \ref{thm:rate_location_estimator}}
	
	\begin{proof}[of Theorem \ref{thm:rate_location_estimator}]
		In this proof, we use $K_1, K_2,\dots$ to denote universal constants. Note that $\E [Z_n(s)] = \Delta_s$, where $\Delta_s$ is defined in (\ref{eqn:Delta_s}). Therefore, $|\E [Z_n(\cdot)] |_\infty$ reaches its maximum at $m$ and we have 
		$$
		|\E [Z_n(m)]|_\infty - |\E [Z_n(s)]|_\infty = \left\{
		\begin{array}{cc}
		\sqrt{m(n-m) \over n} \Big(1-\sqrt{s(n-m) \over m(n-s)}\Big) |\delta_n|_\infty, & \text{if } 1 \le s \le m \\
		\sqrt{m(n-m) \over n} \Big(1-\sqrt{m(n-s) \over s(n-m)}\Big) |\delta_n|_\infty, & \text{if } m < s \le n \\
		\end{array} \right. .
		$$
		If $1 \le s \le m$, then
		\begin{eqnarray*}
			1-\sqrt{s(n-m) \over m(n-s)} &=& {(\sqrt{m(n-s)} - \sqrt{s(n-m)})(\sqrt{m(n-s)} + \sqrt{s(n-m)}) \over \sqrt{m(n-s)}(\sqrt{m(n-s)} + \sqrt{s(n-m)})}\\
			&=& {n(m-s) \over m(n-s) +\sqrt{m(n-m)s(n-s)}}\\
			&\ge& {n(m-s) \over mn+ {n\over 2} \sqrt{m(n-m)} }\\
			&\ge& {m-s \over \sqrt{m}(\sqrt{m} + \sqrt{n-m})}.
		\end{eqnarray*}
		So we get
		\begin{eqnarray*}
			\sqrt{m(n-m) \over n} \Big(1-\sqrt{s(n-m) \over m(n-s)}\Big) |\delta_n|_\infty &\ge& \sqrt{m(n-m) \over n}  {m-s \over \sqrt{m}(\sqrt{m} + \sqrt{n-m})} |\delta_n|_\infty \\
			&=&  \sqrt{n} {\sqrt{1-t_{m}} \over \sqrt{t_{m}} + \sqrt{1-t_{m}}} (t_{m}-t_{s}) |\delta_n|_\infty.
		\end{eqnarray*}
		Likewise, if $m < s \le n$, then 
		$$
		1-\sqrt{m(n-s) \over s(n-m)} \ge {s-m \over \sqrt{n-m}(\sqrt{m} + \sqrt{n-m})}
		$$
		and
		$$
		\sqrt{m(n-m) \over n} \Big(1-\sqrt{m(n-s) \over s(n-m)}\Big) |\delta_n|_\infty \ge   \sqrt{n}{\sqrt{t_{m}} \over \sqrt{t_{m}} + \sqrt{1-t_{m}}} (t_{s}-t_{m}) |\delta_n|_\infty.
		$$
		Hence, for any $1 \le s \le n$, we have 
		\begin{equation}
		\label{eq:EZdiff}
		|\E [Z_n(m)]|_\infty - |\E [Z_n(s)]|_\infty \ge   \sqrt{n}\,{\sqrt{t_{m}} \wedge \sqrt{1-t_{m}} \over \sqrt{t_{m}} + \sqrt{1-t_{m}}} |t_{m}-t_{s}| |\delta_n|_\infty.
		\end{equation}
		By the triangle inequality, 
		\begin{eqnarray*}
			|Z_n(s)|_\infty - |Z_n(m)|_\infty &\le& |Z_n(s)-\E [Z_n(s)]|_\infty + |\E [Z_n(s)]|_\infty \\
			&& \qquad + |Z_n(m)-\E [Z_n(m)]|_\infty - |\E [Z_n(m)]|_\infty\\
			&\le& 2\max_{1 \le s \le n}  |Z_n(s)-\E Z_n(s)|_\infty + |\E [Z_n(s)]|_\infty  - |\E [Z_n(m)]|_\infty.
		\end{eqnarray*}
		Combining the last inequality with (\ref{eq:EZdiff}) and using $\sqrt{t_m} + \sqrt{1-t_m} \le \sqrt{2}$, we get
		$$
		\sqrt{n \over 2} (\sqrt{t_{m}} \wedge \sqrt{1-t_{m}}) |t_{m}-t_{s}|  |\delta_n|_\infty \le 2\max_{1 \le s \le n}  |Z_n(s)-\E [Z_n(s)]|_\infty + |Z_n(m)|_\infty - |Z_n(s)|_\infty.
		$$
		Replacing $s$ by $\hat{m}_{1/2}$ and noticing that $|Z_n(m)|_\infty \le |Z_n(\hat{m}_{1/2})|_\infty$, we obtain that
		\begin{align}
		\nonumber
		|t_{m}-t_{\hat{m}_{1/2}}| \le & {2 \sqrt{2} \max_{1 \le s \le n}  |Z_n(s)-\E [Z_n(s)]|_\infty \over \sqrt{n} (\sqrt{t_{m}} \wedge \sqrt{1-t_{m}}) |\delta_n|_\infty} \\
		\label{eqn:location_bound}
		\le & {2 \sqrt{2} \max_{1 \le s \le n}  |Z_n(s)-\E [Z_n(s)]|_\infty \over\sqrt{n t_m (1-t_m)} |\delta_n|_\infty},
		\end{align}
		where the last step follows from the inequality $t (1-t) \le t \wedge (1-t) \le 2 t (1-t)$ for all $t \in [0,1]$. Next, we bound $\max_{1 \le s \le n}  |Z_n(s)-\E [Z_n(s)]|_\infty$. Recall that $Z_n(s)-\E [Z_n(s)] = \sum_{i=1}^n a_{is}(X_i-\mu_i)$, where $a_{is}$ is defined in (\ref{eqn:a_is_cusum_statistic}). 
		
		\underline{\it{Part (i).}} Assume (C). By Theorem 4  in \cite{adamczak2008}, we have $\forall t > 0$,
		\begin{eqnarray}
		\nonumber
		&& \Prob(\max_{1 \le s \le n} |Z_n(s) - \E[Z_n(s)]|_\infty \ge 2 \E[\max_{1 \le s \le n} |Z_n(s) - \E[Z_n(s)]|_\infty] + t) \\
		\label{eqn:location_rate_concentration_subexp}
		&& \qquad \le \exp\left(-{t^2 \over 3 \tau^2}\right) + 3 \exp\left(-{t \over K_1 \|M\|_{\psi_1}} \right),
		\end{eqnarray}
		where 
		\begin{eqnarray*}
			\tau^2 &=& \max_{1 \le s \le n} \max_{1 \le j \le p} \sum_{i=1}^n a_{is}^2 \E(X_{ij}-\mu_{ij})^2, \\
			M &=& \max_{1 \le i,s \le n} \max_{1 \le j \le p} |a_{is} (X_{ij}-\mu_{ij})|.
		\end{eqnarray*}
		Since $\sum_{i=1}^n a_{is}^2 = 1$, we have $\tau^2 \le \bar{b}$. Since $\max_{1 \le i,s \le n} |a_{is}| \le 1$, by Lemma 2.2.2 in  \cite{vandervaartwellner1996}, we have $\|M\|_2 \le 2 \|M\|_{\psi_1} \le 2 \bar{b} \log(np)$. By Lemma E.1 in \cite{cck2016a}, we have 
		\begin{equation}
		\label{eqn:rate_location_expectation_bound}
		\E[\max_{1 \le s \le n} |Z_n(s) - \E[Z_n(s)]|_\infty] \le K_2 \{ \tau \log^{1/2}(np) + \|M\|_2 \log(np) \}.
		\end{equation}
		Choosing $t = K_3 \bar{b} \log(np) \log(\gamma^{-1})$ in (\ref{eqn:location_rate_concentration_subexp}) for some large enough universal constant $K_3 > 0$, we deduce that there exists a constant $C := C(\bar{b},K) > 0$ such that 
		$$
		\Prob(\max_{1 \le s \le n} |Z_n(s) - \E[Z_n(s)]|_\infty \ge C \log^2(np) ) \le 2 \gamma.
		$$
		Combining the last inequality with (\ref{eqn:location_bound}), we obtain (\ref{eqn:rate_location_estimator_subexp}).
		
		\underline{\it{Part (ii).}} Assume (D) with $q \ge 2$. By Theorem 2 in \cite{adamczak2010_JTP}, we have $\forall t > 0$,
		\begin{eqnarray}
		\nonumber
		&& \Prob(\max_{1 \le s \le n} |Z_n(s) - \E[Z_n(s)]|_\infty \ge 2 \E[\max_{1 \le s \le n} |Z_n(s) - \E[Z_n(s)]|_\infty] + t) \\
		\label{eqn:location_rate_concentration_poly}
		&& \qquad \le \exp\left(-{t^2 \over 3 \tau^2}\right) + C(q) {\E[M^q] \over t^q},
		\end{eqnarray}
		where $\tau^2$ and $M$ have the same definitions as in Part (i). By Lemma 2.2.2 in \cite{vandervaartwellner1996}, we have $\|M\|_2 \le \|M\|_q \le n^{1/q} \bar{b}$. Note that $\tau^2 \le \bar{b}$ and $\E[\max_{1 \le s \le n} |Z_n(s) - \E[Z_n(s)]|_\infty]$ obeys the bound in (\ref{eqn:rate_location_expectation_bound}). Hence, choosing 
		$$t = C(q) \{\bar{b}^{1/2} \log(\gamma^{-1})+ n^{1/q} \bar{b}^{1/q} \gamma^{-1/q} \}$$
		in (\ref{eqn:location_rate_concentration_poly}), we get 
		$$
		\Prob\left(\max_{1 \le s \le n} |Z_n(s) - \E[Z_n(s)]|_\infty \ge C(q)  n^{1/q} (\log(np) + \gamma^{-1/q}) \right) \le 2 \log^{-q}(np).
		$$
		Combining the last inequality with (\ref{eqn:location_bound}), we obtain (\ref{eqn:rate_location_estimator_poly}).
	\end{proof}
	
	%
	%
	
	\subsection{Proof of Theorem \ref{thm:rate_location_estimator_nonstationary}}
	
	\begin{proof}[of Theorem \ref{thm:rate_location_estimator_nonstationary}]
		Without loss of generality, we may assume $\delta_{nj} \le 0$ for all $1 \le j \le p$. In addition, we may assume that
		\begin{equation}
		\label{eqn:reduction_min_signal_delta}
		\underline{\delta}_{n} \gg {\log^2(np) \over n^{1/2}}
		\end{equation}
		in Part (i), and 
		\begin{equation}
		\label{eqn:reduction_min_signal_delta_poly}
		\underline{\delta}_n \gg {\log^{1/2}(np) \over n^{1/2}} \vee {\log(np) \over \gamma^{1/q}  n^{1/2-1/q}}
		\end{equation}
		in Part (ii), because otherwise (\ref{eqn:rate_cp_estimation_theta=0_subsexp}) and (\ref{eqn:rate_cp_estimation_theta=0_poly}) trivially hold by choosing the constant $C > 0$ large enough. Denote $h(t_m) = t_m \wedge (1-t_m)$. To simplify the notation, we write  
		\[
		\tilde{Z}_n(s) := Z_{0,n}(s) = \sum_{i=1}^s X_i - {s \over n} \sum_{i=1}^n X_i = \sqrt{s(n-s) \over n} Z_n(s),
		\]
		where $Z_n(s)$ is defined in (\ref{eqn:cusum_mean_Rp}), and $\tilde{m} = \hat{m}_0$. Let $j^*$ be an index in $\{1,\dots,p\}$ such that $\max_{1 \le j \le p} \tilde{Z}_{nj}(m) = \tilde{Z}_{nj^*}(m)$. It is clear that $j^*$ is a random variable depending on $m$. By Lemma \ref{lem:prob_event_G}, $\tilde{Z}_{nj^*}(m) = \max_{1 \le j \le p} |\tilde{Z}_{nj}(m)| \ge 0$ holds with probability greater than $1-\gamma/36$. For $r \ge 1$, observe that 
		\begin{align*}
		& \{ |t_{\tilde{m}} - t_m| > r/n \} \\
		\subset & \{ \tilde{m}-m > r \} \medcup \{ \tilde{m}-m < -(r-1) \} \\
		\subset & \{ \max_{s \le m+r} |\tilde{Z}_n(s)|_\infty < |\tilde{Z}_n(\tilde{m})|_\infty \} \medcup \{ \max_{s \le m-(r-1)} |\tilde{Z}_n(s)|_\infty = |\tilde{Z}_n(\tilde{m})|_\infty \} \\
		\subset & \{ \max_{s \ge m+r} |\tilde{Z}_n(s)|_\infty \ge |\tilde{Z}_n(m)|_\infty \} \medcup \{ \max_{s \le m-(r-1)} |\tilde{Z}_n(s)|_\infty \ge |\tilde{Z}_n(m)|_\infty \}.
		\end{align*}
		Thus we have $\Prob( |t_{\tilde{m}} - t_m| > r/n ) \le I + II$, where 
		\begin{align*}
		I =& \Prob\left( \max_{s \ge m+r} |\tilde{Z}_n(s)|_\infty \ge |\tilde{Z}_n(m)|_\infty \right) \le \Prob\left( \max_{s \ge m+r} |\tilde{Z}_n(s)|_\infty \ge \tilde{Z}_{nj^*}(m) \right), \\
		II =& \Prob\left( \max_{s \le m-(r-1)} |\tilde{Z}_n(s)|_\infty \ge |\tilde{Z}_n(m)|_\infty \right) \le \Prob\left( \max_{s \le m-(r-1)} |\tilde{Z}_n(s)|_\infty \ge \tilde{Z}_{nj^*}(m) \right).
		\end{align*}
		Because of the symmetry, we only deal with $I$ since $II$ obeys the same bound. Let 
		\[
		r = \left\{
		\begin{array}{cc}
		C(\bar{b}, K, c_1, c_2) {\log^4(np) \over \underline{\delta}_n^{2}} & \text{in Part (i)} \\
		C(\bar{b}, K, q, c_1, c_2) {\log(np) \over \underline{\delta}_n^{2}} \max \left\{1, { n^{2/q} \log(np) \over \gamma^{2/q} } \right\} & \text{in Part (ii)}
		\end{array} \right. 
		\]
		and $\mathcal{G}$ be the event where $\max_{s \ge m+r} |\tilde{Z}_n(s)|_\infty$ is attained at the coordinates of $j \in \calS$, i.e.,
		\begin{align}
		\label{eqn:event_G}
		\mathcal{G} = & \left\{ \max_{s \ge m+r} \max_{1 \le j \le p} \left|\tilde{Z}_{nj}(s) \right| = \max_{s \ge m+r} \max_{j \in \calS} \left|\tilde{Z}_{nj}(s) \right| \right\}.
		\end{align}
		By Lemma \ref{lem:prob_event_G}, $\Prob	(\mathcal{G}) \ge 1- \gamma/18$.  From now on, our analysis will be restricted to events $\mathcal{G}$ and $\tilde{Z}_{nj^*}(m) \ge 0$, where the union event holds for probability greater than $1- \gamma/12$. Note that $|x| \ge y \ge 0$ implies that either $x-y \ge 0$ or $x+y \le 0$. Then we have 
		\begin{align*}
		I \le & \Prob\left( \max_{s \ge m+r} \max_{j \in \calS} |\tilde{Z}_n(s)|_\infty \ge \tilde{Z}_{nj^*}(m) \right) + \gamma/12 \\
		\le & \Prob \left( \medcup_{s \ge m+r} \medcup_{j \in \calS} \left\{ \tilde{Z}_{nj}(s) - \tilde{Z}_{nj^*}(m) \ge 0 \right\} \right) \\
		& \quad + \Prob \left( \medcup_{s \ge m+r} \medcup_{j \in \calS} \left\{ \tilde{Z}_{nj}(s) + \tilde{Z}_{nj^*}(m) \le 0 \right\} \right) + \gamma/12\\
		\le & \Prob \left( \max_{s \ge m+r} \left[ \max_{j \in \calS} \tilde{Z}_{nj}(s) - \max_{j \in \calS} \tilde{Z}_{nj}(m) \right] \ge 0 \right) \\
		& \quad + \Prob \left( \min_{s \ge m+r} \left[ \min_{j \in \calS} \tilde{Z}_{nj}(s) + \max_{j \in \calS} \tilde{Z}_{nj}(m) \right] \le 0 \right)  + \gamma/12\\
		:= & III + IV + \gamma/12.
		\end{align*}
		Since $\max_i a_i - \max_i b_i \le \max_i(a_i-b_i)$ and $\min_i a_i - \min_i b_i \ge \min_i (a_i-b_i)$ hold for any sequences $\{a_i\}$ and $\{b_i\}$, we have 
		\begin{align*}
		III \le& \Prob \left( \max_{s \ge m+r}  \max_{j \in \calS} \left[ \tilde{Z}_{nj}(s) - \tilde{Z}_{nj}(m) \right] \ge 0 \right), \\
		IV \le& \Prob \left( \min_{s \ge m+r} \min_{j \in \calS} \left[  \tilde{Z}_{nj}(s) + \tilde{Z}_{nj}(m) \right] \le 0 \right).
		\end{align*}
		For each $s \ge m+r$ and $j = 1,\dots,p$, since $\tilde{Z}_{nj}(s) - \tilde{Z}_{nj}(m) \ge 0$ if and only if 
		\begin{align*}
		\tilde{Z}_{nj}(s) - \E[\tilde{Z}_{nj}(s)] & - \tilde{Z}_{nj}(m) + \E[\tilde{Z}_{nj}(m)] \\
		& \ge \E[\tilde{Z}_{nj}(m)] - \E[\tilde{Z}_{nj}(s)]  = -{s-m \over n} m \delta_{nj}.
		\end{align*}
		Since $\xi_i = X_i - \E(X_i)$ and $\delta_{nj} \le 0$ for all $1 \le j \le p$, we have 
		\begin{align*}
		& \left\{ \tilde{Z}_{nj}(s) - \tilde{Z}_{nj}(m) \ge 0 \right\} \subset \left\{ \left| \sum_{i=m+1}^s \xi_{ij} - {s-m \over n} \sum_{i=1}^n \xi_{ij} \right| \ge (s-m) h(t_m) |\delta_{nj}| \right\} \\
		& \qquad \subset \left\{ \left| \sum_{i=m+1}^s \xi_{ij} \right| \ge {1 \over 2} (s-m) h(t_m) |\delta_{nj}| \right\} \medcup \left\{ {1 \over n} \left| \sum_{i=1}^n \xi_{ij} \right| \ge {1 \over 2} h(t_m) |\delta_{nj}| \right\}.
		\end{align*}
		Then we have $III \le V+ VI$, where
		\begin{align}
		\nonumber
		V = & \Prob\left( \max_{s \ge m+r}  \max_{j \in \calS}  \left| {1 \over s-m} \sum_{i=m+1}^s \xi_{ij} \right| \ge {h(t_m) \over 2} \underline{\delta}_{n} \right) \\
		\label{eqn:rate_theta=0_V}
		\le & \Prob\left( \max_{r \le s' \le n-m}  \max_{j \in \calS}  \left| {1 \over s'} \sum_{i=1}^{s'} \xi_{ij} \right| \ge {h(t_m) \over 2} \underline{\delta}_{n} \right), \\
		\label{eqn:rate_theta=0_VI}
		VI = & \Prob\left( \max_{j \in \calS}  \left| {1 \over n} \sum_{i=1}^n \xi_{ij} \right| \ge {h(t_m) \over 2} \underline{\delta}_{n} \right).
		\end{align}
		Here the second inequality for bounding $V$ is due to $\xi_1,\dots,\xi_n$ are i.i.d.
		Similarly, for each $s \ge m+r$ and $j = 1,\dots,p$, $-\tilde{Z}_{nj}(s) - \tilde{Z}_{nj}(m) \ge 0$ if and only if 
		\begin{align*}
		-\tilde{Z}_{nj}(s) + \E[\tilde{Z}_{nj}(s)] &- \tilde{Z}_{nj}(m) + \E[\tilde{Z}_{nj}(m)] \\
		&\ge \E[\tilde{Z}_{nj}(m)] + \E[\tilde{Z}_{nj}(s)]  = -{2n-s-m \over n} m \delta_{nj}.
		\end{align*}
		Then we have 
		\begin{align*}
		& \qquad \left\{ -\tilde{Z}_{nj}(s) - \tilde{Z}_{nj}(m) \ge 0 \right\} \\
		& \subset \left\{ \left| {1 \over n} \sum_{i=1}^n \xi_{ij} \right| \ge {h(t_m) \over 2} |\delta_{nj}| \right\} \medcup \left\{ {1 \over 2n-s-m} \left| \sum_{i=m+1}^n \xi_{ij} + \sum_{i=s+1}^n \xi_{ij} \right| \ge {h(t_m) \over 2} |\delta_{nj}| \right\}.
		\end{align*}
		Since $(2n-s-m)^{-1} \le (n-m)^{-1}$, we have 
		\begin{align*}
		&\left\{ {1 \over 2n-s-m} \left| \sum_{i=m+1}^n \xi_{ij} + \sum_{i=s+1}^n \xi_{ij} \right| \ge {h(t_m) \over 2} |\delta_{nj}| \right\}\\
		&\subset  \left\{ {1 \over n-m} \left| \sum_{i=m+1}^n \xi_{ij}  \right| \ge {h(t_m) \over 4} |\delta_{nj}| \right\} \medcup \left\{ {1 \over n-m} \left|  \sum_{i=s+1}^n \xi_{ij} \right| \ge {h(t_m) \over 4} |\delta_{nj}| \right\}.
		\end{align*}
		Then we obtain that 
		\begin{align*}
		IV =  \Prob \left( \max_{s \ge m+r} \max_{j \in \calS} \left[ - \tilde{Z}_{nj}(s) - \tilde{Z}_{nj}(m) \right] \ge 0 \right) \le  VI + 2 VII,
		\end{align*}
		where 
		\begin{equation}
		\label{eqn:rate_theta=0_VII}
		VII = \Prob \left( \max_{s \ge m} \max_{j \in \calS} \left| {1 \over n} \sum_{i=s+1}^n \xi_{ij} \right| \ge { h(t_m) \over 8} \underline{\delta}_{n} \right).
		\end{equation}
		So now we have $I \le V + 2VI +2VII + \gamma/12$.

		\underline{\it{Part (i)}}. 	Suppose (C) holds. 
		To bound V, applying Lemma \ref{lem:maximal_subexponential}, we have for any $u > 0$
		\begin{align*}
		& \Prob \left( \max_{r \le s' \le n-m}  \max_{1 \le j \le p}  \left|  \sum_{i=1}^{s'} {1 \over s'} \xi_{ij} \right| \ge 2 \E\left[  \max_{r \le s' \le n-m} \max_{1 \le j \le p} \left| \sum_{i=1}^{s'} {1 \over s'} \xi_{ij} \right| \right] + u \right)\\
		& \qquad \le \exp \left( -{u^2 \over 3 \tau_1^2} \right) + 3 \exp \left( -{u \over K_1 \|M_1\|_{\psi_1}}\right),
		\end{align*}
		where
		\begin{align*}
		\tau_1^2 = \max_{r \le s' \le n-m} \max_{1 \le j \le p} \sum_{i=1}^{s'} {1 \over s'^2} \E(\xi_{ij}^2)  \quad \text{and} \quad M_1 = \max_{1 \le i \le n} \max_{r \le s' \le n-m} \max_{1 \le j \le p} {1 \over s'} |\xi_{ij}|.
		\end{align*}
		Note that $\tau_1^2 \le r^{-1} \bar{b} $ and 
		\begin{align*}
		\|M_1\|_2 & \le K_2 \|M_1\|_{\psi_1} = K_2 \left\| \max_{1 \le i \le n} \max_{1 \le j \le p} \max_{r \le s' \le n-m} ({s'}^{-1}) |\xi_{ij}| \right\|_{\psi_1} \\
		& \quad \le K_2 r^{-1} \log(np) \max_{1 \le i \le n} \max_{1 \le j \le p}  \left\| \xi_{ij} \right\|_{\psi_1} \le K_2 \bar{b} r^{-1} \log(np).
		\end{align*}
		Using Lemma E.1 in \cite{cck2016a}, we have
		\begin{align*}
		\E\left[  \max_{r \le s' \le n-m} \max_{1 \le j \le p} \left| \sum_{i=1}^{s'} {1 \over s'} \xi_{ij} \right| \right] \le & K_3 \left\{ \sqrt{\log(np)} \tau_1 + \log(np) \|M_1\|_2 \right\} \\
		\le & C_1(\bar{b}) \left\{ r^{-1/2} \log^{1/2}(np) + r^{-1} \log^2(np) \right\}\\
		\le & C_1(\bar{b}) r^{-1/2} \log^2(np).
		\end{align*}
		Therefore, we have 
		\begin{align*}
		& \Prob \left( \max_{r \le s' \le n-m}  \max_{1 \le j \le p}  \left|  \sum_{i=1}^{s'} {1 \over s'} \xi_{ij} \right| \ge 2  C_1(\bar{b}) r^{-1/2} \log^2(np) + u \right)\\
		& \qquad \le \exp \left( -{r u^2 \over 3 \bar{b}} \right) + 3 \exp \left( -{r u \over K_3 \bar{b} \log(np)}\right).
		\end{align*}
		Let $u^* = C^*(\bar{b}, K) r^{-1/2} \log^2(np)$. Then it follows from the assumption $\log(1/\gamma) \le K \log(np)$ that 
		\begin{align*}
		V \le \Prob \left( \max_{r \le s' \le n-m}  \max_{1 \le j \le p}  \left| {1 \over s'} \sum_{i=1}^{s'} \xi_{ij} \right| \ge u^* \right)  \le \gamma/12.
		\end{align*} 
		%
		Similarly, to bound VI, by Lemma \ref{lem:maximal_subexponential}, we have for any $u > 0$
		\begin{align*}
		& \Prob\left( \max_{1 \le j \le p}  \left| {1 \over n} \sum_{i=1}^n \xi_{ij} \right| \ge  2 \E \left[\max_{1 \le j \le p}  \left| {1 \over n} \sum_{i=1}^n \xi_{ij} \right|\right] + u \right) \\
		& \qquad \le  \exp \left( -{u^2 \over 3 \tau_{2}^2} \right) + 3 \exp \left( -{u \over K_1 \|M_{2}\|_{\psi_1}}\right),
		\end{align*}
		where
		\begin{align*}
		\tau_{2}^2 = \max_{1 \le j \le p} \sum_{i=1}^{n} {1 \over n^2} \E(\xi_{ij}^2) \le \bar{b} n^{-1} \quad \text{and} \quad M_{2} = \max_{1 \le i \le n} \max_{1 \le j \le p} {1 \over n} |\xi_{ij}|.
		\end{align*}
		Note that $\|M_{2}\|_2 \le \sqrt{2} \|M_{2}\|_{\psi_1} \le K_4 \bar{b}  n^{-1} \log(np)$. Since $\log^3(np) \le K n$, we have 
		\begin{align*}
		\E \left[\max_{1 \le j \le p}  \left| {1 \over n} \sum_{i=1}^n \xi_{ij} \right|\right] \le & C_2(\bar{b}) \left\{ n^{-1/2} \log^{1/2}(np) + n^{-1} \log^2(np) \right\} \\
		\le & C_2(\bar{b}, K) n^{-1/2} \log^{1/2}(np).
		\end{align*}
		Let $u^\diamond = C^\diamond(\bar{b}, K) n^{-1/2} \log^{1/2}(np)$. Then it yields that 
		\begin{align*}
		VI \le 	\Prob\left( \max_{1 \le j \le p}  \left| {1 \over n} \sum_{i=1}^n \xi_{ij} \right| \ge u^\diamond  \right) \le \gamma/12.
		\end{align*}
		For VII, notice that  
		\begin{align*}
		&\Prob \left(\max_{n-1 \ge s \ge m} \max_{1 \le j \le p} \left| {1 \over n} \sum_{i=s+1}^n \xi_{ij} \right| \ge 2 \E \left[ \max_{n-1 \ge s \ge m} \max_{1 \le j \le p} \left| {1 \over n} \sum_{i=s+1}^n \xi_{ij} \right| \right] + u\right) \\
		& \qquad \le \exp \left( -{u^2 \over 3 \tau_{3}^2} \right) + 3 \exp \left( -{u \over K_1 \|M_{3}\|_{\psi_1}}\right),
		\end{align*}
		where
		\begin{align*}
		&\tau_{3}^2 = \max_{n-1 \ge s \ge m} \max_{1 \le j \le p} \sum_{i=s+1}^{n} {1 \over n^2} \E(\xi_{ij}^2) \le \bar{b} n^{-2}(n-m) \le \bar{b} n^{-1}, \\
		&M_{3} =  \max_{n-1 \ge s \ge m} \max_{1 \le i \le n} \max_{1 \le j \le p} |{1 \over n}\xi_{ij}| = M_{2}, \\
		& \E \left[\max_{n-1 \ge s \ge m} \max_{1 \le j \le p} \left| {1 \over n} \sum_{i=s+1}^n \xi_{ij} \right|\right] \le C_3(\bar{b}) \left\{ n^{-1/2} \log^{1/2}(np) + n^{-1} \log^2(np) \right\}.
		\end{align*}
		Then it follows that 
		\begin{align*}
		VII \le \Prob  \left( \max_{n-1 \ge s \ge m} \max_{1 \le j \le p} \left| {1 \over n} \sum_{i=s+1}^n \xi_{ij} \right| \ge u^\diamond \right) \le \gamma/12.
		\end{align*}
		Now combining these estimates into (\ref{eqn:rate_theta=0_V}), (\ref{eqn:rate_theta=0_VI}), and (\ref{eqn:rate_theta=0_VII}), we conclude that $I \le \gamma / 2$ holds under the assumption (\ref{eqn:reduction_min_signal_delta}) and choosing a large enough constant $C(\bar{b}, K, c_1, c_2) > 0$ in the definition of $r$. Same bound holds for $II$ and (\ref{eqn:rate_cp_estimation_theta=0_subsexp}) follows.
		
		\underline{\it{Part (ii)}}. Suppose (D) holds. To bound V, applying Lemma \ref{lem:maximal_subexponential}, we have
		\begin{align*}
		& \Prob \left( \max_{r \le s' \le n-m}  \max_{1 \le j \le p}  \left|  \sum_{i=1}^{s'} {1 \over s'} \xi_{ij} \right| \ge 2 \E\left[  \max_{r \le s' \le n-m} \max_{1 \le j \le p} \left| \sum_{i=1}^{s'} {1 \over s'} \xi_{ij} \right| \right] + u \right)\\
		& \qquad \le \exp \left( -{u^2 \over 3 \tau_1^2} \right) + C(q)  {\E [M_1^q] \over u^q}
		\end{align*}
		holds for all $u > 0$. Note that $\tau_1^2 \le r^{-1} \bar{b} $, and for $q \ge 2$ we have $\|M_1\|_2 \le \|M_1\|_q$ with
		\begin{align*}
		\|M_1\|_q^q = \E \left[ \max_{1 \le i \le n} \max_{1 \le j \le p} (\max_{s' \ge r} {1 \over s'}) |\xi_{ij}|^q \right] \le r^{-q} \sum_{i=1}^n \E \left[  \max_{1 \le j \le p} |\xi_{ij}|^q \right] \le \bar{b}^q r^{-q} n.
		\end{align*}
		Thus, 
		\begin{align*}
		\E\left[  \max_{r \le s' \le n-m} \max_{1 \le j \le p} \left| \sum_{i=1}^{s'} {1 \over s'} \xi_{ij} \right| \right] \le & K_5 \left\{ \sqrt{\log(np)} \tau_1 + \log(np) \|M_1\|_2 \right\} \\
		\le & C_4(\bar{b},q) \left\{  r^{-1/2} \log^{1/2}(np) + r^{-1} n^{1/q} \log(np) \right\}.
		\end{align*}
		Let 
		$$
		u^* = C^*(\bar{b}, K, q) r^{-1/2} \log^{1/2}(np) \max \{ 1, \gamma^{-1/q} n^{1/q} \log^{1/2}(np) \}. 
		$$
		Then, we have 
		\begin{align*}
		V &\le \Prob \left( \max_{r \le s' \le n-m}  \max_{1 \le j \le p}  \left|  \sum_{i=1}^{s'} {1 \over s'} \xi_{ij} \right| \ge u^* \right) \le \gamma/12.
		\end{align*}
		To bound VI, note that
		\begin{align*}
		\Prob\left( \max_{1 \le j \le p}  \left| {1 \over n} \sum_{i=1}^n \xi_{ij} \right| \ge  2 \E \left[\max_{1 \le j \le p}  \left| {1 \over n} \sum_{i=1}^n \xi_{ij} \right| \right] + t \right) \le  \exp \left( -{t^2 \over 3 \tau_{2}^2} \right) + C(q)  {\E [M_2^q] \over t^q},
		\end{align*}
		where $\tau_{2}^2$ and $M_{2}$ are defined the same as in in Part (i). Since	$\tau_{2}^2 \le n^{-1} \bar{b}$ and
		\begin{align*}
		&\|M_2\|_q^q = {1 \over n^q} \E \left[ \max_{1 \le i \le n} \max_{1 \le j \le p} |\xi_{ij}|^q \right] \le {1 \over n^q} \sum_{i=1}^n \E \left[  \max_{1 \le j \le p} |\xi_{ij}|^q \right] \le \bar{b}^q n^{1-q},\\
		& \E \left[ \max_{1 \le j \le p}  \left| {1 \over n} \sum_{i=1}^n \xi_{ij} \right| \right] \le K_6 \left\{ \sqrt{\log(np)} \tau_2 + \log(np) \|M_2\|_2 \right\} \\
		&\qquad \qquad \qquad \qquad\le C_5(\bar{b},q) \left\{ n^{-1/2} \log^{1/2}(np) + n^{1/q-1} \log(np) \right\} .
		\end{align*}
		As $\log(\gamma^{-1}) \le K \log(np)$, we can take 
		\begin{align*}
		u^\diamond &= C^\diamond(\bar{b},K,q) n^{-1/2} \log^{1/2}(np) \max \{ 1, \gamma^{-1/q} n^{1/q-1/2} \log^{1/2}(np) \} 
		\end{align*}
		so that 
		\begin{align*}
		VI \le 	\Prob\left( \max_{1 \le j \le p}  \left| \sum_{i=1}^n \xi_{ij} \right| \ge u^\diamond  \right) \le \gamma/12.
		\end{align*} 
		For VII, notice that
		\begin{align*}
		&\Prob \left(\max_{s \ge m} \max_{1 \le j \le p} \left| {1 \over n} \sum_{i=s+1}^n \xi_{ij} \right| \ge 2 \E \left[ \max_{s \ge m} \max_{1 \le j \le p} \left| {1 \over n} \sum_{i=s+1}^n \xi_{ij} \right| \right] + t\right) \\
		& \qquad \le \exp \left( -{t^2 \over 3 \tau_{3}^2} \right) + C(q) {\E [M_3^q] \over t^q},
		\end{align*}
		where $\tau_{3}^2 \le \bar{b} n^{-1}$ and $M_{3}$ are defined the same as in Part (i). 
		Then we have 
		\begin{align*}
		VII \le \Prob  \left( \max_{s \ge m} \max_{1 \le j \le p} \left| {1 \over n} \sum_{i=s+1}^n \xi_{ij} \right| \ge t^\diamond \right) \le \gamma/12.
		\end{align*}
		Hence, $I \le \gamma / 2$ under the assumption (\ref{eqn:reduction_min_signal_delta_poly}), and choosing a large enough constant $C(\bar{b},K,q,c_1,c_2) > 0$ in the definition of $r$. By a similar argument, $II$ obeys the same bound as $I$.
	\end{proof}
	
	\begin{lem}
		\label{lem:prob_event_G}
		Suppose that (B) holds and $H_1$ is true with a change point $m$ satisfying $c_1 \le t_m \le c_2$ for some constants $c_1,c_2 \in (0,1)$. Suppose that $\log^3(np) \le K n$ and $\log(\gamma^{-1}) \le K\log(np)$ for some constant $K > 0$. Let $\mathcal{G}$ be defined in (\ref{eqn:event_G}). 
		(i) If (C) and (\ref{eqn:reduction_min_signal_delta}) hold, then $\Prob ( \mathcal{G} ) \ge 1- \gamma/18$, where  $r = C(\bar{b}, K, c_1, c_2) \underline{\delta}_n^{-2} \log^4(np)$.
		(ii) If (D) and (\ref{eqn:reduction_min_signal_delta_poly}) hold, then $\Prob ( \mathcal{G} ) \ge 1- \gamma/18$, where $r = C(\bar{b}, K, q, c_1, c_2) \underline{\delta}_n^{-2} \log(np) \max\{1, \gamma^{-2/q} n^{2/q} \log(np)\}$.
		In addition, if $\delta_{nj} \le 0$ for all $1 \le j \le p$, then $\Prob \left( \max_{1 \le j \le p} \left| \tilde{Z}_{nj}(m) \right| = \tilde{Z}_{nj^*}(m) \right) \ge 1 - \gamma/36$ in both (i) and (ii), where $j^* \in \{1,\dots,p\}$ is defined as $\tilde{Z}_{nj^*}(m) = \max_{1 \le j \le p} \tilde{Z}_{nj}(m)$.
	\end{lem}
	
	\subsection{Proof of Theorem \ref{thm:binseg_consistency_mcp}}
	
	\begin{proof}[of Theorem \ref{thm:binseg_consistency_mcp}]
		\underline{\it{Part (i)}} Assuming (C).
		Suppose there are undetected change-points in $[b,e]$ which satisfies
		\[
		m_{k_0} \le b < m_{k_0+1} < \dots < m_{k_0+q} < e \le m_{k_0+q+1}  
		\]
		for $0 \le k_0 \le \nu -q$.  
		Let $\lambda_1 = \lambda_2 = \bar{b} \log^2 (np)$,
		\[
		\calA_n = \left\{ | (e-b+1)^{-1/2}  \sum_{i=b}^e \xi_i|_\infty < K_2 \lambda_2 , \forall 1 \le b \le e \le n\right\}
		\]
		and
		\[
		\calB_n = \left\{ \max_{1 \le b \le s \le e \le n} |Z_{n,b,e}(s) - EZ_{n,b,e}(s)|_\infty < K_1 \lambda_1 \right\}.
		\]
		By Lemma \ref{lem:randomness_ctrl}, $\Prob (\calA_n \cap \calB_n) \ge 1- 2\gamma$. In the following, we will only consider the case of $\calA_n \cap \calB_n$ where the random component of $\{X_i\}_{i=1}^n$, namely the behavior of $\{\xi_i\}_{i=1}^n$, is well characterized.
		
		Let $\hat{m}_{b,e} = \argmax_{b \le s \le e} |Z_{n,b,e}(s)|_\infty $ be the location where the maximum CUSUM statistic within interval $[b,e]$ is reached and let $h = \argmax_{j=1, \dots, p} |Z_{n,b,e}(\hat{m}_{b,e})|_j$ be its dimension. According to Algorithm \ref{alg:multi_cp}, if $\hat{m}_{b,e}$ is large enough to pass the bootstrap test and it is close to one of $\{ m_{k_0+1}, \dots, m_{k_0+q} \}$, then one change point is consistently identified. The remaining proof follows the structure of proof in Theorem 1 in \cite{fryzlewicz2014} to complete this claim with modifications on the key steps: i) $\hat{m}_{b,e} \in \calS_n$; ii) Power of bootstrap test is guaranteed such that $\Prob (\calS_n)$ is bounded below.
		Note that, the main differences between the proof of Theorem 1 in \cite{fryzlewicz2014} and our argument are in the following aspects.
		First, \cite{fryzlewicz2014} considers one-dimensional observations while we extend it to $p$-dimensional case. The set $\calA_n$ and $\calB_n$ are adapted to sub-exponential random vectors.
		Second, Lemma A.5 in \cite{fryzlewicz2014} provides stopping conditions when the following (\ref{eqn:m_r_within_range}) and (\ref{eqn:boundary_cut}) fail, that is the search stops if (i) no change point in $[b,e]$ or (ii) the only 1 or 2 change points left in $[b,e]$ are within the distance of $\us$ from $b$ or $e$ (therefore they should be classified as the same change point as $b$ or $e$). But under our context, testing procedure substitutes subjective pick of threshold so that we do not need to control the magnitude of $ |Z_{n,b,e}(s)|_\infty$. 
		
		Apply Theorem 1 in \cite{fryzlewicz2014}, we can make the following statements on set $\calB_n \cap \calA_n$. 
		When 
		\begin{equation}
		\label{eqn:m_r_within_range}
		b < m_{k_0+r} - C_3 D_\nu < m_{k_0+r} + C_3 D_\nu <e \text{ for some } 1 \le r \le q
		\end{equation}
		and
		\begin{equation}
		\label{eqn:boundary_cut}
		\max \left\{ \min \{m_{k_0 + 1} - b, b - m_{k_0} \} , \min \{m_{k_0 + q + 1} -e, e - m_{k_0+ q} \} \right\} \le C_4 \epsilon_n
		\end{equation}
		hold for some constants $C_3, C_4$, by Lemma A.3 in \cite{fryzlewicz2014},
		$$
		| \hat{m}_{b,e} - m_{k_0+r}| \le C_5 \epsilon_n' = C_5 \lambda_2^2 n^2 D_\nu^{-2} (\delta^{(p_0+r)}_h)^{-2} 
		$$
		i.e. $\hat{m}_{b,e}$ falls within the distance of $C_5 \bar{b}^2 \epsilon_n$ of a previously undetected change point $m_{k_0+r}$.
		
		Note that $h \in \calD_{k_0+r}$, i.e. $\delta^{(k_0+r)}_h \neq 0$. Otherwise, if the $h^{th}$-dimension has a mean-shift within the working interval $[b,e]$, then it contradicts with $|Z_{n,b,e}(s)|_\infty = |Z_{n,b,e}(\hat{m}_{b,e})|_h$. If there is no mean-shift in the $h^{th}$-dimension within $[b,e]$, then $|Z_{n,b,e}( \hat{m}_{b,e} )|_h \le  C_6(\bar{b}) \log^2(np)$.
		But by Lemma \ref{lem:max_delta_s_location},
		\begin{align*}
		\max_{b + \us \le s \le e-\us} |E Z_{n,b,e}(s)|_\infty &\ge \max_{l=k_0, \dots, k_0+q+1} {C D_\nu^2 \bar\delta_n \over \sqrt{(e-b) (m_{l}-b) (e-m_{l})}}\\ 
		&\ge {C D_\nu^2 \udelta_n \over \sqrt{n D_\nu (n-D_\nu)}} \\
		&\ge C_7(\bar{b}) D_{\nu}^{3/2} n^{-1} \udelta_n
		\end{align*}
		The second line is due to $D_{\nu} \ge \epsilon_n \gtrsim b - m_{k_0} \ge m_{k_0+1} - m_{k_0} \ge D_{\nu}$ by (\ref{eqn:boundary_cut}) and Assumption a), b), e). 
		It leads to contradiction. When $\delta^{(p_0+r)}_h = 0$ but $\delta^{(p_0+r+1)}_h \neq 0$, we just need to record $r+1$ as $r$. 
		
		Next, according to Lemma \ref{lem:power_mcp}, in order to have type II error (of performed test in current interval $[b,e]$) is less than $\gamma +2\zeta + C_2 \varpi_{1,(e-b)}$, we only need to show $|\tilde{\Delta}|_\infty = |\E Z_{n,b,e}(s)|_\infty \ge C \nu \log^{1/2}((e-b)p)$ for some $s \in [b+\us,e-\us]$, $C_1 = C_1 (\ub, \bar{b}, K)$ and $C_2 = C_2 (\ub, \bar{b}, K)$. Here, $\varpi_{1,(e-b)} = (\log^7 ((e-b)p) /\us)^{1/6}$ with $\us$ determined by $n$, not $(e-b)$.
		Then by Assumption d), 
		\begin{align*}
		\max_{b + \us \le s \le e-\us} |E Z_{n,b,e}(s)|_\infty \ge C_7 D_{\nu}^{3/2} n^{-1} \udelta_n &\ge C_7 C^2 n^{{3 \over 2}\Theta-1-\omega}\\ 
		&> C_8(\bar{b},\ub,K) \sqrt{\log(\zeta^{-1}) \log(np) + \nu^2 \log(np/\alpha)} ,
		\end{align*}
		
		which indicates the change point is significant to pass the bootstrap test with probability no less than $1-\gamma -2\zeta - C_0^{'} \varpi_{1,n}$ where $\varpi_{1,n} \ge \varpi_{1,(e-b)}$ by Assumption e).
		
		As a consequence, the procedure then moves on to operate on the intervals $[b, \hat{m}_{b,e}]$ and $[\hat{m}_{b,e}, e]$ where both (\ref{eqn:m_r_within_range}) and (\ref{eqn:boundary_cut}) still hold. Therefore, all change points will be detected one by one until the conditions in Lemma A.5 of \cite{fryzlewicz2014} are met. Assumption c) is implicitly called in Theorem 1 of \cite{fryzlewicz2014}.

		\underline{\it{Part (ii)}} Assuming (D).
		Let $\lambda_1 = \lambda_2 = C(q) \bar{b} n^{3/q} (\log(np) + \gamma^{-1/q})$. Similarly, Lemma \ref{lem:randomness_ctrl} shows $\Prob (\calA_n \cap \calB_n) \ge 1- 2\gamma$.
		According to the proof of Theorem \ref{thm:power_bootstrap_cusum_test}, Type-II error is less than $\gamma +2\zeta + C_0^{'} ( \varpi_{1,(e-b)} + \varpi_{2,(e-b)})$ when $|\tilde{\Delta}|_\infty = |\E Z_{n,b,e}(s)|_\infty \ge C \log((e-b)p)$ for some $s \in [b+\us,e-\us]$. Using the same arguments, we conclude that for each step in binary segmentation the bootstrap test is passed with probability no less than $1-\gamma -2\zeta - C_0^{'} (\varpi_{1,n} + \varpi_{2,n})$ and estimated location $\hat{m}_{b,e}$ falls within distance of $C\epsilon_n = C(\bar{b},q) n^{2+6/q} D_\nu^{-2} (\udelta_n)^{-2} (\log^2(np) + \gamma^{-2/q})$ until the stopping conditions are met.
	\end{proof}
	
	\begin{lem}
		\label{lem:randomness_ctrl}
		$\Prob (\calA_n) \ge 1- \gamma$ and $\Prob (\calB_n) \ge 1- \gamma$ for $\gamma$ defined the same as Theorem \ref{thm:binseg_consistency_mcp}.
	\end{lem}

	\begin{proof}[of Lemma \ref{lem:power_mcp}]
		The structure of this proof is similar to the one for Theorem \ref{thm:power_bootstrap_cusum_test}.
		Without less of generality we may assume $\mu=0$. For $\xi_ i = X_i - \mu_i$ where $\mu_i = \E[X_i]$ has the form of 
		\[
		\mu_{m_k + 1} = \dots = \mu_{m_{k+1}} = \sum_{l=0}^k \delta^{(l)},
		\]
		the CUSUM statistic computed on $X_1,\dots,X_n$ can be decomposed as $Z_n(s) = Z_n^\xi(s) + \Delta_s$, where $	Z_n^\xi(s) = \sqrt{s(n-s) \over n} \left\{{1\over s} \sum_{i=1}^s \xi_i - {1 \over n-s} \sum_{i=s+1}^n \xi_i \right\}$ is defined the same as in the proof of Theorem \ref{thm:power_bootstrap_cusum_test} but $\Delta_s = \E Z_{n}(s)$ is extended to multiple change point model as in (\ref{eqn:Delta_s_mcp}). Again, by Lemma \ref{lem:max_delta_s_location}, $|\Delta_s|_\infty$ reaches its maximum at one of the change points, i.e.,
		\begin{equation}
		\label{eqn:max_signal_lowerbd_mcp}
		\max_{\us \le s \le n - \us} |\Delta_s|_\infty = \max_{k = 1, \dots, \nu} |\Delta_{m_k}|_\infty \gtrsim \max_{l=1, \dots, \nu} {C D_\nu^2 \bar\delta_n \over \sqrt{n m_{l} (n-m_{l})}} =: \tilde\Delta.
		\end{equation}
		The type II error obeys the same bound in arguments of (\ref{eqn:type_II_error_lower_bound}) and (\ref{eqn:type_II_error_bound}). 
		
		\underline{\it{Part (i).}} Assume (C). Take $\beta_n = C_1 \varpi_{1,n} + 2 \zeta$, $q_{\tilde{T}_n}(1-\beta_n)$ holds for the same bound in (\ref{eqn:type_II_error_tilde_T_n}) . Note that $q_{T_{n}^{*} \mid X_{1}^{n}}(1-\alpha)$ is bounded the same as in (\ref{eqn:type_II_error_T_n*}), but $\psi$ changes to
		\begin{equation}
		\label{eqn:bound_bar_psi_mcp}
		\bar\psi^2 \le 2 \max_{\us \le s \le n-\us} \max_{1 \le j \le p} \left\{ {n-s \over n} \hat{S}^{\xi,-}_{n,s,jj} + {s \over n} \hat{S}^{\xi,+}_{n,s,jj} \right\} + 4 \nu^2 \bar\delta_n^2.
		\end{equation}
		The sketch of proof for (\ref{eqn:bound_bar_psi_mcp}) is shown in Part (iii).
		Therefore, based on the same probability bounds of $\max_{\us \le s \le n-\us} \max_{1 \le j \le p} \left| {1 \over s} \sum_{i=1}^s (\xi_{ij}^2  - \Sigma_{jj}) \right|$ and $\max_{\us \le s \le n-\us} \max_{1 \le j \le p} \left| \bar{\xi}_{sj}^- \right|^2$, we deduce that
		\[
		\Prob \left( q_{T_{n}^{*} \mid X_{1}^{n}}(1-\alpha) \ge C_5 \max\{\nu \bar\delta_n, 1\} \log^{1/2}(np/\alpha) \right) \le \gamma.
		\]
		Therefore, (\ref{eqn:power_lower_bound_subexp}) follows.
		
		\underline{\it{Part (ii).}} Assume (D). Arguments are exactly the same as Part (ii) in Theorem \ref{thm:power_bootstrap_cusum_test}.
		
		\underline{\it{Part (iii).}} The result of (\ref{eqn:bound_bar_psi_mcp}) comes from the proof of Lemma \ref{lem:bound_bar_psi} with a modification to multiple mean-shifts model (\ref{eqn:model_mcp}).
		Recall that $\bar{X}_s^- = s^{-1} \sum_{i=1}^s X_i$, $\bar{X}_s^+ = (n-s)^{-1} \sum_{i=s+1}^n X_i$, and $\bar{\xi}_s^-, \bar{\xi}_s^+, \bar{\mu}_s^-, \bar{\mu}_s^+$ are similarly defined by replacing $X_1^n$ with $\xi_1^n$ and $\mu_1^n$, respectively.
		Then, elementary calculations yield $\bar{X}_s^- = \bar{\xi}_s^- + \bar{\mu}_s^-$ and $\bar{X}_s^+ = \bar{\xi}_s^+ + \bar{\mu}_s^+ $, where
		\[
		\bar{\mu}_s^- = {1 \over s} \sum_{l=0}^k (s-m_l) \delta^{(l)} \text{  and  } \bar{\mu}_s^+ = \sum_{l=0}^k \delta^{(l)} + {1 \over n-s} \sum_{l=k+1}^{\nu} (n-m_l) \delta^{(l)} \text{  for  } m_k < s \le m_{k+1}.
		\]
		Note that 
		\[
		\left\{
		\begin{array}{l}
		{1 \over \sqrt{s}} \sum_{i=1}^s e_i (X_i - \bar{X}_s^-) = \underbrace{ {1 \over \sqrt{s}} \sum_{i=1}^s e_i (\xi_i - \bar{\xi}_s^-)}_{:=A} + {1 \over \sqrt{s}} \sum_{i=1}^s e_i (\mu_i - \bar{\mu}_s^-) \\
		{1\over \sqrt{n-s}} \sum_{i=s+1}^n e_i (X_i - \bar{X}_s^+) = \underbrace{ {1\over \sqrt{n-s}}  \sum_{i=s+1}^n e_i (\xi_i - \bar{\xi}_s^+)}_{:=B} + {1\over \sqrt{n-s}} \sum_{i=s+1}^n e_i (\mu_i - \bar{\mu}_s^+) \\
		\end{array} \right. ,
		\]
		where
		\[ 
		\left\{
		\begin{array}{lllll}
		\sum_{i=1}^s (\mu_{ij} - \bar{\mu}_{sj}^-)^2  &=\sum_{i=1}^s \mu_{ij}^2 -s (\bar{\mu}_{sj}^-)^2 &\le  \sum_{i=1}^s \mu_{ij}^2 &\le& s \nu^2 \bar\delta_n^2  \\
		\sum_{i=s+1}^n  (\mu_{ij} - \bar{\mu}_{sj}^+)^2  &=\sum_{i=s+1}^n \mu_{ij}^2 - (n-s) (\bar{\mu}_{sj}^+)^2 &\le \sum_{i=s+1}^n \mu_{ij}^2 &\le& (n-s) \nu^2 \bar\delta_n^2   \\
		\end{array} \right. 
		\]
		as $|\mu_s|_\infty = |\sum_{l=0}^k \delta^{(l)}|_\infty \le \nu \bar\delta_n$ for $s = m_{k}+1 , \dots, m_{k+1}$.
		Similarly, we have
		\begin{eqnarray*}
			\bar\psi^2 &=& \max_{\us \le s \le n-\us} \max_{1 \le j \le p} \Big\{ {n-s \over n} \Cov_e \big( A_j + {1 \over \sqrt{s}} \sum_{i=1}^s  e_i (\mu_{ij} - \bar{\mu}_{sj}^-) \big) \\
			&& \qquad \qquad \qquad + {s \over n}  \Cov_e \big( B_j + {1 \over \sqrt{n-s}} \sum_{i=s+1}^n e_i (\mu_{ij} - \bar{\mu}_{sj}^+) \big) \Big\} \\
			&\le& 2 \max_{\us \le s \le n-\us} \max_{1 \le j \le p} \Big\{ {n-s \over n} \Cov_e(A_j) + {s \over n}  \Cov_e( B_j) \Big\} \\
			&& + 2 \max_{\us \le s \le n-\us} \max_{1 \le j \le p} \left[ {1 \over s} \sum_{i=1}^s  (\mu_{ij} - \bar{\mu}_{sj}^-)^2 +  {1 \over n - s} \sum_{i=s+1}^n  (\mu_{ij} - \bar{\mu}_{sj}^+)^2\right] .
		\end{eqnarray*}
		Then (\ref{eqn:bound_bar_psi_mcp}) is immediate.
	\end{proof}
	
	\begin{lem}
		\label{lem:max_delta_s_location}
		$\argmax_{s=1, \dots, n-1} |\E Z_{n}(s)|_\infty \subset \{m_k, k = 1, \dots, \nu\}$. Consequently, we automatically have $\argmax_{s=b, \dots, e} |\E Z_{n,b,e}(s)|_\infty \subset \{m_k, k = 1, \dots, \nu\} \cap [b,e]$. Moreover, there exist some constant $C$ such that
		$$
		\max_{l=1, \dots, \nu} |\Delta_{m_l}|_\infty \ge \max_{l=1, \dots, \nu} {C D_\nu^2 \bar\delta_n \over \sqrt{n m_{l} (n-m_{l})}} \ge \max_{l=1, \dots, \nu} C n^{-5/2}  \sqrt{m_{l} (n-m_{l})} D_\nu^2 \bar\delta_n.
		$$
	\end{lem}

	\section{Proof of auxiliary lemmas}
	
	This section contains auxiliary lemmas for the supplementary material Section~\ref{sec:proofs}. 
	
	\begin{proof}[of Lemma \ref{lem:maximal_subexponential}]
		For $i=1,\dots,n$, let 
		\[
		Y_{i} = \left( \begin{array}{ccc}
		a_{1i} X_{i1} & \dots & a_{1i} X_{ip} \\
		\vdots & \ddots & \vdots \\
		a_{ni} X_{i1} & \dots & a_{ni} X_{ip} \\
		\end{array} \right).
		\]
		Then $Y_1,\dots,Y_n$ is a sequence of independent mean-zero random matrices in $\R^{n \times p}$. Note that $W_n = \sum_{i=1}^n Y_i$ and $Z_n = |W_n|_\infty$. Then (\ref{eqn:maximal_subexponential}) is an immediate consequence of Lemma E.3 in \cite{cck2016a}.
	\end{proof}
	
	\begin{proof}[of Lemma \ref{lem:hatDelta}]
		\underline{\it{Part (i)}}. Assume (C). Write  
		$$
		\hat\Delta_1 = \max_{\us \le s \le n-\us}  \max_{1 \le j,k \le p} \left |\sum_{i=1}^n b_{is} (X_{ij} X_{ik} - \sigma_{jk})  \right|,
		$$
		where
		$$
		b_{is} = \left\{
		\begin{array}{cc}
		s^{-1}, & \text{if } 1 \le i \le s \\
		0, & \text{if } s+1 \le i \le n \\
		\end{array} \right. .
		$$
		By Part (i) of Lemma \ref{lem:maximal_subexponential}, there exists a universal constant $K_1 > 0$ such that for all $t > 0$, 
		$$
		\Prob(\hat\Delta_1 \ge 2 \E [\hat\Delta_1] + t) \le \exp\left(-{t^2 \over 3\tau^2} \right) + 3 \exp \left[-\left({t \over K_1 \|M\|_{\psi_{1/2}}}\right)^{1/2} \right],
		$$
		where 
		\begin{eqnarray*}
			\tau^2 &=& \max_{\us \le s \le n-\us} \max_{1 \le j,k \le p} \sum_{i=1}^n b_{is}^2 \E(X_{ij} X_{ik} - \sigma_{jk})^2, \\
			M &=& \max_{\us \le s \le n-\us} \max_{1 \le j,k \le p} \max_{1 \le i \le s} | s^{-1} (X_{ij} X_{ik} - \sigma_{jk}) |.
		\end{eqnarray*}
		Note that 
		\begin{align*}
		\|M\|_2 \le K_2 \|M\|_{\psi_{1/2}} \le & {K_2 \over \us} \| \max_{1 \le i \le n}  \max_{1 \le j,k \le p} |X_{ij} X_{ik}| \|_{\psi_{1/2}} \le {K_2 \over \us} \|\max_{1 \le i \le n} \max_{1 \le j \le p} X_{ij}^2 \|_{\psi_{1/2}} \\
		= & {K_2 \over \us} \|\max_{1 \le i \le n} \max_{1 \le j \le p} |X_{ij}| \|_{\psi_1} \le {K_2 \over \us} \bar{b}^2 \log^2(np).
		\end{align*}
		By the Cauchy-Schwarz inequality and assumption (B), we have 
		$$
		\tau^2 \le \max_{\us \le s \le n-\us} \max_{1 \le j,k \le p} \sum_{i=1}^s {1 \over s^2}\E(X_{ij} X_{ik})^2 \le \bar{b}^2 \us^{-1}.
		$$
		By Lemma E.1 in \cite{cck2016a}, there exists a universal constant $K_3 > 0$ such that
		\begin{eqnarray*}
			\E[\hat\Delta_1] &\le& K_3 \left\{ \tau \log^{1/2}(np^2) + \|M\|_2 \log(np^2) \right\} \\
			&\le& K_3  \left\{ \bar{b} \log^{1/2}(np) \us^{-1/2} + \bar{b}^2 \log^3(np) \us^{-1} \right\}.
		\end{eqnarray*}
		Therefore, we get
		\begin{align*}
		& \Prob\left(\hat\Delta_1 \ge 2 K_3  \left\{ \bar{b} \log^{1/2}(np) \us^{-1/2} + \bar{b}^2 \log^3(np) \us^{-1} \right\} + t\right) \\
		& \qquad \qquad \le \exp\left(-{ t^2 \us \over 3 \bar{b}^2}\right) + 3 \exp\left(- {t^{1/2} \us^{1/2} \over K_4 \bar{b} \log(np)} \right).
		\end{align*}
		Choose
		$t = C_1  \us^{-1/2} \log^{1/2}(np)$
		for some large enough constant $C_1 := C_1(\bar{b}, K) \ge 1$. Using $\log(\gamma^{-1}) \le K \log(np)$ and $\log^5(np) \le \us$, we obtain that 
		\begin{equation}
		\label{eqn:hat_Delta_1_subexp}
		\Prob\left(\hat\Delta_1 \ge C  \us^{-1/2} \log^{3/2}(np) \right) \le \gamma / 4.
		\end{equation}
		Since $X_1,\dots,X_n$ are i.i.d. under $H_0$, $\hat\Delta_1$ and $\hat\Delta_2$ share the same distribution and therefore $\hat\Delta_2$ also obeys the bound (\ref{eqn:hat_Delta_1_subexp}). $\hat\Delta_3$ and $\hat\Delta_4$ can be dealt similarly. Indeed, by Lemma \ref{lem:maximal_subexponential}, there exists a universal constant $K_5 > 0$ such that for all $t > 0$, 
		$$
		\Prob(\hat\Delta_3 \ge 2 \E [\max_{\us \le s \le n-\us}  |\sum_{i=1}^n b_{is}X_{i}|_\infty \ ] + t) \le \exp\left(-{t^2 \over 3 \tilde\tau^2} \right) + 3 \exp \left(-{t \over K_5 \|\tilde M\|_{\psi_1}} \right),
		$$
		where $\tilde\tau^2 = \max_{\us \le s \le n-\us, 1 \le j \le p} \sum_{i=1}^n b_{is}^2 \E X_{ij}^2$ and $\tilde{M} = \max_{ 1 \le i \le n, \us \le s \le n-\us, 1 \le j \le p }   |b_{is}X_{ij}|$. By (B), $\tilde\tau^2 \le \bar{b} \us^{-1}$. By (C) and Lemma 2.2.2 in \cite{vandervaartwellner1996}, there exists a universal constant $K_6 > 0$ such that $\|\tilde{M}\|_{\psi_1} \le K_6 \bar{b} \log(np) \us^{-1}$. By Lemma E.1 in \cite{cck2016a}, there exists a universal constant $K_7 > 0$ such that
		\begin{eqnarray*}
			\E [\max_{\us \le s \le n-\us}  |\sum_{i=1}^n b_{si}X_{i}|_\infty \ ] &\le& K_7 \left\{ \tilde\tau \log^{1/2}(np) + \|\tilde{M}\|_2 \log(np) \right\} \\
			&\le& K_7  \left\{ \bar{b} \log^{1/2}(np) \us^{-1/2} + \bar{b} \log^2(np) \us^{-1} \right\}.
		\end{eqnarray*}
		So it follows that 
		\begin{align*}
		& \Prob\left(\hat\Delta_3 \ge 2 K_7   \left\{ \bar{b} \log^{1/2}(np) \us^{-1/2} + \bar{b} \log^2(np) \us^{-1} \right\} + t \right) \\
		& \qquad \qquad \le \exp\left(-{t^2 \us \over 3 \bar{b}} \right) + 3 \exp \left(-{t \us \over K_8 \bar{b} \log(np)} \right).
		\end{align*}
		Using $t = C  \us^{-1/2} \log^{1/2}(np) \log(\gamma^{-1})$, we get 
		\begin{equation*}
		\Prob\left(\hat\Delta_3 \ge C  \us^{-1/2} \log^{3/2}(np) \right) \le \gamma / 4.
		\end{equation*}
		
		\underline{\it{Part (ii)}}. Assume (D). By Part (ii) of Lemma \ref{lem:maximal_subexponential}, there exists a constant $C_1 := C_1(q) > 0$ such that for all $t > 0$, 
		$$
		\Prob(\hat\Delta_1 \ge 2 \E[\hat\Delta_1] + t) \le \exp\left(-{t^2 \over 3\tau^2} \right) + C_1 {\E[M^{q/2}] \over t^{q/2}},
		$$
		where $\tau^2$ and $M$ have the same definition as in Part (i). As in Part (i), $\tau^2 \le \bar{b}^2 \us^{-1}$. Note that 
		\begin{eqnarray*}
			\E[M^{q/2}] \le C_2(q) \us^{-q/2} \{ \E[\max_{1 \le i \le n} \max_{1 \le j \le p} |X_{ij}|^q] + \max_{1 \le j,k \le p} |\sigma_{jk}|^{q/2} \} \le C_2(q) \bar{b}^q n \us^{-q/2}.
		\end{eqnarray*}
		and 
		\begin{eqnarray*}
			\E[\hat\Delta_1] &\le& K_1 \{\tau \log^{1/2}(np^2) + \|M\|_{q/2} \log(np^2) \} \\
			&\le& K_1 \{ \bar{b} \us^{-1/2} \log^{1/2}(np) + C_2(q)^{2/q} \bar{b}^2 n^{2/q} \log(np) \us^{-1} \}.
		\end{eqnarray*}
		Therefore, there exists a constant $C_3(q) > 0$ such that 
		\begin{align*}
		& \Prob\left(\hat\Delta_1 \ge 2 K_1  \left\{ \bar{b} \us^{-1/2} \log^{1/2}(np) + C_2(q)^{2/q} \bar{b}^2 n^{2/q} \log(np) \us^{-1} \right\} + t\right) \\
		& \qquad \qquad \le \exp\left(-{ t^2 \us \over 3 \bar{b}^2}\right) + C_3(q) {n \bar{b}^q \over t^{q/2} \us^{q/2}}.
		\end{align*}
		Now, choosing $t = C_4 \{ \us^{-1/2} \log^{1/2}(np) + \gamma^{-2/q} \us^{-1} n^{2/q} \}$
		for some large enough constant $C_4 := C_4(\bar{b}, K, q) \ge 1$. Using $\log(\gamma^{-1}) \le K \log(np)$ and $\log^3(np) \le n$, we obtain that 
		\begin{equation}
		\label{eqn:hat_Delta_1_poly}
		\Prob\left(\hat\Delta_1 \ge C_5 \{\us^{-1/2} \log^{3/2}(np) + \gamma^{-2/q} \us^{-1} n^{2/q} \log(np) \} \right) \le \gamma / 4.
		\end{equation}
		Other terms $\hat\Delta_i, i =2,3,4$ can be similarly handled and details are omitted. 
	\end{proof}

	\begin{proof}[of Lemma \ref{lem:bound_bar_psi}]
		Recall that $\bar{X}_s^- = s^{-1} \sum_{i=1}^s X_i$, $\bar{X}_s^+ = (n-s)^{-1} \sum_{i=s+1}^n X_i$, and $\bar{\xi}_s^-, \bar{\xi}_s^+$ are similarly defined by replacing $X_1^n$ with $\xi_1^n$. Then, elementary calculations yield 
		$$
		\bar{X}_s^- = \left\{
		\begin{array}{ll}
		\bar{\xi}_s^-, & \text{if } 1 \le s \le m-1 \\
		\bar{\xi}_s^- + {s-m \over s} \delta_n, & \text{if } m \le s \le n-1 \\
		\end{array} \right. 
		$$
		and
		$$
		\bar{X}_s^+ = \left\{
		\begin{array}{ll}
		\bar{\xi}_s^+ + {n-m \over n-s} \delta_n, & \text{if } 1 \le s \le m-1 \\
		\bar{\xi}_s^+ + \delta_n, & \text{if } m \le s \le n-1 \\
		\end{array} \right. .
		$$
		Let 
		$$
		Z_n^{\xi,*}(s) = \sqrt{n-s \over n} \underbrace{s^{-1/2} \sum_{i=1}^s e_i (\xi_i - \bar{\xi}_s^-)}_{:= A} - \sqrt{s \over n} \underbrace{(n-s)^{-1/2} \sum_{i=s+1}^n e_i (\xi_i - \bar{\xi}_s^+)}_{:=B}
		$$ 
		be the bootstrap CUSUM statistic computed on the transformed data $\xi_1^n$. For $m+1 \le s \le n-1$, we define $b_{is} = -(s-m)/s$ if $1 \le i \le m$, $b_{is} = m / s$ if $m+1 \le i \le s$, and $b_{is} = 0$ otherwise. For $1 \le s \le m-1$, we define $b'_{is} = -(n-m)/(n-s)$ if $s+1 \le i \le m$, $b'_{is} = (m-s) / (n-s)$ if $m+1 \le i \le n$, and $b'_{is} = 0$ otherwise. Then routine algebra show that
		$$
		{1 \over \sqrt{s}} \sum_{i=1}^s e_i (X_i - \bar{X}_s^-) = \left\{
		\begin{array}{ll}
		A, & \text{if } 1 \le s \le m \\
		A + {\delta_n \over \sqrt{s}} \sum_{i=1}^s b_{is} e_i, & \text{if } m+1 \le s \le n \\
		\end{array} \right.
		$$
		and
		$$
		{1 \over \sqrt{n-s}} \sum_{i=s+1}^n e_i (X_i - \bar{X}_s^-) = \left\{
		\begin{array}{ll}
		B + {\delta_n \over \sqrt{n-s}} \sum_{i=s+1}^n b'_{is} e_i, & \text{if } 1 \le s \le m-1 \\
		B, & \text{if } m \le s \le n \\
		\end{array} \right. .
		$$
		Denote $\Cov_e(\cdot)$ as the covariance operator taken w.r.t. the random variables $e_1,\dots,e_n$ By the Cauchy-Schwarz inequality, we have 
		\begin{eqnarray*}
			\bar\psi^2 &=& \max_{\us \le s \le n-\us} \max_{1 \le j \le p} \Big\{ {n-s \over n} \Cov_e(A_j + {\delta_{nj} \over \sqrt{s}} \sum_{i=1}^s b_{is} e_i \vone_{(m+1 \le s \le n-1)} ) \\
			&& \qquad \qquad \qquad + {s \over n}  \Cov_e( B_j + {\delta_{nj} \over \sqrt{n-s}} \sum_{i=s+1}^n b'_{is} e_i \vone_{(1 \le s \le m-1)} ) \Big\} \\
			&\le& 2 \max_{\us \le s \le n-\us} \max_{1 \le j \le p} \Big\{ {n-s \over n} \Cov_e(A_j) + {s \over n}  \Cov_e( B_j) \Big\} \\
			&& + 2 |\delta_n|_\infty^2 \max_{\us \le s \le n-\us} \left[ {1 \over s} \sum_{i=1}^s b_{is}^2 \vone_{(m+1 \le s \le n-1)} \right] \\
			&& + 2 |\delta_n|_\infty^2 \max_{\us \le s \le n-\us} \left[ {1 \over n-s} \sum_{i=s+1}^n {b'}_{is}^2 \vone_{(1 \le s \le m-1)} \right].
		\end{eqnarray*}
		Note that
		\begin{eqnarray*}
			\sum_{i=1}^s b_{is}^2 &=& {m (s-m) \over s} \le s \qquad \text{for } m+1 \le s \le n-1, \\
			\sum_{i=s+1}^n {b'}_{is}^2 &=& {(n-m) (m-s) \over n-s} \le n-s \qquad \text{for } 1 \le s \le m-1.
		\end{eqnarray*}
		Then (\ref{eqn:bound_bar_psi}) is immediate.
	\end{proof}

	\begin{proof}[of Lemma \ref{lem:prob_event_G}]
		\underline{\it{Part (i)}}. If (C) holds,
		First, note that we can write $\tilde{Z}_n(s) = \sum_{i=1}^n v_{is} X_i$, where 
		\[
		v_{is} = \left\{
		\begin{array}{cc}
		{n-s \over n} & \text{if } 1 \le i \le s \\
		-{s \over n} & \text{if } s+1 \le i \le n\\
		\end{array}
		\right. .
		\] 
		Since $\xi_i = X_i - \E(X_i)$, we have $\tilde{Z}_{n}(s) - \E[\tilde{Z}_{n}(s)] = \sum_{i=1}^n v_{is} \xi_i$. By Lemma \ref{lem:maximal_subexponential}, we have
		\begin{align*}
		& \Prob \left( \max_{s \ge m+r} \max_{1 \le j \le p} \left| \sum_{i=1}^n v_{is} \xi_{ij} \right|  \ge 2 \E\left[  \max_{s \ge m+r} \max_{1 \le j \le p} \left| \sum_{i=1}^n v_{is} \xi_{ij} \right| \right] + t \right) \\
		& \quad \le \exp \left( - {t^2 \over 3 \tau^2} \right) + 3 \exp \left( -{t \over K_1 \|M\|_{\psi_1}}\right),
		\end{align*}
		where 
		\[
		\tau^2 = \max_{s \ge m+r} \max_{1 \le j \le p} \sum_{i=1}^n v_{is}^2 \E(\xi_{ij}^2) \quad \text{and} \quad M = \max_{1 \le i \le n} \max_{s \ge m+r} \max_{1 \le j \le p} |v_{is} \xi_{ij}|.
		\]
		Note that 
		\begin{align*}
		\tau^2 & \le \bar{b} \max_{s \ge m+r} \sum_{i=1}^n v_{is}^2 = {\bar{b} \over n} \max_{s \ge m+r} (n-s) s \le {\bar{b} n \over 4}, \\
		\|M\|_2 & = \left\| \max_{1 \le i \le n} \max_{1 \le j \le p} (\max_{s \ge m+r} |v_{is}|) |\xi_{ij}| \right\|_2 \le K_2 \left\| \max_{1 \le i \le n} \max_{1 \le j \le p} (\max_{s \ge m+r} |v_{is}|) |\xi_{ij}| \right\|_{\psi_1} \\
		& \quad \le K_2 \log(np) \max_{1 \le i \le n} \max_{1 \le j \le p}  (\max_{s \ge m+r} |v_{is}|) \left\| \xi_{ij} \right\|_{\psi_1} \le K_2 \bar{b} \log(np).
		\end{align*}
		Using Lemma E.2 in \cite{cck2016a} and $\log^3(np) \le K n$, we have 
		\begin{align*}
		\E\left[  \max_{s \ge m+r} \max_{1 \le j \le p} \left| \sum_{i=1}^n v_{is} \xi_i \right| \right] \le & K_3 \left\{ \sqrt{\log(np)} \tau + \log(np) \|M\|_2 \right\} \\
		\le & K_3 \left\{ \sqrt{\bar{b} n \log(np)} + \bar{b} \log^2(np) \right\} \\
		\le & C_1(\bar{b}, K) \sqrt{n \log(np)}.
		\end{align*}
		Thus we get 
		\begin{align*}
		& \Prob \left( \max_{s \ge m+r} \max_{1 \le j \le p} \left| \sum_{i=1}^n v_{is} \xi_i \right|  \ge  C_1(\bar{b}, K) \sqrt{n \log(np)}  + t \right) \\
		& \quad \le \exp \left( - {4 t^2 \over 3 \bar{b} n} \right) + 3 \exp \left( -{t \over K_4 \bar{b} \log(np)}\right).
		\end{align*}
		Choosing $t = C_2(\bar{b},K)  \sqrt{n \log(\gamma^{-1})}$ and using $\log(\gamma^{-1}) \le K \log(np)$, we have 
		\begin{align*}
		\Prob \Bigg( \max_{s \ge m+r} \max_{1 \le j \le p} \left|\tilde{Z}_{nj}(s) - \E[\tilde{Z}_{nj}(s)]  \right| \ge t^\dagger \Bigg) \le \gamma/36,
		\end{align*}
		where $t^\dagger = C_3(\bar{b}, K) \sqrt{n \log(np)}$. Note that for any two sequences $\{a_i\}$ and $\{b_i\}$, we have by the elementary inequality $| \max_i |a_i| - \max_i |b_i| | \le \max_i | |a_i|-|b_i| | \le \max_i | a_i-b_i |$ that 
		\begin{align*}
		& \Prob \left( \left| \max_{s \ge m+r} \max_{j \in \calS} \left|\tilde{Z}_{nj}(s) \right| - \max_{s \ge m+r} \max_{j \in \calS} \left|\E\tilde{Z}_{nj}(s) \right| \right| \ge t^\dagger \right) \le \gamma/24, \\
		& \Prob \left( \max_{s \ge m+r} \max_{j \in \calS^c} \left|\tilde{Z}_{nj}(s) \right| \ge t^\dagger \right) \le \gamma/36.
		\end{align*}
		Since  
		\begin{align*}
		\max_{s \ge m+r} \max_{j \in \calS} \left|\E\tilde{Z}_{nj}(s) \right| = \max_{s \ge m+r} \max_{j \in \calS} \left|{(n-s)m \over n} \delta_{nj} \right| = n t_m(1-t_m-t_r) |\delta_n|_\infty,
		\end{align*}
		it follows that 
		\begin{align*}
		|\delta_n|_\infty \ge {C_4 \log^{1/2}(np) \over n^{1/2}}
		\end{align*}
		for some large enough constant $C_4 = C_4(\bar{b}, K, c_1, c_2) > 0$ implies that
		\begin{align*}
		\Prob(\mathcal{G}^c) \le &\Prob \left( \max_{s \ge m+r} \max_{j \in \calS} \left|\tilde{Z}_{nj}(s) \right| - \max_{s \ge m+r} \max_{j \in \calS^c} \left|\tilde{Z}_{nj}(s) \right| \le 0 \right) \\
		\le& \Prob \left(\max_{s \ge m+r} \max_{j \in \calS} \left|\tilde{Z}_{nj}(s) \right| \le \max_{s \ge m+r} \max_{j \in \calS} \left|\E\tilde{Z}_{nj}(s) \right| - t^\dagger \right)\\
		&\qquad \qquad + \Prob \left( - \max_{s \ge m+r} \max_{j \in \calS^c} \left|\tilde{Z}_{nj}(s)\right| \le - t^\dagger \right) \\
		\le& \gamma/36 + \gamma/36 = \gamma/18.	
		\end{align*}
		
		In addition, following the same arguments, 
		\begin{align*}
		\Prob \Bigg(  \max_{1 \le j \le p} \left|\tilde{Z}_{nj}(m) - \E[\tilde{Z}_{nj}(m)]  \right| \ge t^\dagger \Bigg) \le \gamma/36.
		\end{align*}
		If $\delta_{nj} \le 0, \forall 1 \le j \le p$, then $\E[\tilde{Z}_{nj}(m)] = n t_m (1-t_m) |\delta_{nj}| \ge 0$, and 
		\begin{align*}
		\min_{1 \le j \le p} \tilde{Z}_{nj}(m) \ge - t^\dagger, \quad \max_{1 \le j \le p} \tilde{Z}_{nj}(m) \ge n t_m (1-t_m) |\delta_{n}|_\infty - t^\dagger \ge t^\dagger 
		\end{align*}
		with probability greater than $1-\gamma/36$ when $C_4 \ge 2 C_3$. In other words, $\left| \max_{1 \le j \le p} \tilde{Z}_{nj}(m) \right| \ge \left| \min_{1 \le j \le p} \tilde{Z}_{nj}(m) \right|$, which implies $ \max_{1 \le j \le p} \left| \tilde{Z}_{nj}(m) \right| =  \max_{1 \le j \le p} \tilde{Z}_{nj}(m) = \tilde{Z}_{nj^*}(m) \ge 0$. Therefore,
		\begin{align*}
		\Prob \left( \max_{1 \le j \le p} \left| \tilde{Z}_{nj}(m) \right| = \tilde{Z}_{nj^*}(m) \right) \ge 1 - \gamma/36.
		\end{align*}
		
		\underline{\it{Part (ii)}}. If (D) holds,
		By Lemma \ref{lem:maximal_subexponential}, we have
		\begin{align*}
		& \Prob \left( \max_{s \ge m+r} \max_{1 \le j \le p} \left| \sum_{i=1}^n v_{is} \xi_i \right|  \ge 2 \E\left[  \max_{s \ge m+r} \max_{1 \le j \le p} \left| \sum_{i=1}^n v_{is} \xi_i \right| \right] + t \right) \\
		& \quad \le \exp \left( -{t^2 \over 3 \tau^2} \right) + K_5 {\E \left[ M^q \right] \over t^q },
		\end{align*}
		where $\tau^2$ and $M$ have the same definition as in Part (i). Since $\tau^2 \le n \bar{b} /4$, we have 
		\begin{align*}
		\|M\|_q^q & = \E \left( \max_{1 \le i \le n} \max_{1 \le j \le p} (\max_{s \ge m+r} |v_{is}|^q) |\xi_{ij}|^q \right) \le \sum_{i=1}^n \E \left(  \max_{1 \le j \le p} |\xi_{ij}|^q \right) \le n\bar{b}^q,
		\end{align*}
		which implies that $\|M\|_2 \le \|M\|_q = n^{1/q} \bar{b}$ for $q \ge 2$. Using Lemma E.2 in \cite{cck2016a} and $\log^3(np) \le K n$, we have 
		\begin{align*}
		\E\left[  \max_{s \ge m+r} \max_{1 \le j \le p} \left| \sum_{i=1}^n v_{is} \xi_i \right| \right] \le & K_4 \left\{ \sqrt{\bar{b} n \log(np)} + \bar{b} n^{1/q}\log(np) \right\}.
		\end{align*}
		Thus we get 
		\begin{align*}
		& \Prob \left( \max_{s \ge m+r} \max_{1 \le j \le p} \left| \sum_{i=1}^n v_{is} \xi_i \right|  \ge  C_5(\bar{b},K) \left\{ \sqrt{ n \log(np)} +  n^{1/q}\log(np) \right\}  + t \right) \\
		& \quad \le \exp \left( - {4 t^2 \over 3 \bar{b} n} \right) + K_5 {n \bar{b^q} \over t^q }.
		\end{align*}
		Choosing $t = C_6(\bar{b},K,q)  \{\sqrt{n \log(\gamma^{-1})} + \gamma^{-1/q} n^{1/q} \}$ and using $\log(\gamma^{-1}) \le K \log(np)$, we have 
		\begin{align*}
		\Prob \Bigg( \max_{s \ge m+r} \max_{1 \le j \le p} \left|\tilde{Z}_{nj}(s) - \E[\tilde{Z}_{nj}(s)]  \right| \ge t^\dagger \Bigg) \le \gamma/36,
		\end{align*}
		where $t^\dagger = C_7(\bar{b}, K, q) \{\sqrt{n \log(np)} + \gamma^{-1/q} n^{1/q} \log(np)\}$. 
		If 
		\begin{align*}
		|\delta_n|_\infty \ge {C_8 \log^{1/2}(np) \over n^{1/2}} \max\left\{ 1, \gamma^{-1/q} n^{1/q-1/2} \log^{1/2}(np) \right\},
		\end{align*}
		for some large enough constant $C_8 = C_8(\bar{b}, K, q, c_1, c_2) > 0$, then it follows from the same argument as in Part (i) that $\Prob(\mathcal{G}^c) \le \gamma/18$. In addition, $\Prob \left( \max_{1 \le j \le p} \left| \tilde{Z}_{nj}(m) \right| = \tilde{Z}_{nj^*}(m) \right) \ge 1 - \gamma/36$.
	\end{proof}

	\begin{proof}[of Lemma~\ref{lem:randomness_ctrl}]
		\underline{\it{Part (i).} Assume (C).} Consider $\calA_n$ first.
		Apply Lemma~\ref{lem:maximal_subexponential} to $a_{i,s_1,s_2} = {1 \over \sqrt{s_2-s_1+1} } \vone_{\{s_1 \le i \le s_2\}}$ (i.e. $s = (s_1,s_2) \in \N^2$) and $X_{ij} = \xi_{ij}$, we have for $\forall t > 0$
		\begin{eqnarray}
		\nonumber
		&& \Prob(\max_{1 \le s_1 \le s_2 \le n} | \sum_{i=1}^{n} a_{i,s_1,s_2} \xi_i|_\infty \ge 2 \E[\max_{1 \le s_1 \le s_2 \le n} | \sum_{i=1}^{n} a_{i,s_1,s_2} \xi_i|_\infty ] + t) \\
		\label{eqn:location_rate_concentration_subexp_mcp}
		&& \qquad \le \exp\left(-{t^2 \over 3 \tau^2}\right) + 3 \exp\left(-{t \over K_1 \|M\|_{\psi_1}} \right),
		\end{eqnarray}
		where 
		\begin{eqnarray*}
			\tau^2 &=& \max_{1 \le s_1 \le s_2 \le n} \max_{1 \le j \le p} \sum_{i=1}^n a_{i,s_1,s_2}^2 \E \xi_{ij}^2, \\
			M &=& \max_{1 \le i \le n} \max_{1 \le s_1 \le s_2 \le n} \max_{1 \le j \le p} |a_{i,s_1,s_2} \xi_{ij}|.
		\end{eqnarray*}
		Since $\tau^2 \le \bar{b}$, and by Lemma 2.2.2 in \cite{vandervaartwellner1996}, we have $\|M\|_2 \le 2 \|M\|_{\psi_1} \le C \bar{b} \log(n^3p) \le C' \bar{b} \log(np)$. By Lemma E.1 in \cite{cck2016a}, we have 
		\begin{equation}
		\label{eqn:rate_location_expectation_bound_mcp}
		\E[\max_{1 \le s_1 \le s_2 \le n} | \sum_{i=1}^{n} a_{i,s_1,s_2} \xi_i|_\infty] \le K_2 \{ \tau \log^{1/2}(np) + \|M\|_2 \log(np) \}.
		\end{equation}
		Choosing $t = K_3 \bar{b} \log(np) \log(\gamma^{-1})$ in (\ref{eqn:location_rate_concentration_subexp_mcp}) for some large enough universal constant $K_3 > 0$, we deduce that there exists a constant $C := C(\bar{b},K) > 0$ such that 
		\[
		\Prob(\max_{1 \le s_1 \le s_2 \le n} |a_{i,s_1,s_2}  \sum_{i=1}^s \xi_i|_\infty \ge K \bar{b}  \log^2 (np) ) \le \gamma.
		\]
		So $\Prob (\calA_n) \ge 1- \gamma$. 
		
		For $\calB_n$, replacing $a_{i,s}$ by  $a_{i,s_1,s_2} = {1 \over \sqrt{s_2-s_1+1} } \vone_{\{s_1 \le i \le s_2\}}$ in our [45] and [46], we have 
		\[
		\Prob(\max_{1 \le s_1 \le s \le s_2 \le n} |Z_n(s) - \E[Z_n(s)]|_\infty \ge K \bar{b}  \log^2 (np) ) \le \gamma.
		\]
		for $\log(\gamma^{-1}) \le K  \log (np)$. Then, $\Prob (\calB_n) \ge 1- \gamma$.
		
		\underline{\it{Part (ii).} Assume (D).} Assume (D) with $q \ge 2$. 
		Similarly, apply Lemma~\ref{lem:maximal_subexponential} to $a_{i,s_1,s_2}$ and $X_{ij} = \xi_{ij}$, we have for $\forall t > 0$
		\begin{eqnarray}
		\nonumber
		&& \Prob(\max_{1 \le s_1 \le s_2 \le n} | \sum_{i=1}^{n} a_{i,s_1,s_2} \xi_i|_\infty \ge 2 \E[\max_{1 \le s_1 \le s_2 \le n} | \sum_{i=1}^{n} a_{i,s_1,s_2} \xi_i|_\infty ] + t) \\
		\label{eqn:location_rate_concentration_poly_mcp}
		&& \qquad \le \exp\left(-{t^2 \over 3 \tau^2}\right) + C(q) {\E[M^q] \over t^q},
		\end{eqnarray}
		where $\tau^2$ and $M$ have the same definitions as in Part (i). By Lemma 2.2.2 in \cite{vandervaartwellner1996}, we have $\|M\|_2 \le \|M\|_q \le n^{3/q} \bar{b}$. Note that $\tau^2 \le \bar{b}$ and $\E[\max_{1 \le s \le n} |Z_n(s) - \E[Z_n(s)]|_\infty]$ obeys the bound in (\ref{eqn:rate_location_expectation_bound_mcp}). Hence, choosing $t = C(q) \{ \bar{b}^{1/2} \log(\gamma^{-1}) +  \bar{b} n^{3/q}  \gamma^{-1/q} \}$ in (\ref{eqn:location_rate_concentration_poly_mcp}), we get 
		$$
		\Prob\left(\max_{1 \le s \le n} |Z_n(s) - \E[Z_n(s)]|_\infty \ge C(q) \bar{b} n^{3/q} (\log(np) + \gamma^{-1/q}) \right) \le 2 \log^{-q}(np).
		$$
		
	\end{proof}

	\begin{proof}[of Lemma~\ref{lem:max_delta_s_location}]
		Denote the CUSUM mean computed on $X_1^n$ as 
		\[
		\Delta_s = \E Z_{n}(s) = -\sum_{k=1}^{\nu} \left[ \sqrt{s \over n(n-s)} (n-m_k) \ \delta^{(k)} \ \vone_{\{s \le m_k\}} + \sqrt{(n-s) \over ns} m_k  \ \delta^{(k)} \ \vone_{\{s > m_k\}} \right].
		\]
		In particular, for $m_k < s \le m_{k+1}$ of $k = 0, \dots, \nu$,
		\begin{equation}
		\label{eqn:Delta_s_mcp}
		\Delta_s = - \sqrt{s(n-s) \over n} \left( {m_0 \delta^{(0)} + \dots + m_k \delta^{(k)} \over s } + { (n- m_{k+1}) \delta^{(k+1)} + \dots + (n-m_{\nu+1}) \delta^{(\nu+1)} \over n-s }  \right).
		\end{equation}
		Let $\Delta (s)$ be defined as the same expression of (\ref{eqn:Delta_s_mcp}) but on the whole real numbers $s \in (1,n)$. 
		
		\underline{Step 1: Suppose time of changes $\nu = 2$ and data are univariate $p=1$.}
		
		\begin{equation}  
		f (s) := {d  \over ds}\Delta(s) = \left\{  
		\begin{aligned}
		&- {1 \over 2} \sqrt{n \over s(n-s)^3} \left( (n-m_1) \delta^{(1)} + (n-m_2) \delta^{(2)} \right), s < m_1  \\  
		& \ {1 \over 2} \sqrt{n \over s(n-s)} \left( {m_1 \over s} \delta^{(1)} - {n-m_2 \over n-s} \delta^{(2)} \right), m_1 < s < m_2 \\  
		& \ {1 \over 2} \sqrt{n \over s(n-s)^3} \left( m_1 \delta^{(1)} + m_2 \delta^{(2)} \right), s>m_2     
		\end{aligned}  
		\right.  
		\end{equation}  
		
		(i) Suppose $\sign(\delta^{(1)}) = \sign(\delta^{(2)}) \neq 0$, then the sign of the first derivative of $\Delta(s)$ is summarized as the table below, where $s_0$ satisfies $s_0/(n-s_0) = (m_1 \delta^{(1)}) / ((n-m_2)\delta^{(2)}) $ when $|\delta^{(1)}| / |\delta^{(2)}| \in ( {n-m_2 \over n-m_1}, {m_2 \over m_1}) $.
		\begin{table}
			\begin{tabular}{c|c||c|c|c|c}
				\hline
				\multicolumn{2}{c||}{$\sign(f(s))$}                                                & $s<m_1$                                 & \multicolumn{2}{c|}{$m_1<s<m_2$}                & $m_2<s$                                \\ \hline
				\multirow{4}{*}{$|\delta^{(1)}| \over  |\delta^{(2)}|$} & $(0, {n-m_2 \over n-m_1})$                                 & $-\sign (\delta^{(1)})$                 & \multicolumn{3}{c}{$\sign(\delta^{(1)})$}                                               \\ \cline{2-6} 
				& \multirow{2}{*}{$( {n-m_2 \over n-m_1}, {m_2 \over m_1})$} & \multirow{2}{*}{$-\sign(\delta^{(1)})$} & $m_1<s<s_0$           & $s_0<s<m_2$            & \multirow{2}{*}{$\sign(\delta^{(1)})$} \\ \cline{4-5}
				&                                                            &                                         & $\sign(\delta^{(1)})$ & $-\sign(\delta^{(1)})$ &                                        \\ \cline{2-6} 
				& $({m_2 \over m_1}, + \infty)$                               & \multicolumn{3}{c|}{$-\sign(\delta^{(1)})$}                                               & $\sign(\delta^{(1)})$                  \\ \hline
			\end{tabular}
		\end{table}
		Observing that $\Delta(s)$ is continuous and always with sign of $-\sign(\delta^{(1)})$. The maximum of $|\Delta(s)|$ locates at either $m_1$ or $m_2$. 
		
		(ii) Suppose $\sign(\delta^{(1)}) = -\sign(\delta^{(2)}) \neq 0$, then based on the summarized $\sign(f(s))$ in the table below and the fact that $\Delta(s)$ is continuous and always with sign $-\sign(\delta^{(1)})$, we claim the maximum of $|\Delta(s)|$ locates at either $s=m_1$ or $s=m_2$. 
		
		\begin{table}
			\begin{tabular}{c|c||c|c|c}
				\hline
				\multicolumn{2}{c||}{$\sign(f(s))$}                                                                  & $s<m_1$                & $m_1<s<m_2$           & $m_2<s$                \\ \hline
				\multirow{3}{*}{$|\delta^{(1)}| \over  |\delta^{(2)}|$} & $(0, {n-m_2 \over n-m_1})$                & \multicolumn{2}{c|}{$\sign (\delta^{(1)})$}     & $-\sign(\delta^{(1)})$ \\ \cline{2-5}
				& $( {n-m_2 \over n-m_1}, {m_2 \over m_1})$ & $-\sign(\delta^{(1)})$ & $\sign(\delta^{(1)})$ & $-\sign(\delta^{(1)})$ \\ \cline{2-5} 
				& $({m_2 \over m_1}, + \infty)$              & $-\sign(\delta^{(1)})$  & \multicolumn{2}{c}{$\sign(\delta^{(1)})$}   \\ \hline
			\end{tabular}
		\end{table}

		\underline{Step 2: Generalization to multiple change points $\nu \ge 2$ and multiple dimension $p \ge 1$.}
		Based on arguments in Step 1, for any $s \in (m_k, m_{k+1})$ of $k = 1, \dots, \nu-1$, the maximum of $Z_n(s)$ happens at boundaries of intervals, i.e. $\{m_1, \dots, m_\nu\}$. Notice that when $s \in (m_k, m_{k+1})$ for $k=0, \nu$,  it reduces to the case of single change point where $Z_n(s)$ reaches its maximum at $m_1$ or $m_\nu$. In addition, 
		\[
		\underset{s=1, \dots, n-1}{\argmax} |\Delta_s|_\infty \subset \underset{j=1, \dots, p}{\cup} \underset{s}{\argmax} |\Delta_j (s)| \subset \{m_k, k = 1, \dots, \nu\}
		\]
		Therefore, $\argmax_{s=1, \dots, n-1} |\Delta_s|_\infty \subset \{m_k, k = 1, \dots, \nu\}$ for $\nu \ge 1$ and $p \ge 1$.
		
		\underline{Step 3: Characterization of signal size $\max_{k = 1, \dots, \nu} |\Delta_{m_k}|_\infty$.} 
		Suppose $p=1$. Let $\nu \ge 2$ otherwise it reduces to single change point case which obeys the lower bound in Lemma (see (\ref{eqn:Delta_s}) in the proof of Theorem \ref{thm:power_bootstrap_cusum_test}). Denote $L = \sum_{k=1}^l m_k \delta^{(k)}, R = \sum_{k=l+1}^\nu (n-m_k) \delta^{(k)}$ for some $1\le l \le \nu$. Then by (\ref{eqn:Delta_s_mcp}),
		\[
		\left\{
		\begin{array}{l}
		\Delta_{m_{l}} = - {1 \over \sqrt{n}} \underbrace{\left[ \sqrt{n-m_l \over m_l} L + \sqrt{m_l \over n-m_l} R \right]}_{:=A} \\
		\Delta_{m_{l+1}} =  - {1 \over \sqrt{n}} \underbrace{\left[ \sqrt{n-m_{l+1} \over m_{l+1}} L + \sqrt{m_{l+1} \over n-m_{l+1}} R \right]}_{:=B} \\
		\end{array} \right. .
		\]
		Since $A$ and $B$ cannot be 0 at the same time, we will show that (i) when $L$ and $R$ have the same sign, $|A|$ itself is large and (ii) when $L$ and $R$ have the opposite sign, $|A-B|$ is still large even if signal cancellation exists. Therefore,  $\max \{O(|A|), O(|B|)\} ) = O(|A-B|)$ can be bounded below in either case, and so does $\max_{l=1, \dots, \nu} |\Delta_{m_l}|$.
		
		First, suppose $L$ and $R$ have the opposite sign. WLOG, let $L>0$. 
		\begin{align*}
		|A-B| &= \left( \sqrt{n-m_l \over m_l} - \sqrt{n-m_{l+1} \over m_{l+1}} \right) |L| + \left( \sqrt{m_l \over n-m_l} -\sqrt{m_{l+1} \over n-m_{l+1}} \right) |R| \\
		&= {n (m_{l+1} - m_{l}) / \sqrt{m_{l+1} m_{l}}  \over \sqrt{m_{l+1} (n-m_{l})}  \sqrt{m_{l1} (n-m_{l+1})}} \ |L| +  {n (m_{l+1} - m_{l}) / \sqrt{ (n- m_{l+1}) (n- m_{l})} \over  \sqrt{m_{l+1} (n-m_{l})}  \sqrt{m_{l1} (n-m_{l+1})}} \ |R| \\
		&\ge  { m_{l+1} - m_{l} \over 2 m_{l+1}}{ n  \over \sqrt{m_{l} (n-m_{l})} } \ |L| + { m_{l+1} - m_{l} \over 2 (n-m_{l}) }{ n  \over \sqrt{m_{l+1} (n-m_{l+1})} } \ |R|
		\end{align*}
		Denote $\udelta = \min_{k = 1, \dots, \nu} |\delta^{(k)}|$. Claim that we can find an $l$ such that $\max\{|L|, |R| \} \ge C_1 D_v \udelta$ for some constant $C_1$. Otherwise, we just need to move $l$ to $l+1$ such that $L$ and $R$ change $m_{l+1} \delta^{(l+1)} \ge D_\nu \udelta$ and $(n-m_{l+1}) \delta^{(l+1)} \ge D_\nu \udelta$, respectively, in opposite direction so that they still have opposite signs.  (Move $l$ to $l-1$ if $l$ is the last change point.) 
		Note that $\max\{  { m_{l+1} - m_{l} \over 2 m_{l+1}}, { m_{l+1} - m_{l} \over 2 (n-m_{l}) }\} \ge {D_\nu \over n}$. Consequently, we have $|A-B| \ge C_1 {D_\nu^2 \udelta \over \sqrt{m_{l} (n-m_{l})}}$. 
		Next, suppose $L$ and $R$ have the same sign. 
		We still have $\max\{|L|, |R| \} \ge C_2 D_v \udelta$ for some constant $C_2$. Otherwise, we move $l$ to $l+1$ and it reduces to the case above. Then, $|A| \ge C_2 (\sqrt{n-m_l \over m_l} \wedge \sqrt{m_l \over n-m_l}) D_v \udelta \ge C_2 {D_\nu^2 \udelta \over \sqrt{m_{l} (n-m_{l})}}$.
		Therefore, when $p=1$, there exists some constant $C$ such that  
		$
		\max_{l=1, \dots, \nu} |\Delta_{m_l}| \ge \max_{l=1, \dots, \nu} {C D_\nu^2 \udelta \over \sqrt{n m_{l} (n-m_{l})}}.
		$
		
		For the case of multiple dimension $p>1$, recall $\bar\delta_n = \min_{k=1, \dots, \nu} |\delta_n^{(k)}|_\infty$. Since for the dimension $h = \argmax_{j = 1,\dots, p} |\delta_{n,j}^{(l+1)}|$, when $l$ moves to $l+1$ the terms, $L_h$ and $R_h$ changes in opposite direction with distances at least $D_v |\delta_h^{(l+1)}| \ge D_v \bar\delta_n$.  Then, $\max\{|L|_\infty, |R|_\infty \} \ge C_3 D_v \bar\delta_n$ still holds. So based on similar argument, we can conclude that
		$$
		\max_{l=1, \dots, \nu} |\Delta_{m_l}|_\infty \ge \max_{l=1, \dots, \nu} {C D_\nu^2 \bar\delta_n \over \sqrt{n m_{l} (n-m_{l})}}.
		$$
		
	\end{proof}


	\section{Additional simulation results}
	\label{sec:addtional_sim_results}
	
	\subsection{Comparison of our bootstrap CUSUM test with other methods under $H_0$.}
	\label{subsec:compare_H0_benchmark_simu}
	Under the same set up as in Section~\ref{subsubsec:size_comparison} , we first compare our bootstrap CUSUM test with two benchmark methods in the following.
	
	\begin{enumerate}[leftmargin=1cm,itemindent=.5cm,labelwidth=\itemindent,align=parleft, label=(\roman*)]
		\item The bootstrapped log-ratio of maximized likelihood test (denoted as logLik) corresponds to $\log(\Lambda_{s})$ in \eqref{eqn:max_likelihood_ratio} when $p < n$ and
		$$\log (\Lambda_s^*) = Z_n^*(s)^T \hat{\Sigma}^{-1} Z_n^*(s)/2,$$
		where $\hat{\Sigma}$ is the sample covariance matrix of $X_{1}^{n}$ and $Z_n^*(s)$ is the bootstrap CUSUM statistic (\ref{eqn:bootstrap_cusum}). The testing procedure is similar to our bootstrap CUSUM test: reject $H_0$ if $\log(\Lambda_{s})$ is greater than the $(1-\alpha)$-th quantile of bootstrapped $\log (\Lambda_s^*)$.
		
		\item The oracle test is based on $\bar{Y}_n = |Y_n|_\infty$ , i.e., the Gaussian approximation of $T_n$, with {\it known} covariance matrix. By our definition (\ref{eqn:gaussian_analog_cusum_statistic}), $Y_n$ is a Gaussian copy of $\vec(\vZ_n) = (Z_n(\us)^\top, \dots, Z_n(n-\us)^\top)^\top$, whose covariance is close to $\Cov(Z_n^* | X_1^n)$. The oracle test rejects $H_0$ if $T_n$ is greater than the $(1-\alpha)$-th quantile of $\bar{Y}_n$. 
	\end{enumerate}
	%
	Table~\ref{tab:our_vs_benchmarks_size} lists the uniform error-in-size $\sup_{\alpha \in (0,1)} |\hat{R}(\alpha) - \alpha|$ for our test and the two benchmarks.  Each column corresponds to a combination of one distribution family and one cross-sectional dependence structures, and the rows compare the logLik (when $p<n$), our proposed bootstrap method with $\us=30, 40$ and the corresponding benchmark $\bar{Y}_n$ under different settings of $p=10,300,600$. 
	There are several observations we can draw. First, our proposed test has much smaller errors than logLik even when $p=10$ and $300$. 
	Second, the errors of our tests are comparable with those from corresponding $\bar{Y}_n$ in all settings. It shows that our bootstrapped $T_n^*$ is remarkably close to $\bar{Y}_n$ that is generated with given $\Sigma$.
	So our CUSUM test works well as a fully data-dependent approach compared to benchmarks.

	\begin{table}
		\caption{\label{tab:our_vs_benchmarks_size} Comparison of uniform error-in-size $\sup_{\alpha \in [0,1]} |\hat{R}(\alpha) - \alpha|$ under $H_0$ among our test and two benchmarks. (Parameters: $n=500$ $p=10,300,600$, $\us=30, 40$) and data are simulated from all combinations of distribution families and covariance dependence structures.)}
		\centering
		\fbox{
		\begin{tabular}{l|l|lll|lll|lll}
			\multicolumn{2}{c|}{\multirow{2}{*}{$n=500$}} & \multicolumn{3}{c|}{Gaussian} & \multicolumn{3}{c|}{$t_6$} & \multicolumn{3}{c}{ctm-Gaussian} \\ \cline{3-11} 
			\multicolumn{2}{c|}{}          & I        & II       & III     & I       & II      & III    & I         & II        & III       \\ \hline
			\multirow{5}{*}{$p=10$}        & logLik       & 0.150    & 0.151   & 0.159   & 0.122   & 0.134   & 0.124  & 0.136     & 0.134     & 0.124     \\ \cline{2-11} 
			& $\us=30$       & 0.034    & 0.036    & 0.041   & 0.048   & 0.041   & 0.039  & 0.036     & 0.042     & 0.021     \\
			& $\bar{Y}_n$        & 0.037    & 0.041    & 0.022   & 0.030   & 0.030   & 0.038  & 0.035     & 0.029     & 0.052     \\ \cline{2-11} 
			& $\us=40$       & 0.042    & 0.034    & 0.037   & 0.043   & 0.037   & 0.033  & 0.041     & 0.042     & 0.043     \\
			& $\bar{Y}_n$        & 0.028    & 0.041    & 0.030   & 0.028   & 0.053   & 0.044  & 0.023     & 0.033     & 0.035     \\ \hline
			\multirow{5}{*}{$p=300$}       & logLik       & 0.837    & 0.835    & 0.834   & 0.680   & 0.669   & 0.670  & 0.691     & 0.679     & 0.681     \\ \cline{2-11} 
			& $\us=30$       & 0.054    & 0.051    & 0.050   & 0.085   & 0.036   & 0.049  & 0.115     & 0.025     & 0.065     \\
			& $\bar{Y}_n$        & 0.024    & 0.046    & 0.039   & 0.057   & 0.064   & 0.066  & 0.044     & 0.067     & 0.036     \\ \cline{2-11} 
			& $\us=40$       & 0.046    & 0.026    & 0.035   & 0.058   & 0.030   & 0.040  & 0.057     & 0.032     & 0.055     \\
			& $\bar{Y}_n$        & 0.033    & 0.025    & 0.066   & 0.060   & 0.033   & 0.064  & 0.045     & 0.025     & 0.061     \\ \hline
			\multirow{4}{*}{$p=600$}       & $\us=30$       & 0.051    & 0.035    & 0.048   & 0.122   & 0.044   & 0.088  & 0.103     & 0.030     & 0.096     \\
			& $\bar{Y}_n$        & 0.030    & 0.045    & 0.049   & 0.048   & 0.043   & 0.068  & 0.039     & 0.055     & 0.051     \\ \cline{2-11} 
			& $\us=40$       & 0.060    & 0.055    & 0.046   & 0.083   & 0.038   & 0.087  & 0.079     & 0.026     & 0.057     \\
			& $\bar{Y}_n$        & 0.041    & 0.041    & 0.042   & 0.041   & 0.025   & 0.049  & 0.034     & 0.038     & 0.077     \\ 
		\end{tabular}
	}
	\end{table}

	Next, we compare our method with the tests in \cite{jirak2015} and \cite{enikeevaharchaoui2014} under the setting $n=500, p=600, \us=40$ under $H_0$. In \cite{jirak2015}, change points are allowed to occur at different locations in each coordinate. 
	To avoid simulation issue in estimating the long-run variance (which is necessary because \cite{jirak2015} uses non-stationary CUSUM statistics), we employ the true variance $\hat\sigma_h^2 = \sigma_h^2$ in their test statistic \citep[Equation (1.2)]{jirak2015} and take the suggested multiplier $\xi_l^2=1$ in $\hat{s}_h^2$ \citep[Equation (4.5)]{jirak2015} when calculating conditional long-run variance in bootstrap. This modified stronger algorithm is denoted by their test statistic $B_n$.
	In observation of the fact that long-run variances for all coordinates are the same in our simulation setting, we also enhanced it to another competitor denoted as $B_n$-enhanced by removing both $\hat\sigma_h^2$ and $\hat{s}_h^2$. 
	In \cite{enikeevaharchaoui2014}, the original test without boundary removing and an improved version with $\us=40$ are implemented (denoted as their test statistic $\psi$ and $\psi$-improved, respectively).
	We set the tuning parameter $\kappa$ to be 6.6 as suggested by the authors.
	Similar to Table~\ref{tab:our_size_and_nominal05}, uniform error-in-size empirical type I error at $\alpha = 0.05$ are shown in Table~\ref{tab:compareH0_size_nominal05}, where we draw the following comparison results. First, the test $\psi$ in \cite{enikeevaharchaoui2014} suffers from size distortion except for the case of Gaussian distribution with identity covariance structure because the $\psi$ relies heavily on the assumption $X_i \sim N(\vzero,\sigma^2 \Id_{p})$. 
	Second, boundary removal helps $\psi$ to reduce the uniform error-in-size, cf.\ the $\psi$-improved. 
	Third, the test $B_n$ in \cite{jirak2015} and $B_n$-enhanced behave similarly and they are comparable with ours. However, note that the two $B_n$ tests receive stronger priori that coordinate-wise (long-run) variances are all equal or assumed given, which is impractical. On the contrary, our proposed method does not involve such estimators of variance. 
	Lastly, in lower Table~\ref{tab:compareH0_size_nominal05}, the test $B_n$ inflates Type-I error than pre-specified 5\% while ours behaves conversely.

	\begin{table}
		\caption{\label{tab:compareH0_size_nominal05} Uniform error-in-size $\sup_{\alpha \in [0,1]} |\hat{R}(\alpha) - \alpha|$ and empirical type I error $\hat{R}(0.05)$ for \cite{jirak2015} ($B_{n}$ and $B_{n}$-enhanced) and \cite{enikeevaharchaoui2014} ($\psi$ and $\psi$-improved) under $H_0$. (Parameters: $n=500, p=600, \us = 1,40$ and data are simulated from all combinations of distribution families and covariance dependence structures.)}
		\centering
		\fbox{
		\begin{tabular}{ll|lll|lll|lll}
			\multicolumn{2}{c|}{\multirow{2}{*}{}} & \multicolumn{3}{c|}{Gaussian} & \multicolumn{3}{c|}{$t_6$} & \multicolumn{3}{c}{ctm-Gaussian} \\ \cline{3-11} 
			\multicolumn{2}{c|}{}                                & I        & II       & III     & I       & II      & III    & I         & II        & III       \\ \hline
			\multicolumn{11}{c}{$\sup_{\alpha \in [0,1]} |\hat{R}(\alpha) - \alpha|$} \\\hline
			$B_n$                        & $\us=40$              & 0.073    & 0.026    & 0.049   & 0.101   & 0.019   & 0.087  & 0.083     & 0.017     & 0.068     \\
			$B_n$-enhanced               & $\us=40$              & 0.071    & 0.025    & 0.063   & 0.063   & 0.041   & 0.068  & 0.061     & 0.016     & 0.067     \\ \hline
			$\psi$                       & $\us=1$               & 0.368    & 0.988    & 0.882   & 0.990   & 0.990   & 0.990  & 0.990     & 0.990     & 0.990     \\
			$\psi$-improved              & $\us=40$              & 0.186    & 0.959    & 0.646   & 0.990   & 0.990   & 0.990  & 0.990     & 0.987     & 0.990     \\ \hline
			\multicolumn{11}{c}{$\hat{R}(0.05)$} \\\hline
			$B_n$                        & $\us=40$              & 0.054    & 0.073    & 0.063   & 0.066   & 0.060   & 0.059  & 0.056     & 0.057     & 0.063     \\
			$B_n$-enhanced               & $\us=40$              & 0.057    & 0.065    & 0.061   & 0.038   & 0.052   & 0.054  & 0.061     & 0.059     & 0.058     \\ \hline
			$\psi$                       & $\us=1$               & 0.110    & 0.998    & 0.923   & 1       & 1       & 1      & 1         & 1         & 1         \\
			$\psi$-improved              & $\us=40$              & 0.055    & 0.973    & 0.696   & 1       & 1       & 1      & 1         & 0.997     & 1         \\ 
		\end{tabular}
	}
	\end{table}

\subsection{Additional results on RMSE for location estimators}
\label{subsec:rmse_Supplementary material}

Figure~\ref{fig:loc_compare} illustrates contrasts among the three algorithms. 
Figure~\ref{fig:loc_compare_subfig:a} compares our non-truncated $t_{\hat{m}_0}$ with \cite{wangsamworth2017} which has no boundary removal. On the left plot (where $t_m=3/10$ is fixed), the RMSEs of \cite{wangsamworth2017} are large when signal is sparse ($k=1$) or data are non-Gaussian distributed with cross-sectional dependent. However, our $t_{\hat{m}_0}$  can identify sparse signal under $t_6$ distribution with high cross-sectional correlation and performs even comparably to their estimator when there is dense signal ($k=50$). On the right plot (where distribution is fixed as ctm-Gaussian with cross-sectional structure III), both approaches have large RMSEs when the change point is close to boundary. But ours is uniformly better under the corresponding sparse alternative. 
Figure~\ref{fig:loc_compare_subfig:b} compares our truncated $t_{\hat{m}_{1/2}}$ with \cite{chofryzlewicz2015} which also has boundary removal in their R package. We set $\us = 40$ for both algorithms. Their method works well for the case of $t_m=1/10$ as shown in the left of Figure~\ref{fig:loc_compare_subfig:b}, and distribution has less effect on \cite{chofryzlewicz2015} than on \cite{wangsamworth2017}. But when signal size is 1 and larger, our $t_{\hat{m}_{1/2}}$ almost dominates solid-symbol curves that are under the same distribution. Corresponding to the non-monotone RMSE issue mentioned in simulation, the right of Figure~\ref{fig:loc_compare_subfig:b} shows similar phenomenon even we choose their threshold parameter by using upper $\alpha=0.01$ quantile from their bootstrap sample rather than default $\alpha= 0.05$ in their R package. Under dense alternative, their searching process fail to discover a change when signal is strong.

\begin{figure}
	\centering
	\subfigure[RMSEs of \texttt{Inspect} and our non-truncated $t_{\hat{m}_0}$ ($\us=1$). Left: distribution effect for $t_m = 3/10$; Right: location effect for data from ctm-Gaussian distribution with Covariance II.]{
		\includegraphics[trim = 0 0 0 50, clip,width=0.7\textwidth,height=0.27\textheight]{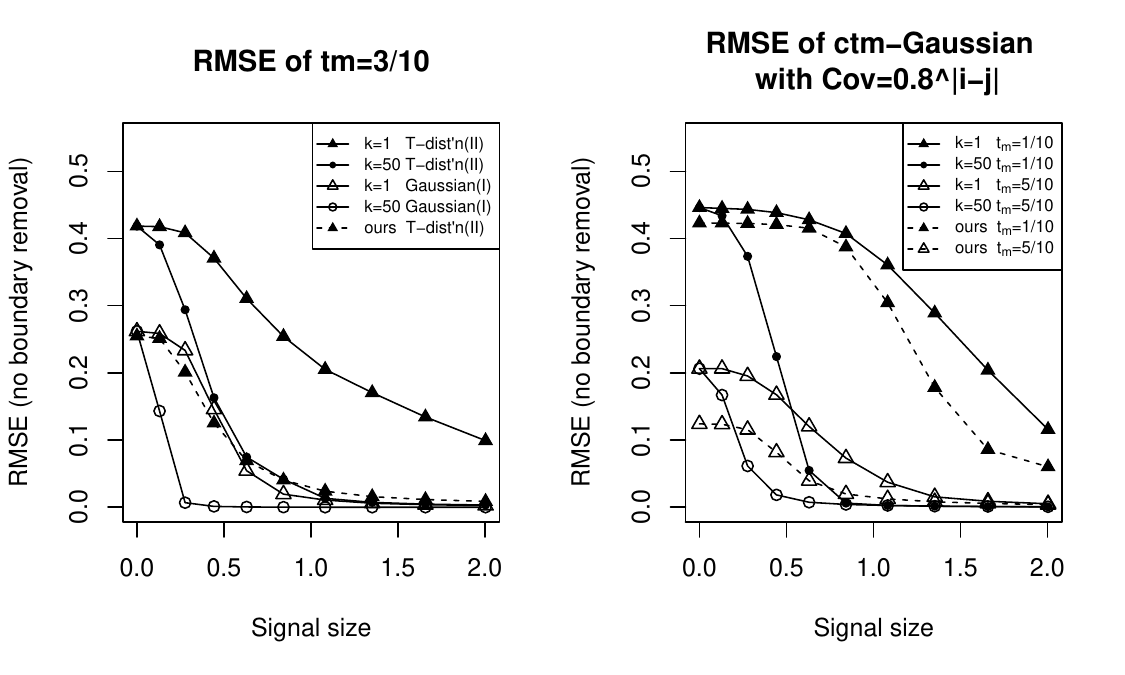}
		\label{fig:loc_compare_subfig:a}}
	\vspace{-0.1in}
	\subfigure[RMSEs of \texttt{SBS} and our truncated $t_{\hat{m}_{1/2}}$ ($\us=40$ for both methods). Left: distribution effect for $t_m = 1/10$; Right: non-monotone RMSE for $t_m = 5/10$, where data are from ctm-Gaussian distribution with Covariance II.]{
		\includegraphics[trim = 0 0 280 50, clip, width=0.35\textwidth,height=0.27\textheight]{RMSE_Choeg2.pdf}
		\includegraphics[trim = 0 0 0 50, clip,width=0.35\textwidth,height=0.27\textheight]{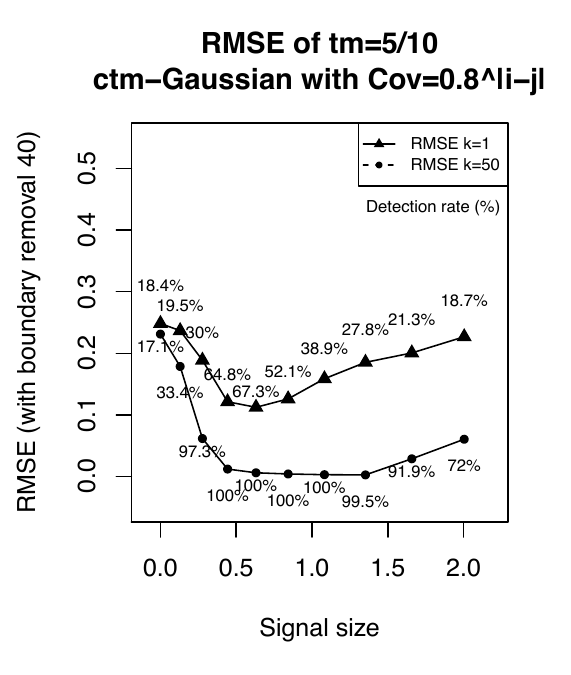}
		\label{fig:loc_compare_subfig:b}}
	\caption{Comparison of RMSEs among \texttt{Inspect}, \texttt{SBS} and ours (under sparse signal only). Parameters: $n=500, p=600$ and signal sparsity $k=1,50$.}
	\label{fig:loc_compare} 
\end{figure}


\subsection{Extension to time series: a block multiplier bootstrap}\label{subsec:extension_block_mb}	
\label{subsec:time_series_block_bootstrap_CUSUM_test}

	We shall study the empirical performance of the block multiplier bootstrap CUSUM test for some dependent process $\xi_{i}$. Theoretical supports for the block bootstrap CUSUM test are beyond the scope of this paper and they can be derived using the recent development of the Gaussian and bootstrap approximation results for high-dimensional time series (see e.g.,~\cite{zhangcheng2014,zhangwu2016a}).
	In our simulation, we consider the stationary vector autoregression of order 1 (denote as VAR(1)) error process:
	$
	\xi_i = A\xi_{i-1} + \eta_i = \sum_{k=0}^{\infty} A^{k} \eta_{i-k},
	$ 
	where $\{\eta_{i}\}_{i \in \mathbb{Z}}$ is a sequence of i.i.d.\ mean-zero random vectors in $\mathbb{R}^{p}$ and $A$ is a $p \times p$ coefficient matrix, where random matrix $A$ is generated with i.i.d.\ $N(0,1)$ entries. To ensure the stationarity of $\xi_{i}$ process, $A$ is normalized such that $\|A\|_2=1/1.8<1$. In this section, we fix $n=500, p=600, \us=40$, and consider different block sizes $M=2, 5, 10, 15$. 
	
	We first investigate the performance of the modified block Gaussian multiplier bootstrap CUSUM test. Since distributional impact has already been evaluated in Section~\ref{subsec:simulation_single_cp}, the same setup in Figure~\ref{fig: AlphaEg}  are selected as example to illustrate the impact of $M$. In Figure~\ref{fig: AlphaEg_TS}, the $\hat{R}(\alpha)$ curves show similar (and slightly less accurate) behavior as in temporally independent case. The approximation accuracy also depends on the block size parameter $M$ (which adjusts for the temporal dependence), in addition to cross-sectional dependence. 

	The block bootstrap algorithm $B_n$ \citep{jirak2015} delivers similar message in Figure~\ref{fig:JirakTS}. Note that, we primarily compare the performances of bootstrap testing instead of estimation of $\hat{\sigma}_h^2$ that is an influential factor in practice. So the long-run variance estimator $\hat{\sigma}_h^2$ is substituted by its theoretical value ${\sigma}_h^2, h=1,\cdots,p$, i.e.\ the $h$-th diagonal element of $\Sigma(0) + \sum_{l=1}^\infty A^l \Sigma(0) + \Sigma(0) \sum_{l=1}^\infty (A^T)^l$ where $\Sigma(0) = \sum_{l=0}^\infty A^l \Sigma (A^T)^l$ is the lag-0 auto-covariance of $\{X_i\}$.
	The size approximation is accurate under spatial independent scenarios $V=Id_p$ in general, while larger block size $M$ is suggested when $V=0.8J+0.2Id_p$.

	\begin{figure}
		\centering
		\includegraphics[trim = 0 20 10 30, clip, width= 5in]{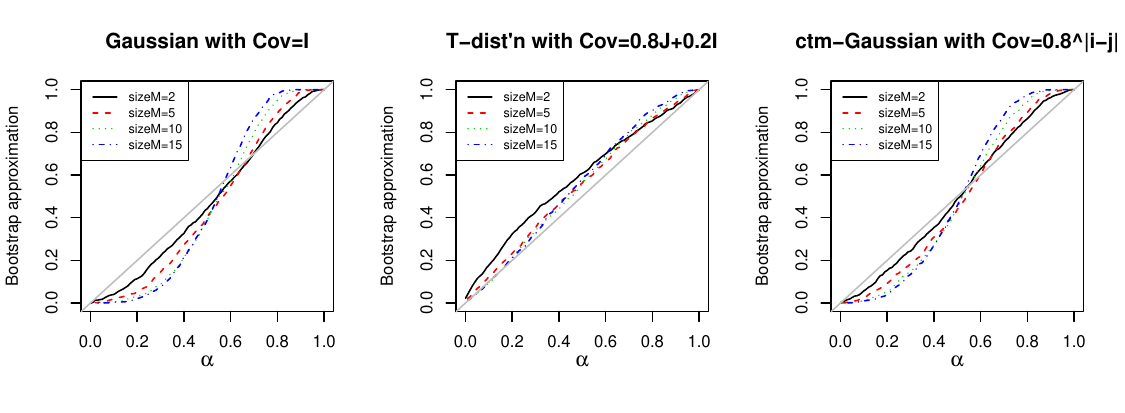} 
		\caption{Empirical rejection rate $\hat{R}(\alpha)$ in selected time-series data-generating schemes under $H_0$: (Left) Gaussian distribution with covariance structure I; (Middle) $t_6$ distribution with covariance structure II; (Right) ctm-Gaussian distribution with covariance structure III. (parameters: $n=500$, $p=600$, $\us=40$ and size $M = 2,5,10,15$) } 
		\label{fig: AlphaEg_TS} 
	\end{figure}

	\begin{figure}
		\centering
		\subfigure[]{
			\includegraphics[trim = 6 10 30 45, clip, width=0.45\textwidth]{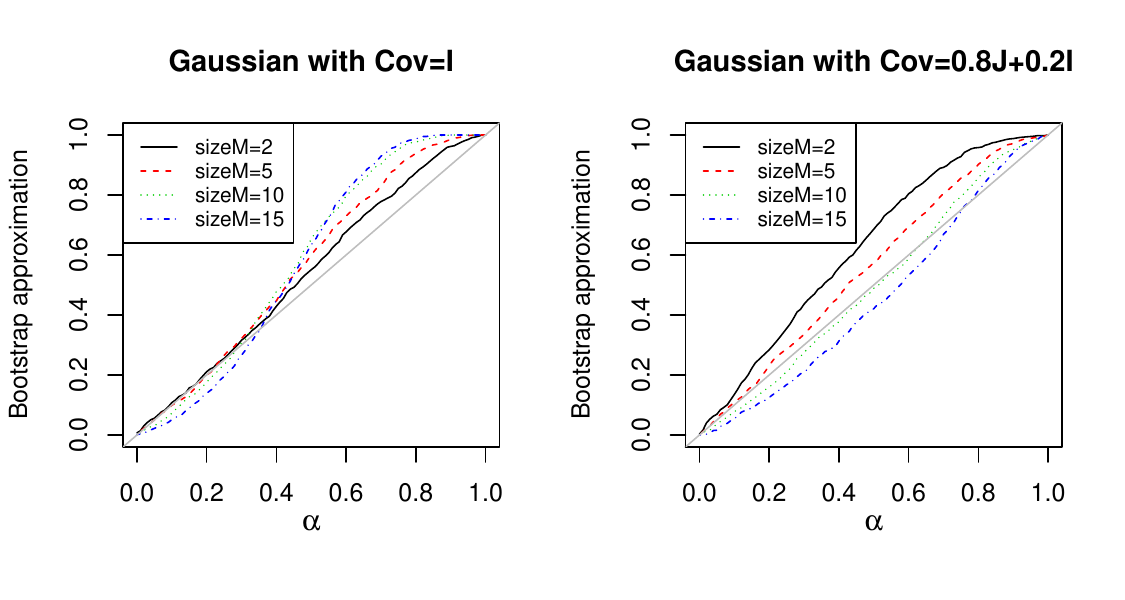}
			\label{fig:JirakTS_subfig:a}}
		\hspace{0.15in}
		\subfigure[]{
			\includegraphics[trim = 6 10 30 45, clip, width=0.45\textwidth]{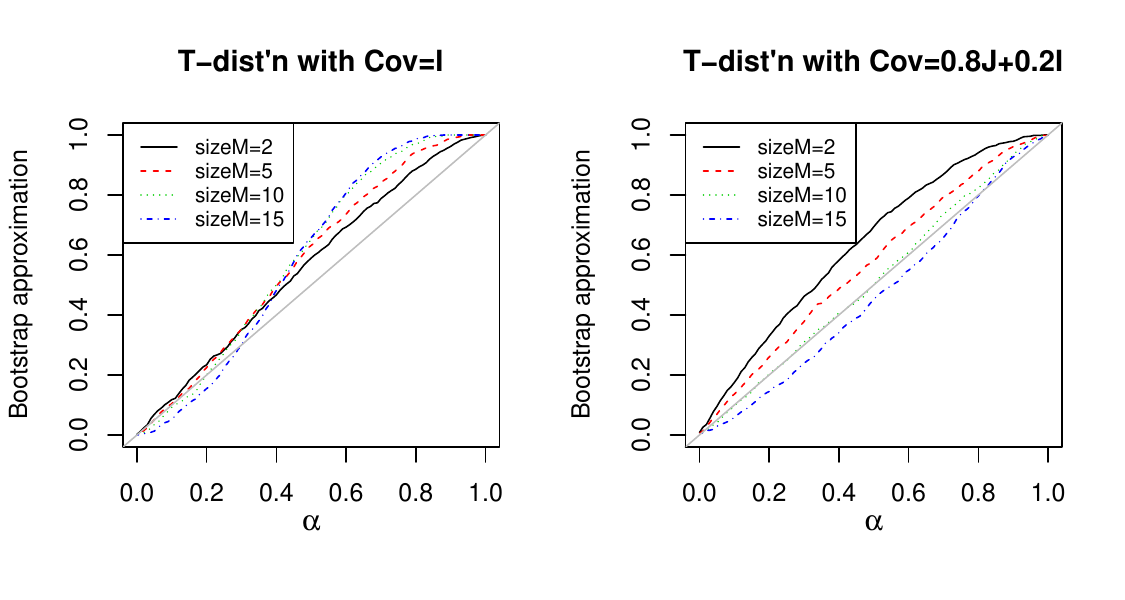}
			\label{fig:JirakTS_subfig:b}}
		\vspace{-0.1in}
		\caption{Empirical rejection rate $\hat{R}(\alpha)$ of the block-wise bootstrap test in \cite{jirak2015} under $H_0$: (a) Gaussian distributed $\eta_i$ with Covariance I (left) and II (right); (b) $T_6$ distributed $\eta_i$ with Covariance I (left) and II (right). (Parameters: $n=500, p=600, \us=40$, $\hat{\sigma}_h^2 = {\sigma}_h^2$ and $\xi_l=1$ is used in conditional long-run variance estimators $\hat{s}_h^2$ in bootstrap.) 
		}
		\label{fig:JirakTS} 
	\end{figure}
	
	We also examine our location estimators, \texttt{Inspect} \citep{wangsamworth2017} and \texttt{SBS} \citep{chofryzlewicz2015} for the temporal dependence case. Figure~\ref{fig:locTS}  provides RMSEs of the three algorithms for data from ctm-Gaussian with covariance III. Similar conclusions can be drawn as in temporal independence case. 
	The RMSEs are reported in Table~\ref{tab:RMSE_our}, \ref{tab:RMSE_Wang} and \ref{tab:RMSE_Cho}.

	\begin{figure}
		\centering
		\subfigure[]{
			\includegraphics[trim = 0 10 10 45, clip, width=0.3\textwidth]{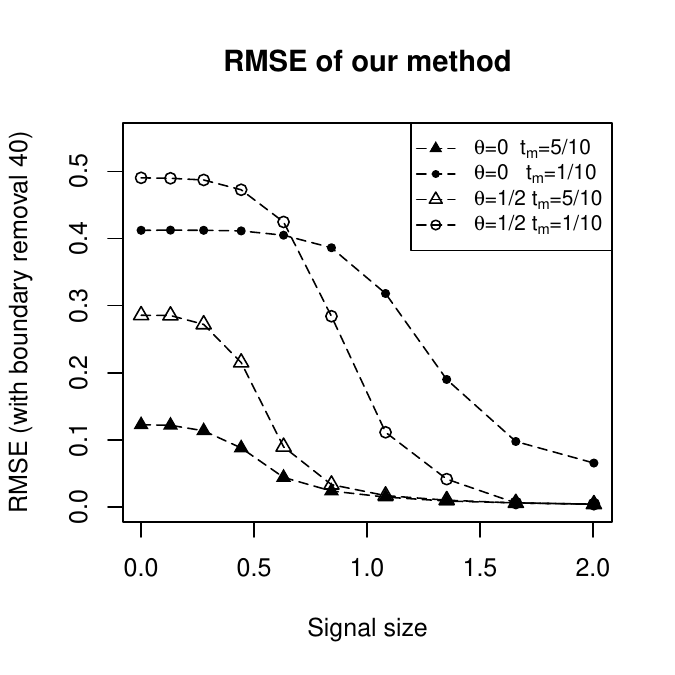}
			\label{fig:locTS_subfig:a}}
		\subfigure[]{
			\includegraphics[trim = 0 10 10 45, clip, width=0.3\textwidth]{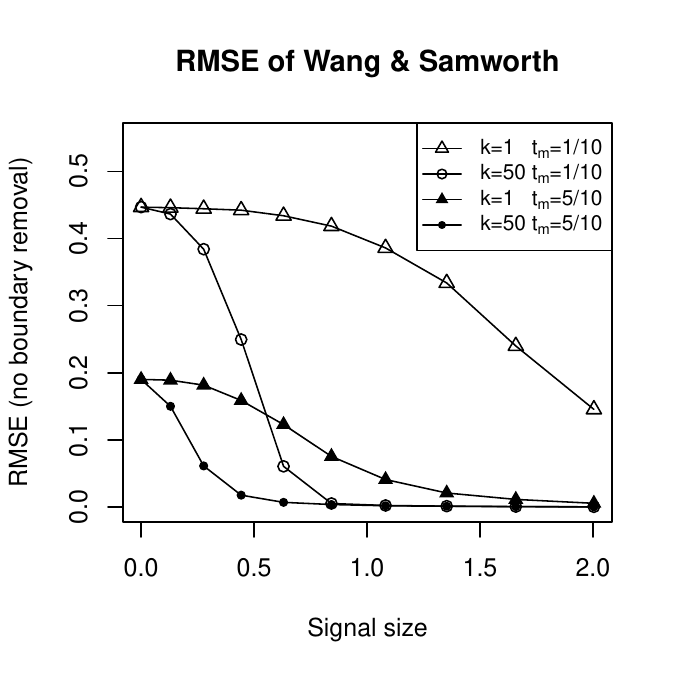}
			\label{fig:locTS_subfig:b}}
		\subfigure[]{
			\includegraphics[trim = 0 10 10 45, clip, width=0.3\textwidth]{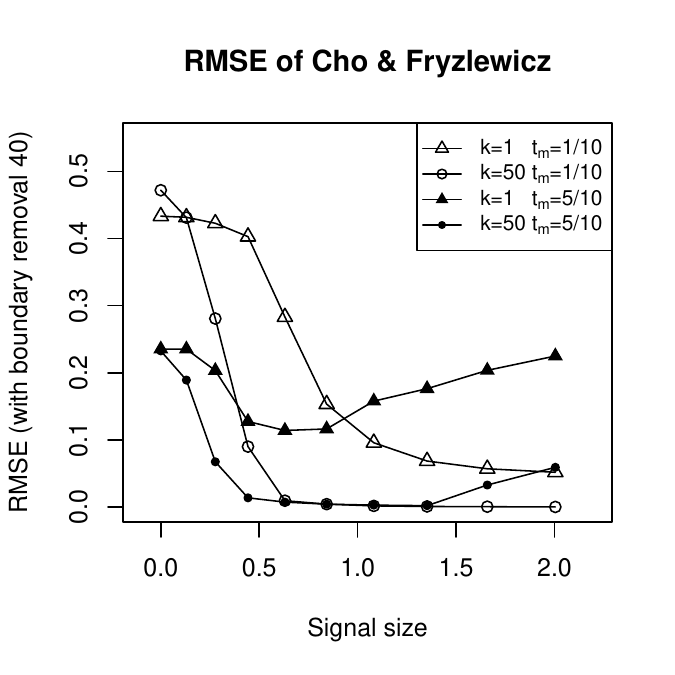}
			\label{fig:locTS_subfig:c}}
		\caption{Comparison of RMSEs: (a) our two estimators $t_{\hat{m}_0}$ and $t_{\hat{m}_{1/2}}$ with $k=1$ and $\us=40$; (b) \texttt{Inspect} with $k=1,50$ and $\us = 1$; (c) \texttt{SBS} with $k=1,50$ and $\us = 40$; . Parameters: $n=500, p=600, t_m = 1/10, 5/10$ and data are from ctm-Gaussian distribution with Covariance II.}
		\label{fig:locTS} 
	\end{figure}

	\subsection{Additional tables and figures}
	\label{subsec:additional_tabs_figs}
	
	\begin{table}
		\caption{\label{tab:power} Power report of our bootstrap CUSUM test for sparse alternative $k=1$. Parameters: $t_m=5/10, 3/10, 1/10$, $n=500, p=600, \us=40, \alpha=0.05$ and data are simulated from all combinations of distribution families and covariance dependence structures.}
		
		\centering
		\fbox{
		\begin{tabular}{l|lll|lll|lll}
			\multicolumn{1}{c|}{Our method}                 & \multicolumn{3}{c|}{Gaussian}               & \multicolumn{3}{c|}{$t_6$}                 & \multicolumn{3}{c}{ctm-Gaussian} \\ \cline{2-10} 
			\multicolumn{1}{c|}{$|\delta|_\infty=\delta_1$} & I     & II    & \multicolumn{1}{l|}{III}    & I     & II    & \multicolumn{1}{l|}{III}   & I         & II        & III       \\ \hline
			\multicolumn{10}{c}{$t_m=5/10$}                                                                                                                                                                                                                                                        \\ \hline
			0                          & 0.031                 & 0.038                  & 0.036 & 0.020                 & 0.044                  & 0.016 & 0.015                 & 0.042                  & 0.027                   \\
			0.13                       & 0.035                 & 0.043                  & 0.034 & 0.020                 & 0.044                  & 0.022 & 0.014                 & 0.047                  & 0.027                   \\
			0.28                       & 0.098                 & 0.295                  & 0.128 & 0.042                 & 0.137                  & 0.038 & 0.026                 & 0.126                  & 0.050                   \\
			0.44                       & 0.662                 & 0.884                  & 0.677 & 0.296                 & 0.559                  & 0.279 & 0.235                 & 0.567                  & 0.280                   \\
			0.63                       & 0.989                 & 1                      & 0.993 & 0.831                 & 0.939                  & 0.851 & 0.797                 & 0.958                  & 0.843                   \\
			0.84                       & 1                     & 1                      & 1     & 0.990                 & 1                      & 0.997 & 0.997                 & 1                      & 0.996                   \\ \hline
			\multicolumn{10}{c}{$t_m=3/10$}                                                                                                                                                \\ \hline
			0                          & 0.042 & 0.056 & 0.034  & 0.024 & 0.044 & 0.023 & 0.020     & 0.049     & 0.018     \\
			0.13                       & 0.043 & 0.060 & 0.034  & 0.022 & 0.047 & 0.027 & 0.021     & 0.047     & 0.018     \\
			0.28                       & 0.087 & 0.209 & 0.082  & 0.03  & 0.109 & 0.034 & 0.033     & 0.093     & 0.029     \\
			0.44                       & 0.502 & 0.756 & 0.513  & 0.181 & 0.477 & 0.192 & 0.186     & 0.431     & 0.187     \\
			0.63                       & 0.966 & 0.996 & 0.972  & 0.652 & 0.919 & 0.711 & 0.675     & 0.890     & 0.690     \\
			0.84                       & 1     & 1     & 1      & 0.981 & 0.997 & 0.979 & 0.977     & 0.999     & 0.987     \\
			1.08                       & 1     & 1     & 1      & 1     & 1     & 0.999 & 0.999     & 1         & 1         \\ \hline
			\multicolumn{10}{c}{$t_m=1/10$}                                                                                                                                                \\ \hline
			0                          & 0.039 & 0.046 & 0.037  & 0.023 & 0.052 & 0.032 & 0.018     & 0.058     & 0.017     \\
			0.13                       & 0.036 & 0.045 & 0.035  & 0.025 & 0.053 & 0.033 & 0.025     & 0.052     & 0.016     \\
			0.28                       & 0.047 & 0.070 & 0.0400 & 0.024 & 0.057 & 0.035 & 0.025     & 0.053     & 0.020     \\
			0.44                       & 0.094 & 0.253 & 0.091  & 0.036 & 0.121 & 0.049 & 0.031     & 0.108     & 0.032     \\
			0.63                       & 0.370 & 0.645 & 0.432  & 0.119 & 0.344 & 0.147 & 0.123     & 0.310     & 0.109     \\
			0.84                       & 0.861 & 0.948 & 0.861  & 0.426 & 0.75  & 0.503 & 0.418     & 0.707     & 0.406     \\
			1.08                       & 0.998 & 1     & 0.997  & 0.870 & 0.967 & 0.860 & 0.826     & 0.953     & 0.847     \\
			1.35                       & 1     & 1     & 1      & 0.991 & 0.998 & 0.989 & 0.988     & 0.998     & 0.992     \\ 
			
		\end{tabular}
	}
	\end{table}
	
	\begin{table}
		\caption{\label{tab:Jirak_power} Power report of $B_n$ in \cite{jirak2015} for sparse alternative $k=1$. Parameters: $t_m=5/10, 3/10, 1/10$, $n=500, p=600, \us=40, \alpha=0.05$ and data are simulated from all combinations of distribution families and covariance structures.}
		\centering
		\fbox{
		\begin{tabular}{clll|lll|lll}
			\multicolumn{1}{c|}{$B_n$}                      & \multicolumn{3}{c|}{Gaussian}                                               & \multicolumn{3}{c|}{$t_6$}                                                  & \multicolumn{3}{c}{ctm-Gaussian}                                        \\ \cline{2-10} 
			\multicolumn{1}{c|}{$|\delta|_\infty=\delta_1$} & \multicolumn{1}{c}{I} & \multicolumn{1}{c}{II} & \multicolumn{1}{c|}{III}   & \multicolumn{1}{c}{I} & \multicolumn{1}{c}{II} & \multicolumn{1}{c|}{III}   & \multicolumn{1}{c}{I} & \multicolumn{1}{c}{II} & \multicolumn{1}{c}{III} \\ \hline
			\multicolumn{10}{c}{$t_m=5/10$}                                                                                                                                                                                                                                                        \\ \hline
			\multicolumn{1}{c|}{0}     & 0.053    & 0.062 & 0.043 & 0.070 & 0.049 & 0.050 & 0.061        & 0.053 & 0.060 \\
			\multicolumn{1}{c|}{0.13}  & 0.063    & 0.077 & 0.055 & 0.068 & 0.049 & 0.051 & 0.058        & 0.058 & 0.071 \\
			\multicolumn{1}{c|}{0.28}  & 0.219    & 0.418 & 0.210 & 0.147 & 0.219 & 0.110 & 0.111        & 0.210 & 0.109 \\
			\multicolumn{1}{c|}{0.44}  & 0.821    & 0.946 & 0.809 & 0.502 & 0.750 & 0.504 & 0.433        & 0.706 & 0.429 \\
			\multicolumn{1}{c|}{0.84}  & 1        & 1     & 1     & 1     & 1     & 1     & 0.998        & 1     & 1     \\ \hline
			\multicolumn{10}{c}{$t_m=3/10$}                                                                                                                                                                                                                                                        \\ \hline
			\multicolumn{1}{c|}{0}     & 0.061    & 0.052 & 0.053 & 0.056 & 0.053 & 0.061 & 0.067        & 0.056 & 0.056 \\
			\multicolumn{1}{c|}{0.13}  & 0.053    & 0.064 & 0.063 & 0.061 & 0.055 & 0.061 & 0.072        & 0.059 & 0.063 \\
			\multicolumn{1}{c|}{0.28}  & 0.110    & 0.232 & 0.110 & 0.081 & 0.132 & 0.073 & 0.076        & 0.121 & 0.080 \\
			\multicolumn{1}{c|}{0.44}  & 0.553    & 0.828 & 0.561 & 0.263 & 0.537 & 0.257 & 0.245        & 0.469 & 0.257 \\
			\multicolumn{1}{c|}{0.84}  & 1        & 1     & 1     & 0.996 & 1     & 0.996 & 0.987        & 1     & 0.988 \\ \hline
			\multicolumn{10}{c}{$t_m=1/10$}                                                                                                                                                                                                                                                        \\ \hline
			\multicolumn{1}{c|}{0}     & 0.054    & 0.073 & 0.063 & 0.066 & 0.060 & 0.059 & 0.056        & 0.057 & 0.063 \\
			\multicolumn{1}{c|}{0.13}  & 0.062    & 0.069 & 0.062 & 0.070 & 0.061 & 0.055 & 0.049        & 0.058 & 0.069 \\
			\multicolumn{1}{c|}{0.44}  & 0.058    & 0.084 & 0.061 & 0.073 & 0.068 & 0.065 & 0.050        & 0.054 & 0.071 \\
			\multicolumn{1}{c|}{0.84}  & 0.189    & 0.520 & 0.243 & 0.111 & 0.246 & 0.101 & 0.079        & 0.229 & 0.098 \\
			\multicolumn{1}{c|}{1.08}  & 0.684    & 0.925 & 0.730 & 0.264 & 0.615 & 0.302 & 0.210        & 0.556 & 0.252 \\
			\multicolumn{1}{c|}{1.35}  & 0.986    & 1     & 0.985 & 0.703 & 0.940 & 0.724 & 0.620        & 0.903 & 0.670 \\
			\multicolumn{1}{c|}{1.66}  & 1        & 1     & 1     & 0.99  & 1     & 0.986 & 0.962        & 0.997 & 0.968 \\ 
		\end{tabular}
	}
	\end{table}
	
	\begin{table}
		\caption{\label{tab:Psi_power} Power report of $\psi$ in \cite{enikeevaharchaoui2014} for both sparse $k=1$ and dense $k=50$  alternatives. Parameters: $\alpha=0.05$, $t_m=5/10, 3/10, 1/10$, $n=500, p=600$ and data are simulated from Gaussian distribution with Covariance~I.}
		\fbox{
		\begin{tabular}{l|lll|lll}
			\multicolumn{1}{c|}{$\psi$}            & \multicolumn{3}{c|}{$\us=1$}                              & \multicolumn{3}{c}{$\us=40$}        \\ \cline{2-7} 
			\multicolumn{1}{c|}{$|\delta|_\infty$} & $t_m=5/10$ & $t_m=3/10$ & \multicolumn{1}{l|}{$t_m=1/10$} & $t_m=5/10$ & $t_m=3/10$ & $t_m=1/10$ \\ \hline
			\multicolumn{7}{c}{Sparse $H_1$: Gaussian (I) }                                                                                                 \\ \hline
			0                 & 0.107      & 0.101      & 0.116      & 0.074      & 0.059      & 0.701      \\
			0.13              & 0.107      & 0.105      & 0.119      & 0.075      & 0.062      & 0.705      \\
			0.28              & 0.126      & 0.115      & 0.125      & 0.096      & 0.075      & 0.710      \\
			0.44              & 0.177      & 0.154      & 0.136      & 0.130      & 0.106      & 0.720      \\
			0.63            & 0.312      & 0.248      & 0.157      & 0.265      & 0.215      & 0.735      \\
			0.84             & 0.625      & 0.483      & 0.195      & 0.604      & 0.478      & 0.764      \\
			1.08              & 0.943      & 0.839      & 0.307      & 0.941      & 0.850      & 0.800      \\
			1.35             & 1          & 0.997      & 0.534      & 1          & 0.997      & 0.843      \\
			1.66              & 1          & 1          & 0.846      & 1          & 1          & 0.923      \\
			2.00                & 1          & 1          & 0.992      & 1          & 1          & 0.990      \\ \hline
			\multicolumn{7}{c}{Dense $H_1$: Gaussian (I) }                                                                                                  \\ \hline
			0                 & 0.110      & 0.116      & 0.120      & 0.055      & 0.067      & 0.074      \\
			0.13              & 0.755      & 0.608      & 0.233   & 0.735      & 0.573      & 0.169      \\
			0.28              & 1          & 1          & 0.977     & 1          & 1          & 0.973      \\ 
			
		\end{tabular}
	}
	\end{table}

	\begin{table}
		\caption{\label{tab:RMSE_our} RMSEs of our estimators $\hat{m}_\theta$ for $\theta = 0, 1/2$. Parameters: $t_m = 5/10, 1/10, \us = 1,40, n=500, p=600$, and data are temporally independent/dependent from selected distributions and covariance structures.}
		\fbox{
		\begin{tabular}{l|cc|cc|cc||cc}
			$ $    & \multicolumn{6}{c||}{Temporally Independent} & \multicolumn{2}{c}{Dependent} \\ 
			$ $    & \multicolumn{2}{c}{Gaussian (I)} & \multicolumn{2}{c}{$t_6$ (II)} & \multicolumn{2}{c||}{ctm-Gaussian (III)} & \multicolumn{2}{c}{TS: ctm-Gaussian (III)} \\ \cline{2-9}
			$|\delta|_\infty=\delta_1$ & $\theta=0$     & $\theta=1/2$     & $\theta=0$    & $\theta=1/2$    & $\theta=0$        & $\theta=1/2$        & $\theta=0$          & $\theta=1/2$          \\ \hline
			\multicolumn{9}{c}{$\us=1,t_m=5/10$}                                                                                                                                 \\ \hline
			0          & 0.116          & 0.382            & 0.167         & 0.407           & 0.124             & 0.431               & 0.124               & 0.429                 \\
			0.13       & 0.115          & 0.380            & 0.164         & 0.405           & 0.124             & 0.430               & 0.124               & 0.429                 \\
			0.28       & 0.101          & 0.346            & 0.118         & 0.354           & 0.115             & 0.422               & 0.116               & 0.423                 \\
			0.44       & 0.054          & 0.175            & 0.067         & 0.229           & 0.082             & 0.350               & 0.088               & 0.362                 \\
			0.63       & 0.026          & 0.033            & 0.037         & 0.099           & 0.039             & 0.178               & 0.044               & 0.207                 \\
			0.84       & 0.015          & 0.015            & 0.020         & 0.035           & 0.020             & 0.054               & 0.024               & 0.060                 \\
			1.35       & 0.005          & 0.005            & 0.008         & 0.009           & 0.008             & 0.010               & 0.011               & 0.011                 \\
			2.00       & 0.003          & 0.002            & 0.004         & 0.004           & 0.004             & 0.004               & 0.004               & 0.004                 \\ \hline 
			\multicolumn{9}{c}{$\us=1,t_m=1/10$}                                                                                                                                   \\ \hline
			0          & 0.416          & 0.545            & 0.436         & 0.552           & 0.423             & 0.575               & 0.416               & 0.573                 \\
			0.13       & 0.416          & 0.544            & 0.435         & 0.550           & 0.423             & 0.575               & 0.416               & 0.572                 \\
			0.28       & 0.416          & 0.539            & 0.433         & 0.539           & 0.422             & 0.573               & 0.416               & 0.571                 \\
			0.44       & 0.413          & 0.498            & 0.414         & 0.488           & 0.421             & 0.567               & 0.414               & 0.562                 \\
			0.63       & 0.394          & 0.353            & 0.349         & 0.350           & 0.415             & 0.525               & 0.410               & 0.525                 \\
			0.84       & 0.318          & 0.133            & 0.262         & 0.187           & 0.387             & 0.417               & 0.388               & 0.433                 \\
			1.35       & 0.081          & 0.006            & 0.116         & 0.032           & 0.178             & 0.065               & 0.193               & 0.112                 \\
			2.00       & 0.039          & 0.002            & 0.058         & 0.005           & 0.060             & 0.004               & 0.062               & 0.005                 \\ \hline \hline
			\multicolumn{9}{c}{$\us=40,t_m=5/10$}                                                                                                                                 \\ \hline
			0          & 0.119 & 0.280 & 0.168 & 0.296 & 0.124 & 0.290 & 0.123 & 0.286 \\
			0.13       & 0.118 & 0.279 & 0.162 & 0.292 & 0.124 & 0.290 & 0.122 & 0.285 \\
			0.28       & 0.103 & 0.245 & 0.116 & 0.237 & 0.115 & 0.276 & 0.114 & 0.272 \\
			0.44       & 0.055 & 0.126 & 0.068 & 0.132 & 0.077 & 0.204 & 0.088 & 0.215 \\
			0.63       & 0.022 & 0.033 & 0.037 & 0.066 & 0.035 & 0.082 & 0.044 & 0.090 \\
			0.84       & 0.013 & 0.017 & 0.020 & 0.026 & 0.019 & 0.028 & 0.024 & 0.034 \\
			1.35       & 0.005 & 0.006 & 0.008 & 0.009 & 0.008 & 0.009 & 0.009 & 0.010 \\
			2.00       & 0.002 & 0.002 & 0.003 & 0.004 & 0.004 & 0.004 & 0.005 & 0.004 \\ \hline
			\multicolumn{9}{c}{$\us=40,t_m=1/10$}                                                                                                                                   \\ \hline
			0          & 0.422 & 0.484 & 0.434 & 0.503 & 0.426 & 0.505 & 0.412 & 0.490 \\
			0.13       & 0.422 & 0.483 & 0.433 & 0.502 & 0.425 & 0.505 & 0.412 & 0.490 \\
			0.28       & 0.422 & 0.476 & 0.430 & 0.489 & 0.425 & 0.503 & 0.412 & 0.487 \\
			0.44       & 0.418 & 0.431 & 0.409 & 0.400 & 0.424 & 0.483 & 0.411 & 0.473 \\
			0.63       & 0.396 & 0.306 & 0.355 & 0.257 & 0.415 & 0.407 & 0.405 & 0.425 \\
			0.84       & 0.326 & 0.125 & 0.268 & 0.115 & 0.383 & 0.270 & 0.386 & 0.284 \\
			1.35       & 0.072 & 0.005 & 0.116 & 0.010 & 0.175 & 0.030 & 0.190 & 0.042 \\
			2.00       & 0.039 & 0.003 & 0.057 & 0.004 & 0.062 & 0.005 & 0.066 & 0.005 \\ 
		\end{tabular}
	}
	\end{table}

	\begin{table}
		\caption{\label{tab:RMSE_Wang} RMSEs of \texttt{Inspect} \citep{wangsamworth2017}. Parameters: $t_m = 5/10, 3/10, 1/10, k = 1,50, n=500, p=600$, and data are temporally independent/dependent from selected distributions and covariance structures.}
		\fbox{
		\begin{tabular}{l|ll|ll|ll||cc}
			$ $    & \multicolumn{6}{c||}{Temporally Independent} & \multicolumn{2}{c}{Dependent} \\ 
			$ $    & \multicolumn{2}{c}{Gaussian (I)} & \multicolumn{2}{c}{$t_6$ (II)} & \multicolumn{2}{c||}{ctm-Gaussian (III)} & \multicolumn{2}{c}{TS: ctm-Gaussian (III)} \\ \cline{2-9}
			$|\delta|_\infty$ & $k=1$     & $k=50$     & $k=1$    & $k=50$    & $k=1$        & $k=50$        & $k=1$          & $k=50$          \\ \hline
			\multicolumn{9}{c}{$t_m=5/10$}                                                                                                                                 \\ \hline
			0          & 0.166 & 0.166 & 0.364 & 0.364 & 0.206 & 0.206 & 0.190 & 0.190 \\
			0.13       & 0.164 & 0.079 & 0.363 & 0.340 & 0.206 & 0.167 & 0.189 & 0.150 \\
			0.28       & 0.139 & 0.005 & 0.354 & 0.251 & 0.195 & 0.061 & 0.182 & 0.061 \\
			0.44       & 0.080 & 0.001 & 0.315 & 0.132 & 0.167 & 0.018 & 0.159 & 0.018 \\
			0.63       & 0.034 & 0.000 & 0.262 & 0.061 & 0.120 & 0.007 & 0.123 & 0.007 \\
			0.84       & 0.018 & 0.000 & 0.211 & 0.027 & 0.073 & 0.004 & 0.076 & 0.004 \\
			1.35       & 0.006 & 0.000 & 0.132 & 0.008 & 0.015 & 0.001 & 0.021 & 0.001 \\
			2.00       & 0.003 & 0.000 & 0.075 & 0.003 & 0.005 & 0.001 & 0.006 & 0.001 \\ 
			\hline
			\multicolumn{9}{c}{$t_m=3/10$}                     \\ \hline
			0    & 0.262 & 0.262 & 0.419 & 0.419 & 0.287 & 0.287 & 0.277 & 0.277 \\
			0.13 & 0.259 & 0.143 & 0.418 & 0.391 & 0.286 & 0.246 & 0.277 & 0.241 \\
			0.28 & 0.233 & 0.007 & 0.408 & 0.294 & 0.279 & 0.104 & 0.272 & 0.108 \\
			0.44 & 0.146 & 0.001 & 0.371 & 0.163 & 0.249 & 0.019 & 0.257 & 0.022 \\
			0.63 & 0.054 & 0     & 0.311 & 0.075 & 0.202 & 0.007 & 0.214 & 0.007 \\
			0.84 & 0.019 & 0     & 0.254 & 0.042 & 0.121 & 0.004 & 0.145 & 0.004 \\
			1.35 & 0.006 & 0     & 0.171 & 0.007 & 0.029 & 0.001 & 0.037 & 0.001 \\
			2.00 & 0.002 & 0     & 0.099 & 0.003 & 0.006 & 0.001 & 0.007 & 0.001 \\
			\hline
			\multicolumn{9}{c}{$t_m=1/10$}                                                                                                                                 \\ \hline
			0          & 0.425 & 0.425 & 0.555 & 0.555 & 0.446 & 0.446 & 0.447 & 0.447 \\
			0.13       & 0.424 & 0.394 & 0.555 & 0.544 & 0.445 & 0.434 & 0.446 & 0.436 \\
			0.28       & 0.420 & 0.211 & 0.555 & 0.504 & 0.443 & 0.374 & 0.444 & 0.384 \\
			0.44       & 0.405 & 0.004 & 0.548 & 0.424 & 0.439 & 0.224 & 0.442 & 0.250 \\
			0.63       & 0.351 & 0.001 & 0.527 & 0.309 & 0.428 & 0.055 & 0.434 & 0.061 \\
			0.84       & 0.255 & 0.000 & 0.487 & 0.197 & 0.407 & 0.006 & 0.418 & 0.006 \\
			1.35       & 0.031 & 0.000 & 0.390 & 0.047 & 0.289 & 0.002 & 0.334 & 0.002 \\
			2.00       & 0.003 & 0.000 & 0.293 & 0.006 & 0.116 & 0.001 & 0.146 & 0.001 \\ 
		\end{tabular}
	}
	\end{table}

	\begin{table}
		\caption{\label{tab:RMSE_Cho} RMSEs of \texttt{SBS} \citep{chofryzlewicz2015}. Parameters: $t_m = 5/10, 1/10, k = 1,50, n=500, p=600$, and data are temporally independent/dependent from selected distributions and covariance structures.}
		\fbox{
		\begin{tabular}{l|cc|cc|cc||cc}
			$ $    & \multicolumn{6}{c||}{Temporally Independent} & \multicolumn{2}{c}{Dependent} \\ 
			$ $    & \multicolumn{2}{c}{Gaussian (I)} & \multicolumn{2}{c}{$t_6$ (II)} & \multicolumn{2}{c||}{ctm-Gaussian (III)} & \multicolumn{2}{c}{TS: ctm-Gaussian (III)} \\ \cline{2-9}
			$|\delta|_\infty$ & $k=1$     & $k=50$     & $k=1$    & $k=50$    & $k=1$        & $k=50$        & $k=1$          & $k=50$          \\ \hline
			\multicolumn{9}{c}{$t_m=5/10$}                                                                                                                                 \\ \hline
			0                 & 0.207                & 0.221  & 0.246      & 0.240  & 0.249              & 0.243  & 0.243                  & 0.229  \\
			0.13              & 0.201                & 0.113  & 0.246      & 0.205  & 0.244              & 0.210  & 0.237                  & 0.208  \\
			0.28              & 0.187                & 0.011  & 0.172      & 0.134  & 0.224              & 0.086  & 0.234                  & 0.098  \\
			0.44              & 0.131                & 0.004  & 0.107      & 0.072  & 0.174              & 0.017  & 0.185                  & 0.024  \\
			0.63              & 0.116                & 0.003  & 0.072      & 0.036  & 0.137              & 0.008  & 0.137                  & 0.009  \\
			0.84              & 0.119                & 0.002  & 0.061      & 0.028  & 0.130              & 0.006  & 0.123                  & 0.006  \\
			1.35              & 0.126                & 0.001  & 0.066      & 0.023  & 0.153              & 0.003  & 0.135                  & 0.003  \\
			2                 & 0.140                & 0.001  & 0.069      & 0.032  & 0.152              & 0.001  & 0.161                  & 0.001  \\
			\hline
			\multicolumn{9}{c}{$t_m=1/10$}                                                                                                                                 \\ \hline
			0                 & 0.379                & 0.381  & 0.433      & 0.431  & 0.440              & 0.430  & 0.430                  & 0.419  \\
			0.13              & 0.392                & 0.316  & 0.444      & 0.400  & 0.436              & 0.374  & 0.409                  & 0.416  \\
			0.28              & 0.365                & 0.099  & 0.426      & 0.321  & 0.415              & 0.326  & 0.414                  & 0.314  \\
			0.44              & 0.355                & 0.005  & 0.360      & 0.194  & 0.405              & 0.126  & 0.394                  & 0.159  \\
			0.63              & 0.267                & 0.002  & 0.202      & 0.106  & 0.348              & 0.026  & 0.363                  & 0.017  \\
			0.84              & 0.164                & 0.001  & 0.092      & 0.043  & 0.272              & 0.005  & 0.278                  & 0.005  \\
			1.35              & 0.056                & 0.000  & 0.034      & 0.028  & 0.105              & 0.002  & 0.105                  & 0.002  \\
			2                 & 0.040                & 0.000  & 0.036      & 0.023  & 0.078              & 0.001  & 0.061                  & 0.001 \\
		\end{tabular}
	}
	\end{table}

	\begin{figure}
		\centering
		\includegraphics[trim=190 0 0 35,clip,width = 0.75\textwidth]{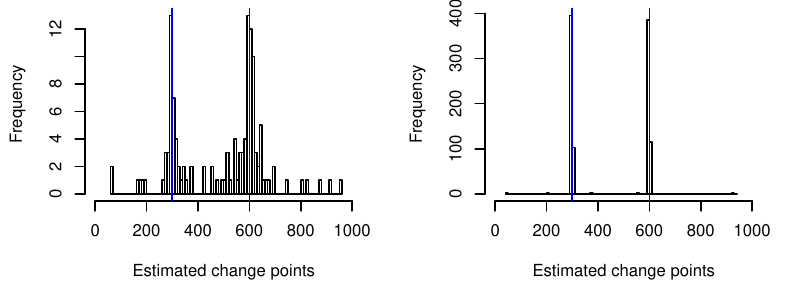}
		\caption{Histograms of estimated $m_i$'s over 500 simulation under multiple change point setup $(m_1, m_2) = (300, 600)$ using \texttt{BABS}: (Left) $\delta=0.317$; (Right) $\delta= 2.004$. Parameters: $n = 1000,p = 1200, \us = 40, \alpha = 0.05$ and data are from $t_6$ distribution with Covariance~III.}
		\label{fig:mcp_identification_2}
	\end{figure}

	\begin{figure}
		\centering
		\includegraphics[trim=190 0 0 35,clip,width = 0.75\textwidth]{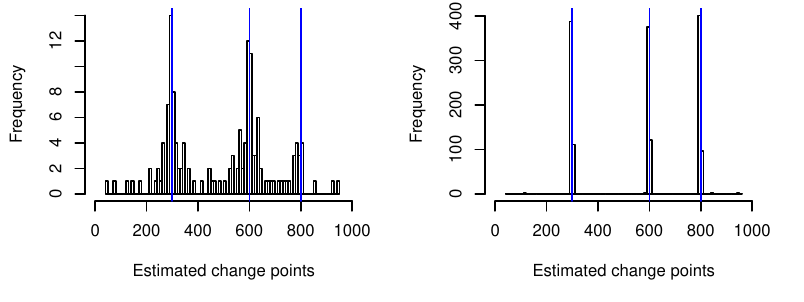}
		\caption{Histograms of estimated $m_i$'s over 500 simulation under multiple change point setup $(m_1, m_2, m_3) = (300, 600, 800)$ using \texttt{BABS}: (Left) $\delta=0.317$; (Right) $\delta= 2.004$. Parameters: $n = 1000,p = 1200, \us = 40, \alpha = 0.05$ and data are from $t_6$ distribution with Covariance~III.}
		\label{fig:mcp_identification_3}
	\end{figure}
	
	\begin{figure}
		\centering
		\subfigure[]{
			\includegraphics[trim=190 0 190 35,clip, width=0.35\textwidth]{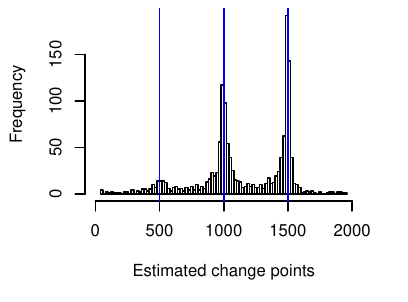}
			\includegraphics[trim=190 0 190 35,clip, width=0.35\textwidth]{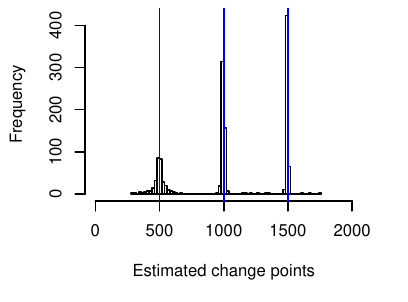}
			\label{fig:loc_mcp_subfig:a}}
		\hspace{0.2in}
		\subfigure[]{
			\includegraphics[trim=190 0 190 35,clip, width=0.35\textwidth]{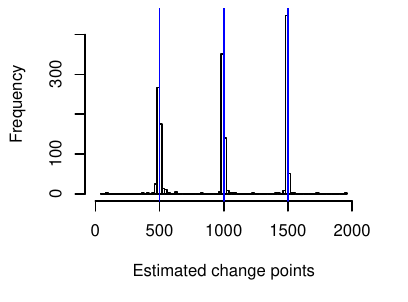}
			\includegraphics[trim=190 0 190 35,clip, width=0.35\textwidth]{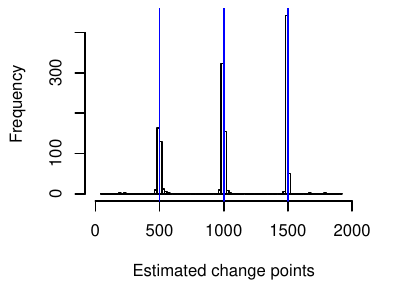}
			\label{fig:loc_mcp_subfig:b}}
		\caption{Histograms of estimated $m_i$'s over 500 simulation under complete-overlap multiple change point setup $(m_1, m_2, m_3) = (500,1000,1500), (|\delta_{n}^{(1)}|_2 , |\delta_{n}^{(2)}|_2 , |\delta_{n}^{(3)}|_2) = (0.6,1.2,1.8)$ using \texttt{BABS} algorithm with $\alpha = 0.05, B=200, \us = 40$ (left) and \texttt{Inspect} with default setting in R package \texttt{InspectChangepoint}: (a) dense signal of $k=40$; (b) sparse signal of $k=1$. Parameters: $n=2000, p=200$.}
		\label{fig:compare_binseg_inspect} 
		\vspace{-0.1in}
	\end{figure}

	\begin{figure}
		\centering 
		\includegraphics[trim=0 0 0 40,clip, width= 2.1in]{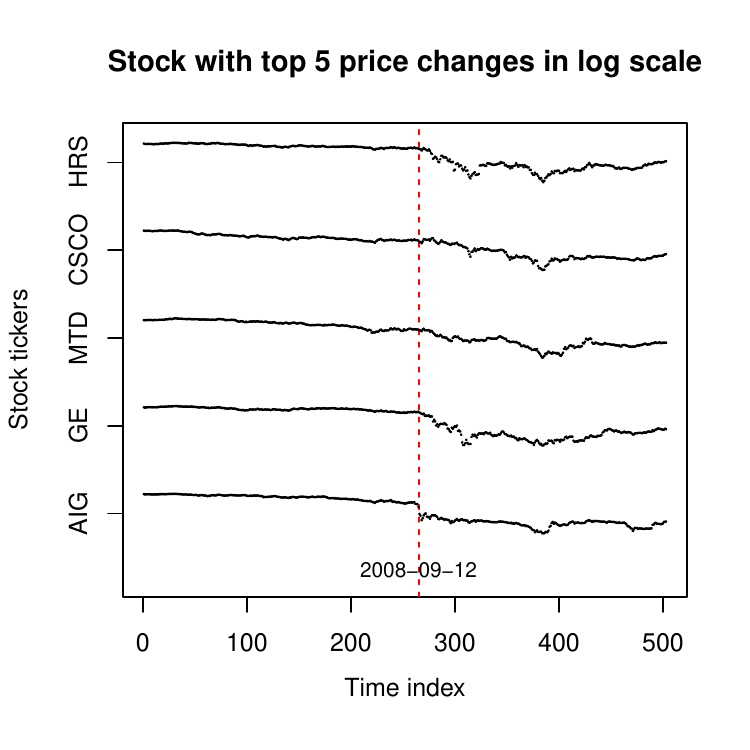} 
		\caption{Top 5 stocks with financial crash in the stock return data on log-scale.}
		\label{fig: Realdata_finance} 
	\end{figure}
	
%

\end{document}